\documentclass[11pt,reqno]{amsart}
\usepackage{amsmath,amsfonts,amssymb,amscd,amsthm,amsbsy,bbm, epsf,calc,  pdfsync, comment, caption}
\usepackage{color}
\usepackage[usenames,dvipsnames]{xcolor}
\usepackage{datetime}

\usepackage{thmtools, thm-restate}
\usepackage{thm-patch, thm-kv, thm-autoref}
\usepackage[colorlinks=true, urlcolor=blue, linkcolor=blue, citecolor=black]{hyperref}

\usepackage[pdftex]{graphicx}
\usepackage{wrapfig}
\usepackage{geometry}
\usepackage{tikz-cd}

\numberwithin{equation}{section}
\newcounter{hours}\newcounter{minutes}

\theoremstyle{plain}
\declaretheorem[title=Theorem, parent=section]{theorem}
\declaretheorem[title=Lemma,sibling=theorem]{lemma}
\declaretheorem[title=Corollary,sibling=theorem]{corollary}
\declaretheorem[title=Assumption,sibling=theorem]{assumption}
\declaretheorem[title=Proposition,sibling=theorem]{proposition}
\declaretheorem[title=Definition,sibling=theorem]{definition}
\declaretheorem[title=Remark,sibling=theorem]{remark}
\declaretheorem[title=Example,sibling=theorem]{example}
\declaretheorem[title=Theorem]{theorem*}

\def\real{\mathbb R}
\def\al{\alpha}
\def\gam{\gamma}
\def\L{\mathcal{L}}
\def\M{\mathcal{M}}
\def\om{\omega}

\def\expct{\mathbb E}
\def\ind{{\mathbbm{1}}}
\def\intersect{\cap}

\def\ind{{\mathbbm{1}}}

\def\dive{\textnormal{div}}
\def\diam{\textnormal{diam}}
\def\grad{\nabla}
\def\hull{\textnormal{c.h.}}
\def\det{\textnormal{det}}
\def\MA{\textnormal{MA}}
\def\tr{\textnormal{tr}}

\def\spt{\textnormal{spt}}

\newcommand{\abs}[1]{\left| #1 \right|}
\newcommand{\norm}[1]{\lVert#1\rVert}

\def\Swiech{{\accent"13S}wie{\hbox{\kern -0.21em\lower 0.79ex\hbox{$\textfont1=\scriptfont1\lhook$}}}ch}
\begin{document}

\title{Min-max formulas for nonlocal elliptic operators on Euclidean Space}

\author{Nestor Guillen}
\author{Russell W. Schwab}

\address{Department of Mathematics\\
University of Massachusetts, Amherst\\
Amherst, MA  01003-9305}
\email{nguillen@math.umass.edu}

\address{Department of Mathematics\\
Michigan State University\\
619 Red Cedar Road \\
East Lansing, MI 48824}
\email{rschwab@math.msu.edu}

\begin{abstract}
  An operator satisfies the Global Comparison Property if anytime a function touches another from above at some point, then the operator preserves the ordering at the point of contact. This is characteristic of degenerate elliptic operators, including nonlocal and nonlinear ones. In previous work, the authors considered such operators in Riemannian manifolds and proved they can be represented by a min-max formula in terms of L\'evy operators. In this note we revisit this theory in the context of Euclidean space. With the intricacies of the general Riemannian setting gone, the ideas behind the original proof of the min-max representation become clearer. Moreover, we prove new results regarding operators that commute with translations or which otherwise enjoy some spatial regularity. 
\end{abstract}

\date{Monday 18th February, 2019\ (This is the revised version per referee suggestions.)}
\thanks{The authors gratefully acknowledge partial support from the National Science Foundation while this work was in progress: N. Guillen DMS-1700307 and R. Schwab DMS-1665285.  The authors also thank the anonymous referee for some suggestions that we believe improved the presentation of our results.}
\subjclass[2010]{
%updated and confirmed 3/24/2016
35J99,      %pde other
35R09,  	%Integro-partial differential equations
45K05,  	%Integro-partial differential equations
46T99,   	%None of the above, but in this section
47G20,      %integro-differential operators
49L25,  	%Optimal Control Viscosity solutions
49N70,  	%Differential games
60J75,      %jump processes
93E20       %optimal stoch. control
}

\maketitle

\markboth{N. Guillen, R. Schwab}{Min-max formulas for nonlocal elliptic operators on Euclidean Space}

%%%%%%%%%%%%%%%%%%%%%%%%%%%%%%%%%
%%%%%%%%%%%%%%%%%%%%%%%%%%%%%%%%%
%%%%%%%%%%%%%%%%%%%%%%%%%%%%%%%%%
%%%%%%%%%%%%%%%%%%%%%%%%%%%%%%%%%
%%%%%%%%%%%%%%%%%%%%%%%%%%%%%%%%%
%%%%%%%%%%%%%%%%%%%%%%%%%%%%%%%%%
\section{Introduction}\label{sec:Introduction}
\setcounter{equation}{0}

A map $I:C^2_b(\mathbb{R}^d)\to C^0_b(\mathbb{R}^d)$ is said to satisfy the \textbf{Global Comparison Property} (GCP) if
\begin{align}\label{equation:GCP}
  u\leq v \textnormal{ in } \mathbb{R}^d \textnormal{ and } u(x)=v(x) \Rightarrow I(u,x)\leq I(v,x).
\end{align}
The Laplacian operator, as well as its fractional powers $-(-\Delta)^{\alpha/2}$ ($\alpha \in (0,2)$) all satisfy this property. More generally, given a L\'evy measure $\nu(dy)$ (a measure on $\mathbb{R}^d\setminus \{0\}$ such that $\min\{1,|y|^2\}$ is integrable with respect to $\nu$) the operator
\begin{align*}
  I(u,x) = \int_{\mathbb{R}^d} u(x+y)-u(x)-\chi_{B_1}(y) \nabla u(x)\cdot y\;\nu(dy),
\end{align*}
will have the GCP. The GCP is also satisfied by Dirichlet-to-Neumann maps for elliptic equations, generators of Markov processes, Bellman-Isaacs operators in control and differential games, among many examples.  When the operator is known a priori to be local, then nonlinear examples of maps with the GCP are of the form,
\begin{align*}
  I(u,x) = F(D^2u(x),\nabla u(x),u(x)),
\end{align*}
where $F:\mathbb{S}_d \times \mathbb{R}^d\times \mathbb{R}\to\mathbb{R}$ is monotone in its first argument, and Lipschitz continuous in all arguments.

The main contribution of this article is to address when certain operators acting on $C^2_b(\real^d)$ must necessarily enjoy a structure similar to those examples above.  The canonical object used to address this question will be a linear operator we choose to say is ``of L\'evy type'': those operators for which there exist functions, $A(x)\in\mathbb{S}_d$, $B(x)\in\real^d$, $C(x)\in\real$, and measures $\mu(x,dy)$ so that
\begin{align}\label{eqIN:LevyTypeLinear}
	L(u,x) & = \tr(A(x)D^2u(x))+B(x)\cdot \nabla u(x)+C(x)u(x)\\
	&+ \int_{\mathbb{R}^d} u(x+y)-u(x)-\ind_{B_1(0)}(y)\nabla u(x)\cdot y\;\mu(x,dy),\nonumber\\
	&\text{with}\ A(x)\geq 0,\ \text{and}\ \sup_x\int_{\real^d}\min(\abs{y}^2,1)\mu(x,dy)<\infty.
	\nonumber 
\end{align}
We will review some recent results that show for $I:C^2_b(\real^d)\to C_b(\real^d)$ that enjoys the GCP, is Lipschitz, and has a natural structural constraint, there exists a family of functions, $f_{ab}$ and linear operators of L\'evy type, $L_{ab}$, so that 
\begin{align}\label{eqIN:MinMaxMeta}
	I(u,x) = \min\limits_{a}\max\limits_{b} \{ f_{ab}(x)+L_{ab}(u,x) \}.
\end{align}
For linear operators, in the 1960's Courr\`ege \cite{Courrege-1965formePrincipeMaximum} showed that all of those that satisfy the GCP must have the form given in (\ref{eqIN:LevyTypeLinear}).  All of our results here should be considered an extension of Courr\`ege's result to the nonlinear setting.

In our previous work, \cite{GuSc-2016MinMaxNonlocalarXiv}, we showed such a min-max representation in (\ref{eqIN:MinMaxMeta}). The result in \cite{GuSc-2016MinMaxNonlocalarXiv} in fact dealt with a more general situation where $I:C^2_b(M)\to C^0_b(M)$ where $M$ is a complete Riemannian manifold. We will review the proof of this result in the context of Euclidean space, where many of the arguments simplify greatly. Moreover,  we prove two refinements of the main result from \cite{GuSc-2016MinMaxNonlocalarXiv} relevant to the Euclidean case, one involving translation invariant operators and one for operators that behave continuously with respect to translation operators.  Stated informally, our results are the following:

\begin{theorem*}\label{theoremA}
  An operator $I(u,x)$ that is Lipschitz and satisfies the GCP  admits a min-max formula in terms of L\'evy type operators. 
\end{theorem*}

\begin{theorem*}\label{theoremB}
  In the previous theorem, assume further that $I(u,x)$ commutes with translations. Then the L\'evy operators appearing in the min-max formula all commute with translations.
\end{theorem*}

\begin{theorem*}\label{theoremC}
  Instead of translation invariance assume that the finite differences of $I(u,x)$ commute with translations up to a certain error depending on a modulus of continuity $\omega(\cdot)$. Then the L\'evy operators appearing in the min-max formula have continuous coefficients with common modulus of continuity of the form $C\omega(2(\cdot))$. 
\end{theorem*}

Theorem \ref{theoremA} above is a special case of the main result in \cite{GuSc-2016MinMaxNonlocalarXiv}, and Theorems \ref{theoremB} and \ref{theoremC} are new.

%%%%%%%%%%%%%%%%%%%%%%%%%%%%%%%%%
%%%%%%%%%%%%%%%%%%%%%%%%%%%%%%%%%
%%%%%%%%%%%%%%%%%%%%%%%%%%%%%%%%%
\subsection{Assumptions and main results}\label{subsection:assumptions and main results}

Here are our main assumptions. 
 
\begin{assumption}\label{assumption:GCP} 
  The map $I:C^2_b(\mathbb{R}^d)\to C^0_b(\mathbb{R}^d)$ is Lipschitz continuous and has the Global Comparison Property \eqref{equation:GCP}.
\end{assumption}

\begin{assumption}\label{assumption:translation invariance} 
  The map $I:C^2_b(\mathbb{R}^d)\to C^0_b(\mathbb{R}^d)$ is translation invariant. Namely, for any $x,z\in\mathbb{R}^d$ and $u\in C^2_b(\mathbb{R}^d)$ we have
  \begin{align}\label{eqn:translation invariance}
    I(\tau_z u,x) = I(u,x+z),\;\textnormal{where } \tau_z u(x):=u(x+z).
  \end{align}
\end{assumption}

\begin{assumption}\label{assumption:tightness bound} 
  There is a non-increasing function $\rho:(0,\infty)\to \mathbb{R}$ with $\rho(R)\to 0$ as $R\to \infty$ such that if $u,v\in C^2_b(\mathbb{R}^d)$ are such that $u\equiv v$ in $B_{2R}(x_0)$, then  
  \begin{align*}
    \| I(u)-I(v) \|_{L^\infty(B_R(x_0))} \leq \rho(R)\|u-v\|_{L^\infty(\mathbb{R}^d)}.
  \end{align*}	  
\end{assumption}

\begin{assumption}\label{assumption:coefficient regularity} 
  There exists a modulus, $\om$, for all $v,u \in C^2_b(\mathbb{R}^d)$, $x, z \in\mathbb{R}^d$, $r>0$, we have
  \begin{align*}
    & | I(v+\tau_{-z}u,x+z)-I(v,x+z) -\left ( I(v+u,x)-I(v,x)\right ) | \\
    & \leq \omega(|z|)C(r)\left ( \|u\|_{C^2(B_{2r}(x))} + \|u\|_{L^\infty(\mathcal{C}B_r(x))}\right ).
  \end{align*}
  It is allowed that $C(r)\to\infty$ as $r\to0$; in some examples $C(r)$ may be bounded and in some it may be unbounded.
\end{assumption}

The meaning of Assumption \ref{assumption:GCP} and Assumption \ref{assumption:translation invariance} is self-evident. Assumption \ref{assumption:tightness bound} seems rather technical, but it will be necessary to obtain compactness for a family of measures arising in the proof (and this assumption is satisfied by a broad family of examples). Note however that this assumption is not needed for the translation invariant case as well as the setting of Theorem \ref{theorem:MinMax Euclidean} as these two theorems are obtained with different methods. 

Last but not least, Assumption \ref{assumption:coefficient regularity} can be thought of as a ``coefficient regularity'' assumption. For instance, in the linear and local case, in which $I$ is a L\'evy operator without integral part, Assumption \ref{assumption:coefficient regularity} is equivalent to the coefficients of the operator having modulus of continuity $C\omega(\cdot)$ for some constant $C>0$.  In fact, Assumption \ref{assumption:coefficient regularity} is stated so that it indeed linearizes to this usual assumption that one expects in the linear case.

\begin{remark}
	As mentioned above, one can check that for linear operators, Assumption \ref{assumption:coefficient regularity} is equivalent to the coefficients of the local part being uniformly continuous and the L\'evy measures being uniformly continuous in the TV norm along shifts in the base point, i.e.
	\begin{align*}
		\norm{\mu(x+x,\cdot)-\mu(x,\cdot)}_{TV(\mathcal{C}B_r)}\leq C\om(\abs{z}).
	\end{align*}
	By its design, Assumption \ref{assumption:coefficient regularity} is a technical artifact of our proof, and as such, it is unlikely to be sharp or even the most natural assumption.  There is most likely room for improvement here.  In fact, one indication of the possibility to make a more natural assumption lies in the fact that even when the original operator, $I$, is translation invariant (so the most regular dependence on $x$), it does not necessarily follow that $I$ also satisfies Assumption \ref{assumption:coefficient regularity}.  This also reflects the fact that we have taken a two completely different methods of proof for the results that concern translation invariant operators, and ones that have a modulus with respect to translations.
\end{remark}

\begin{remark}
	In Section \ref{section:examples}, we give a short list of some operators that fall within the scope of Assumptions \ref{assumption:GCP}--\ref{assumption:coefficient regularity} and Theorems \ref{theorem:MinMax Euclidean}--\ref{theorem:minmax with beta less than 2}. At the end of Section \ref{section:examples}, we give a list of which assumptions each example satisfies.
\end{remark}

\begin{remark}
	We note that one subtle improvement of the current work upon our previous one in \cite{GuSc-2016MinMaxNonlocalarXiv} is that because of a more streamlined proof for the translation invariant case, we were able to establish the non-translation invariant case, Theorem \ref{theorem:MinMax Euclidean} (below), without the technical Assumption \ref{assumption:tightness bound}.  This is purely an artifact of using an approximation scheme in \cite{GuSc-2016MinMaxNonlocalarXiv} to treat all operators by the same method, and this turns out to have been not essential when one does not want the extra information provided by Theorems \ref{theorem:MinMax Euclidean ver2} and \ref{theorem:minmax with beta less than 2}.
\end{remark}

The first theorem uses the notion of ``pointwise'' $C^2$ or $C^1$, and so we will define that property here.

\begin{definition}\label{def:PointwiseC1C2}
	For a fixed $x$ we say that $u\in C^2(x)$ (``pointwise $C^2$ at $x$'') if there exists a vector, $\grad u(x)$, and a symmetric matrix, $D^2u(x)$, such that
	\begin{align*}
		\text{as}\ y\to x,\ \ \abs{u(y)-u(x)-\grad u(x)\cdot(y-x)-\frac{1}{2}(y-x)\cdot \left( D^2u(x)(y-x)\right)}\leq o(\abs{y-x}^2).
	\end{align*}
	Similarly if $u$ only enjoys the existence of $\grad u(x)$ and
	\begin{align*}
		\text{as}\ y\to x,\ \ \abs{u(y)-u(x)-\grad u(x)\cdot(y-x)}\leq o(\abs{y-x}),
	\end{align*}
	we say that $u\in C^1(x)$ (``pointwise $C^1$ at $x$'').
\end{definition}

Now we can restate Theorems 1--3 above, in more precise terms.

\begin{theorem}\label{theorem:MinMax Euclidean}
  If $I:C^2_b(\mathbb{R}^d)\to C^0_b(\mathbb{R}^d)$ satisfies Assumption \ref{assumption:GCP}, then, for each $x$, there exists a family of linear functionals on $C^2(x)$ that depend on $I$ and $x$, called $\mathcal{K}(I)_x$, so that for all $u\in C^2(x)$
  \begin{align*}%\label{eqIN:MinMaxTheoremEnvironment}
    I(u,x) = \min\limits_{v\in C^2_b(\mathbb{R}^d)}\max\limits_{L \in \mathcal{K}(I)_x} \{ I(v,x)+L(u-v) \}.	  
  \end{align*}	  
  Here, each $L\in\mathcal{K(I)}_x$, has the form 
\begin{align*}
	L(u)= \tr(A_xD^2u(x)) + B_x\cdot\nabla u(x) + C_x u(x)+
	\int_{\mathbb{R}^d} u(x+y)-u(x)-\ind_{B_1(0)}(y)\nabla u(x)\cdot y\;\mu_x(dy),
\end{align*}
 and for some universal $C$, the terms also satisfy the bound for all $x$:
  \begin{align*}%\label{eqIN:UpperBoundCoefficients}
    \abs{A_x} +\abs{B_x} + \abs{C_x} + \int_{\mathbb{R}^d} \min\{1,|y|^2\}\;\mu_x(dy) \leq C\norm{I}_{\text{Lip},C^2_b\to C^0_b}.		  
  \end{align*}	  
\end{theorem}

\noindent The proof of Theorem \ref{theorem:MinMax Euclidean} appears in Section \ref{sec:TheoremsWithoutWhitney}, which is at the end of Section \ref{section:Functionals with the GCP}.

We want to point out to the reader that the notation in Theorem \ref{theorem:MinMax Euclidean} is intentional in its use of subscripts for e.g. $A_x$, etc.  This is because our construction does not actually produce $L$ as a linear mapping $C^2_b\to C^0_b$, and so it is not correct to think of having a family of $L$ whose coefficients are actually \emph{functions} of $x$.  Rather, it just says that at each $x$ there is a family functionals that have the desired structure, but it is not clear that they can be put together across all $x$ to make a family of $x$-dependent operators.

This situation changes under other assumptions, and in the next two theorems, our method produces a family of linear operators mapping $C^2_b(\real^d)\to C^0_b(\real^d)$, all of the form (\ref{eqIN:LevyTypeLinear}).  

\begin{theorem}\label{theorem:MinMax Translation Invariant}
  If $I:C^2_b(\mathbb{R}^d)\to C^0_b(\mathbb{R}^d)$ satisfies Assumption \ref{assumption:GCP} and Assumption \ref{assumption:translation invariance} then there exists a family, 
  $
  \displaystyle
  \{f_{ab}, L_{ab}\}_{a,b\in\mathcal{K}(I)},
  $
  that depends only on $I$,
  where for all $a,b$, $f_{ab}$ are constants, and $L_{ab}$ are linear translation invariant operators mapping $C^2_b(\real^d)\to C^0_b(\real^d)$ of the form (\ref{eqIN:LevyTypeLinear}) (i.e. constant coefficients), and for all $u\in C^2_b(\mathbb{R}^d)$ and $x\in \mathbb{R}^d$ we have 
  \begin{align*}%\label{eqIN:MinMax-ab-TheoremEnvironment}
    I(u,x) = \min\limits_{a}\max\limits_{b} \{ f_{ab}+L_{ab}(u,x) \}.	  
  \end{align*}
  Furthermore, for a universal $C$, for all $f_{ab}$ and $L_{ab}$,
  \begin{align*}
   \abs{f_{ab}} + \abs{A_{ab}} +\abs{B_{ab}} + \abs{C_{ab}} + \int_{\mathbb{R}^d} \min\{1,|y|^2\}\;\mu_{ab}(dy) \leq C\norm{I}_{\text{Lip},C^2_b\to C^0_b}.		  
  \end{align*}
\end{theorem}

The proof of Theorem \ref{theorem:MinMax Translation Invariant} appears in Section \ref{sec:TheoremsWithoutWhitney}, which is at the end of Section \ref{section:Functionals with the GCP}.

\begin{theorem}\label{theorem:MinMax Euclidean ver2}
  If $I:C^2_b(\mathbb{R}^d)\to C_b^0(\mathbb{R}^d)$ satisfies Assumption \ref{assumption:GCP}, Assumption \ref{assumption:tightness bound}, and  Assumption \ref{assumption:coefficient regularity}, then, there exists a family, 
  $
  \displaystyle
  \{f_{ab}, L_{ab}\}_{a,b\in\mathcal{K}(I)},
  $
  that depends only on $I$,
  where for all $a,b$, $f_{ab}\in C^0_b(\real^d)$ are functions, and $L_{ab}$ are linear operators mapping $C^2_b(\real^d)\to C^0_b(\real^d)$ of the form (\ref{eqIN:LevyTypeLinear}), and for all $u \in C^2_b(\mathbb{R}^d)$, we have   
  \begin{align*}
    I(u,x) = \min\limits_{a}\max\limits_{b} \{ f_{ab}(x)+L_{ab}(u,x) \},	  
  \end{align*}
  and for a universal $C$, for all $f_{ab}$ and $L_{ab}$,
  \begin{align*}
   \norm{f_{ab}}_{L^\infty} + \norm{A_{ab}}_{L^\infty} +\norm{B_{ab}}_{L^\infty} + \norm{C_{ab}}_{L^\infty} + \sup_x\int_{\mathbb{R}^d} \min\{1,|y|^2\}\;\mu_{ab}(x,dy) \leq C\norm{I}_{\text{Lip},C^2_b\to C^0_b}.		  
  \end{align*}
  
  Furthermore, if $\om$ is as in Assumption \ref{assumption:coefficient regularity}, then the functions $f_{ab},A_{ab},B_{ab},C_{ab},$ all have a modulus of continuity $C\omega(2\cdot)$, while for each $r>0$ we have the estimate,
      \begin{align}\label{eqIN:ThmModulusInTVNorm}
        \|\mu_{ab}(x_1)-\mu_{ab}(x_2)\|_{\textnormal{TV}(\mathcal{C}B_r)} \leq C(r)\omega(2|x_1-x_2|),
      \end{align}  
	  where as above, $C(r)>0$, is a constant that may possibly (but not necessarily) have the property that $C(r)\to\infty$ as $r\to0$. 
\end{theorem}

\noindent The proof of Theorem \ref{theorem:MinMax Euclidean ver2} appears in Section \ref{sec:TheoremsThatUseWhitney}, which is at the end of Section \ref{section:Analysis of finite dimensional approximations}.

Finally, we give a theorem that reduces the possible terms in the min-max over (\ref{eqIN:LevyTypeLinear}).  Namely, there are instances in which there may be no second order terms or first order terms.  To state this, we abuse notation slightly, and we give a shorthand as $C^{\beta}_b(\real^d)$ to mean the following:
\begin{align}\label{equation:Cbeta definition}
	\begin{array}{rl} 
		&\text{if}\ \beta = 2 + \gamma,\ \textnormal{for}\ \gam\in (0,1),\ \text{then, we mean}\ C^\beta_b(\real^d) = C^{2,\gamma}_b(\real^d);\\
		&\text{if}\ \beta = 2^+,\ \text{then, we mean}\ C^\beta_b(\real^d) = C^{2}_b(\real^d);\\
	&\text{if}\ \beta = 2,\ \text{then, we mean}\ C^\beta_b(\real^d) = C^{1,1}_b(\real^d);\\
	&\text{if}\ \beta = 1+\gam,\ \text{for}\ \gam\in(0,1), \text{then, we mean}\ C^\beta_b(\real^d) = C^{1,\gam}_b(\real^d);\\
	&\text{if}\ \beta = 1^+,\ \text{then, we mean}\ C^\beta_b(\real^d) = C^1_b(\real^d);\\
	&\text{if}\ \beta = 1,\ \text{then, we mean}\ C^\beta_b(\real^d) = C^{0,1}_b(\real^d);\\
	&\text{if}\ \beta = \gam,\ \text{for}\ \gam\in(0,1), \text{then, we mean}\ C^\beta_b(\real^d) = C^{0,\gam}_b(\real^d).
\end{array}
\end{align}

\begin{definition}\label{def:PointwiseCBeta}
	For a fixed $x$, we say that $u\in C^{\beta}(x)$ (``pointwise $C^\beta(x)$'') if the same requirements of Definition \ref{def:PointwiseC1C2} hold, but the estimate on the right hand side takes into account the different decay as follows:
	
	\begin{itemize}
		\item if, $\beta=2+\gam$, then $u$ has a second order Taylor expansion and the right hand side is $O(\abs{y-x}^{2+\gam})$;
		\item if, $\beta=2^+$, then $u$ has a second order Taylor expansion and the right hand side is $o(\abs{y-x}^{2})$;
		\item if, $\beta=2$, then we include this in the previous case whenever $u$ has a second order taylor expansion at $x$;
		\item if, $\beta=1+\gam$, then $u$ has a first order Taylor expansion and the right hand side is $O(\abs{y-x}^{1+\gam})$;
		\item if, $\beta=1^+$, then $u$ has a first order Taylor expansion and the right hand side is $o(\abs{y-x})$;
		\item if, $\beta=1$, then we include this in the previous case whenever $u$ has a first order taylor expansion at $x$;
		\item if, $\beta=\gam\in(0,1)$, then $\abs{u(y)-u(x)}\leq C\abs{y-x}^\gam$.
	\end{itemize} 

\end{definition}

\begin{assumption}\label{assumption:CBeta}
	All of Assumptions \ref{assumption:GCP} -- \ref{assumption:coefficient regularity} hold, but with all instances of $C^2_b(\real^d)$ replaced by $C^\beta_b(\real^d)$.
\end{assumption}

\begin{theorem}\label{theorem:minmax with beta less than 2}
  For each of Theorems \ref{theorem:MinMax Euclidean}, \ref{theorem:MinMax Translation Invariant}, \ref{theorem:MinMax Euclidean ver2},  we have the following variation: in each case assume that $I$ satisfies Assumption \ref{assumption:CBeta}, for some $\beta \in [0,2^+]$ (as enumerated above). Then, taking into account Definition \ref{def:PointwiseCBeta} for Theorem \ref{theorem:MinMax Euclidean}, the min-max formula holds in each of the previous results with the following additions: if $\beta<2$ then $A_{ab} = 0$ for all $a,b$, while if $\beta <1$ then $B_{ab} = 0$ for all $a,b$ and the operators $L_{ab}$ take the form
  \begin{align*}
    L_{ab}(u,x) & = C_{ab}(x)u(x)+\int_{\mathbb{R}^d} u(x+y)-u(x)\;\mu_{ab}(x,dy). 
  \end{align*}	  
  Moreover, the smaller $\beta$, the more regular the L\'evy measures $\mu_{ab}$ are at $y=0$, namely, we have
  \begin{align*}
    \sup \limits_{a,b,x}\int_{\mathbb{R}^d}\min\{1,|y|^\beta\}\mu_{ab}(x,dy) <\infty.	  
  \end{align*}	  
\end{theorem}

The proof of Theorem \ref{theorem:minmax with beta less than 2} appears in Section \ref{sec:TheoremsThatUseWhitney}, which is at the end of Section \ref{section:Analysis of finite dimensional approximations}.

\begin{remark}
	In Sections \ref{section:Finite Dimensional Approximations} and \ref{section:Analysis of finite dimensional approximations}, one can see that at its heart, the fact that the modulus for $I$ is passed onto the coefficient functions in (\ref{eqIN:LevyTypeLinear}) is a consequence of our choice to use a Whitney extension in an approximation to $I$, and the Whitney extension is well known to preserve a modulus of continuity.  The actual details are a bit more involved, but that is the main reason.  We note the presence of the factor of $2$ in the new modulus is a consequence of the Whitney Extension method; the interested reader can see \cite[Chapter VI]{Stei-71}.
\end{remark}

A further comment regarding the assumptions is in order. Suppose that $I$ satisfies Assumption \ref{assumption:coefficient regularity} with $\omega \equiv 0$. In this case, taking $v\equiv 0$ the assumption says that
\begin{align*}
  I(\tau_{-h}u,x+h)-I(0,x+h) = I(u,x)-I(0,x), 
\end{align*}
and if we further assume that $I(0,x)$ is constant (i.e. $I$ applied to the zero function returns a constant), then we have 
\begin{align*}
  I(\tau_{-h}u,x+h) = I(u,x),
\end{align*}
that is, $I$ is translation invariant. However, at first sight it is not clear what happens in the reverse direction. That is, we do not know how to show that a translation-invariant operator automatically satisfies Assumption \ref{assumption:coefficient regularity} with $\omega \equiv 0$, and in fact we expect that this assumption can be modified so that it seamlessly includes the translation invariant operators as well.

%%%%%%%%%%%%%%%%%%%%%%%%%%%%%%%%%%%%%%%%%%%%%%
%%%%%%%%%%%%%%%%%%%%%%%%%%%%%%%%%%%%%%%%%%%%%%
\subsection{Notation}\label{section_sub:Notation}

For the readers' convenience, a summary of symbols used in the paper is presented below.

\bigskip

\begin{tabular}{lll} 
\allowdisplaybreaks
Notation  & Definition \\
\hline\\
$d$ & space dimension\\
$C^2_b$ & twice differentiable functions $f$ with bounded $f,\nabla f,$ and $D^2f$\\
$C^\beta_b$ & bounded functions of class $C^\beta$, see \eqref{equation:Cbeta definition} for definition\\
$\mathbb{S}_d$ & symmetric matrices of size $d\times d$\\
$\|\cdot\|_{\textnormal{TV}}$ & total variation norm for a measure\\
$L(X,Y)$ & space of bounded linear operators from $X$ to $Y$\\
$\textnormal{c.h.}(E)$ & the convex hull of a set $E$\\
$\mathcal{C}E$ & complement of a subset of $\mathbb{R}^d$\\
$F^0(x,v)$ & upper gradient of a Lipschitz function (Definition \ref{definition:upper gradient})\\
$\partial F(x)$ & generalized gradient of $F$ at $x$ (Definition \ref{definition:generalized gradient})\\
$G_n$ & grid with step size $2^{-n}$\\
$C(G_n)$ & space of real valued functions defined in $G_n$ (Definition \ref{definition:discrete function spaces})\\
$C_*(G_n)$ & subset of $C(G_n)$ of functions vanishing outside $[-2^n,2^n]\cap G_n$ (Definition \ref{definition:discrete function spaces})\\
$(\nabla_n)^1u(x)$ & discrete gradient for step size $2^{-n}$ (Definition \ref{definition:finite difference operators 1})\\
$(\nabla_n)^2u(x)$ & discrete Hessian for step size $2^{-n}$ (Definition \ref{definition:finite difference operators 2})  
\end{tabular}

%%%%%%%%%%%%%%%%%%%%%%%%%%%%%%%%%%%%%%%%%%%%%%
%%%%%%%%%%%%%%%%%%%%%%%%%%%%%%%%%%%%%%%%%%%%%%
%%%%%%%%%%%%%%%%%%%%%%%%%%%%%%%%%%%%%%%%%%%%%%
\subsection{Background}\label{section_sub:Background}

There were roughly two reasons that motivated the results we present in this paper.  First of all, the link between elliptic equations and a min-max formula for operators has a long history, and it has been exploited extensively in the case of \emph{local} operators.  Until \cite{GuSc-2016MinMaxNonlocalarXiv}, the connection was not known for nonlocal, nonlinear operators.  Even so, the link between the two was natural enough that there are at least a few results that assumed a structure like (\ref{eqIN:MinMaxMeta}), including \cite{BaIm-07}, \cite{JakobsenKarlsen-2006maxpple}, \cite{KoikeSwiech-2013RepFormulaIntegroPDE-IUMJ}, \cite{Schw-10Per}, \cite{Schw-12StochCPDE}, \cite{Silv-2011DifferentiabilityCriticalHJ}, among many others.  Thus the theorems here and in \cite{GuSc-2016MinMaxNonlocalarXiv} give a sort of a posteriori justification to min-max assumptions that appeared in earlier works.  Secondly, a formula such as (\ref{eqIN:MinMaxMeta}) can be very useful in connecting results about the integro-differential theory (of which, there has been a large volume recently) with some other pursuits that may not obviously relate to operators such as (\ref{eqIN:LevyTypeLinear}).  Two recent projects that exploit or were motivated by the min-max formulas are on some Hele-Shaw type free boundary evolutions in \cite{ChangLaraGuillenSchwab2018FBasNonlocal} and some Neumann homogenization problems \cite{GuSc-2014NeumannHomogPart1DCDS-A} \cite{GuSc-2018NeumannHomogPart2SIAM}. Both of these relate to linear and nonlinear Dirichlet-to-Neumann maps, studied in \cite{GuillenKitagawaSchwab2017estimatesDtoN}, and there is plenty more to learn about the integro-differential structure in the nonlinear setting.  The choice to pursue continuity properties such as the dependence given in (\ref{eqIN:ThmModulusInTVNorm}), although a posteriori seems straightforward, was not initially obvious, and it was motivated by recent results about comparison theorems for viscosity solutions of integro-differential equations in \cite{GuillenMouSwiech2018coupling}.

As mentioned earlier, for linear operators, the representation of (\ref{eqIN:LevyTypeLinear}) goes back to Courr\`ege \cite{Courrege-1965formePrincipeMaximum}.  This was naturally connected with generators of Markov processes and boundary excursion processes for reflected diffusions.  Hsu \cite{Hsu-1986ExcursionsReflectingBM} provides a similar representation for the Dirichlet to Neumann map for the Laplacian in a smooth domain $\Omega$, and this corresponds to studying the boundary process for a reflected Brownian motion.  If $I$ is not necessarily linear but happens to satisfy the stronger \emph{local} comparison principle,  there are min-max results by many authors, e.g. Evans \cite{Evans-1984MinMaxRepresentations}, Souganidis \cite{Souganidis-1985MaxMinRep}, Evans-Souganidis \cite{EvansSoug-84DiffGameRepresentation} and Katsoulakis \cite{Katsou-1995RepresentationDegParEq}. In this case, the operator takes the form,
\begin{align*}
  I(u,x) = F(x,u(x),\nabla u(x),D^2u(x)),	
\end{align*}	
which can be expressed as in Theorem \ref{theorem:MinMax Euclidean}, but with $\mu(x,dh)\equiv 0$.
This was extended to even include the possibility of weak solutions acting as a \emph{local} semi-group on $BUC(\real^d)$, related to image processing, in Alvarez-Guichard-Lions-Morel \cite{AlvarezLionsGuichardMorel-1993AxiomsImageProARMA}, and to weak solutions of sets satisfying an order preserving set flow by Barles-Souganidis in \cite{BarlesSouganidis-1998NewApproachFrontsARMA}.  In \cite{AlvarezLionsGuichardMorel-1993AxiomsImageProARMA} it was shown under quite general assumptions that certain nonlinear semigroups must be represented as the unique viscosity solution to a degenerate parabolic equation.

Although it is still too early to tell, one hopes that theorems like those presented here can create a bridge between some nonlocal equations for which regularity questions arise and the known results about such equations when a min-max structured is known to hold.  In the \emph{local} setting, there are a number of results that leverage the min-max to shed new light on certain issues, and it would be interesting to see if similar things can be done for the nonlocal theory (see the discussion in \cite[Section 1]{GuSc-2016MinMaxNonlocalarXiv} for an incomplete list of such results).  The types of regularity results that could find new applications via the min-max theorems here fall into roughly three categories: Krylov-Safonov type results; regularity for translation invariant equations; and Schauder type regularity results.  For Krylov-Safonov, this means that solutions of fully nonlinear equations can be shown to enjoy H\"older estimates depending only on the $L^\infty$ norm of the solution; some examples are: \cite{CaSi-09RegularityIntegroDiff}, \cite{Chan-2012NonlocalDriftArxiv}, \cite{ChDa-2012NonsymKernels}, \cite{KassRangSchwa-2013RegularityDirectionalINDIANA}, and \cite{SchwabSilvestre-2014RegularityIntDiffVeryIrregKernelsAPDE}, among many others.  For translation invariant equations, these are the results that show solutions to translation invariant equations very often enjoy $C^{1,\alpha}$ regularity under mild assumptions; some examples are: \cite{CaSi-09RegularityIntegroDiff}, \cite{ChangLaraKriventsov2017}, \cite{Kriventsov-2013RegRoughKernelsCPDE}, \cite{Ros-OtonSerra-2016RegularityStableOpsJDE}, \cite{Serra-2015RegFullyNonlinIntDiffRoughKernelsCalcVar}, among others.  Finally, for Schauder regularity, we mean results that show that for $x$-dependent operators, under certain regularity for the coefficients (such as Dini), solutions will have as much regularity as those equations with ``constant coefficients''; some examples are: \cite{DongJinZhang2018DiniSchauderNonlocal}, \cite{JinXiong-2015SchauderEstLinearParabolicIntDiffDCDS-A}, \cite{MouZhang2018preprint}, among others.  On top of questions of the type of Krylov-Safonov regularity mentioned above, there is another family of regularity results that accompanies existence and uniqueness techniques for viscosity solutions of elliptic partial-differential / integro-differential equations, and it is typically referred to as the Ishii-Lions method, going back to \cite{IshiiLions-1990ViscositySolutions2ndOrder}.  Both this Ishii-Lions regularity and comparison results could connect well with the operators treated in this paper, as many of the existing works on nonlocal equations assume a min-max.  The types of results that could be applicable are like those in \cite{BaChIm-08Dirichlet}, \cite{BaChIm-11Holder}, \cite{BaChCiIm-2012LipschitzMixedEq}, \cite{BaIm-07}, and \cite{JakobsenKarlsen-2006maxpple}, among others.

There is some more discussion of related works and background inside of the examples that we list in Section \ref{section:examples}.

%%%%%%%%%%%%%%%%%%%%%%%%%%%%%%%%%%%%%%%%%%%%%%
%%%%%%%%%%%%%%%%%%%%%%%%%%%%%%%%%
%%%%%%%%%%%%%%%%%%%%%%%%%%%%%%%%%
\subsection{Another description of operators satisfying the GCP}

Let us describe an elementary but useful way to view operators satisfying the GCP, which is also related to the min-max representation. First, we introduce a family of functional spaces.
\begin{definition}\label{definition:L_beta^infinity spaces}
  For $\beta \in [0,2^+]$ (using the abuse of notation in (\ref{equation:Cbeta definition})) we define the space $L^\infty_{\beta}$ as follows. First, if $\beta \neq 1^+$,
  \begin{align*}
    L^\infty_\beta := \{ h \in L^\infty(\mathbb{R}^d) \;\mid\; |h(y)| = O(|y|^\beta) \textnormal{ as  } |y|\to 0 \},
  \end{align*}
  while for $\beta=1^+$,
  \begin{align*}
    L^\infty_\beta & := \{ h \in L^\infty(\mathbb{R}^d) \;\mid\; |h(y)| = o(|y|^\beta) \textnormal{ as } y\to 0\}.
  \end{align*}
  (We note the first space requires ``Big-O'', while the second space requires ``little-o''.) The spaces $L^\infty_\beta$ are Banach spaces, with norms given by
  \begin{align*}
    \sup \limits_{y}|h(y)| \min\{1,|y|^\beta\}^{-1}.
  \end{align*}	  
\end{definition}	
Now, suppose we are given a continuous function
\begin{align*}
  F: L^\infty_\beta(\mathbb{R}^d) \times \mathbb{S}_d \times \mathbb{R}^d\times \mathbb{R} \times \mathbb{R}^d\to\mathbb{R}.
\end{align*}	   
Assume that this function is monotone (non-decreasing) with respect to the first two variables. Then, given $u \in C^\beta_b(\mathbb{R}^d)$ define
\begin{align*}
  I(u,x) := F(\delta_x u,D^2u(x),\nabla u(x),u(x),x)
\end{align*}
where we are using the notation $\delta_x u(y):= u(x+y)-u(x)-\nabla u(x)\cdot y \chi_{B_1(0)}(y)$ for $\beta\geq 1$, and $\delta_x u(y):=u(x+y)-u(x)$ for $\beta<1$. It is clear the operator $I$ thus defined has the GCP. 

Do all operators with the GCP arise in this form? It is easy to see that the answer is positive, at least when $\beta<2$. Given $I:C^\beta(\mathbb{R}^d)\to C^0(\mathbb{R})$, with $\beta<2$, we define a function
\begin{align*}
 F: L^\infty_\beta(\mathbb{R}^d) \times \mathbb{R}^d\times \mathbb{R} \times \mathbb{R}^d\to\mathbb{R},	   
\end{align*}	   
by the formula $F(h,p,u,x) := I( \tau_{-x}h+\tau_{-x}p\cdot(\cdot)\chi_{B_1}+u,x)$. It is straightforward to see that for $F$ so defined and $u \in C^\beta_b(\mathbb{R}^d)$ we have
\begin{align*}
  I(u,x) = F(\delta_x u,\nabla u(x),u(x),x).
\end{align*}

%%%%%%%%%%%%%%%%%%%%%%%%%%%%%%%%%
%%%%%%%%%%%%%%%%%%%%%%%%%%%%%%%%%
%%%%%%%%%%%%%%%%%%%%%%%%%%%%%%%%%
%%%%%%%%%%%%%%%%%%%%%%%%%%%%%%%%%
%%%%%%%%%%%%%%%%%%%%%%%%%%%%%%%%%
%%%%%%%%%%%%%%%%%%%%%%%%%%%%%%%%%
\section{Real valued Lipschitz functions on Banach Spaces}\label{section:Real Valued Lipschitz Functions}

In this section we review various well known facts about Lipschitz functions on Banach spaces, following Clarke's book \cite[Chapter 2]{Cla1990optimization}. We will refer most of the proofs to the relevant section in \cite{Cla1990optimization}. The section ends with Theorem \ref{theorem:MinMax for scalar functionals} which yields a min-max formula for any real valued, Lipschitz $F$, such a result is neither new nor surprising, but we present it here in complete detail for the sake of completeness.

We fix a Banach Space, denoted by $X$, an open convex subset $\mathcal{K}\subset X$, and a function 
\begin{align*}
  F:\mathcal{K}\subset X\to \mathbb{R},
\end{align*}
which is assumed Lipschitz with constant $L>0$, that is
\begin{align}\label{eqnLipschitzFunctionals:Lipschitz constant}
  |F(x)-F(y)|\leq L\|x-y\| \;\;\forall\;x,y\in \mathcal{K}.
\end{align}

\begin{definition}\label{definition:upper gradient}
  The upper gradient of $F$ at $x\in \mathcal{K}$ in the direction of $v\in X$, is defined as
  \begin{align*}
    F^0(x,v) := \limsup\limits_{t\searrow 0} \frac{F(x+tv)-F(x)}{t}.
  \end{align*}	
  This can be seen as a function $F^0:\mathcal{K}\times X\to \mathbb{R}$.
    
\end{definition}

\begin{proposition}\label{proposition:upper gradient properties}
  The function $F^0(x,v)$ has the following properties	
  \begin{enumerate}
    \item For any $x\in \mathcal{K},v\in X$, and $\lambda>0$ we have $F^0(x,\lambda v) =\lambda F^0(x,v)$.
    \item For any $x\in \mathcal{K}$, and $v,w\in X$ we have $|F^0(x,v)-F^0(x,w)|\leq L\|v-w\|$.
    \item If $(x_k,v_k)\to(x,v)$ then $\limsup F^0(x_k,v_k) \leq F^0(x,v)$.
    \item $F^0(x,-v)=(-F)^0(x,v)$.
  \end{enumerate}
\end{proposition}
\begin{proof}
  We refer the reader to \cite[Proposition 2.1.1]{Cla1990optimization}.
\end{proof}

\begin{definition}\label{definition:generalized gradient}
  The generalized gradient of $F$ at $x\in \mathcal{K}$ is the subset of $X^*$ given by
  \begin{align*}
    \partial F(x) := \{ \ell \in X^* \mid  F^0(x,v)\geq \langle \ell,v\rangle \;\;\forall\;v\in X\}. 
  \end{align*}	
  We will denote by $\partial F$ the convex hull of the union of $\partial F(x)$,
  \begin{align*}
    \partial F := \hull \left (  \bigcup \limits_{x\in \mathcal{K}} \partial F(x)\right).	  
  \end{align*}	        
\end{definition}

\begin{proposition}\label{proposition:generalized gradient properties}
  The set $\partial F(x)$, $x\in \mathcal{K}$, has the following properties
  \begin{enumerate}
    \item $\partial F(x)$ is a non-empty, convex, $\textnormal{weak}^*$-compact subset of $X^*$.
    \item $\|\ell\| \leq L$ for every $\ell\in \partial F(x)$.
    \item For any $v\in X$, we have that
    \begin{align*}
      F^0(x,v) = \max \limits_{\ell\in \partial F(x)} \langle \ell,v\rangle. 		
    \end{align*}			  
  \end{enumerate}	  
\end{proposition}

\begin{proof}
  We refer the reader to \cite[Proposition 2.1.2]{Cla1990optimization}.

\end{proof}

The following theorem, due to Lebourg, is a generalization of the mean value theorem for differentiable functions.
\begin{theorem}[Lebourg's Theorem]\label{theorem:Lebourg}
  Let $x,y$ be points in $\mathcal{K}$. Then there exist $z$ of the form $z= tx+(1-t)y$ for some $t\in[0,1]$, such that  for some $\ell \in \partial F(z)$
  \begin{align*}
    F(x)-F(y) = \langle \ell,x-y\rangle. 	  
  \end{align*}	  
  
\end{theorem}

\begin{proof}
  We refer the reader to \cite[Theorem 2.3.7]{Cla1990optimization}.

\end{proof}

Using the generalized gradient and Lebourg's theorem we can easily prove a min-max formula for Lipschitz functionals. Observe this is a general result for Lipschitz functionals in general Banach spaces, and it does not involve anything like GCP (functionals with the GCP on $C^\beta_b(\real^d)$ are considered in the next section).
\begin{theorem}\label{theorem:MinMax for scalar functionals}
  Let $F:\mathcal{K}\subset X\to\mathbb{R}$ be a Lipschitz function, with $\mathcal{K}$ convex, then for all $x\in \mathcal{K}$,
  \begin{align*}
    F(x) = \min\limits_{y\in \mathcal{K}}\max\limits_{\ell \in \partial F} \{F(y)+\langle \ell,y-x\rangle  \}.	  
  \end{align*}	  
\end{theorem}

\begin{proof}
  According to Theorem \ref{theorem:Lebourg}, given $x,y\in \mathcal{K}$ there is some $\ell\in\partial F$ such that
  \begin{align*}
    F(x)-F(y) = \langle \ell,x-y\rangle.	  
  \end{align*}	  
  In other words, for any $x$ and $y$ in $\mathcal{K}$ we have the inequality
  \begin{align*}
    F(x) \leq  \max \limits_{\ell \in \partial F}\left \{ F(y)+\langle \ell,x-y\rangle \right \}.	  
  \end{align*}	 
  This also yields an equality for $y=x$, thus $F(x) = \min \limits_{y\in \mathcal{K}}\max \limits_{\ell \in \partial F}\left \{ F(y)+\langle \ell,x-y\rangle \right \}$.

\end{proof}

%%%%%%%%%%%%%%%%%%%%%%%%%%%%%%%%%
%%%%%%%%%%%%%%%%%%%%%%%%%%%%%%%%%
%%%%%%%%%%%%%%%%%%%%%%%%%%%%%%%%%
%%%%%%%%%%%%%%%%%%%%%%%%%%%%%%%%%
%%%%%%%%%%%%%%%%%%%%%%%%%%%%%%%%%
%%%%%%%%%%%%%%%%%%%%%%%%%%%%%%%%%
\section{Functionals with the GCP, revisited}\label{section:Functionals with the GCP}

Throughout this section $\mathcal{K}$ denotes an open convex set of $C^\beta_b(\mathbb{R}^d)$ (see \eqref{equation:Cbeta definition}). Moreover, for $\rho>0$, we shall write
\begin{align*}
  \mathcal{K}_{\rho} = \big \{ u \in C^\beta_b(\real^d) \;\mid \; \|v-u\|_{C^\beta}<\rho \Rightarrow v \in \mathcal{K} \big \}.
\end{align*}
\begin{definition}\label{definition:GCP with respect to a point}
  Let $F$ be a map $F:\mathcal{K}\subset C^\beta_b(\mathbb{R}^d)\to\mathbb{R}$ and $x\in \mathbb{R}^d$. Such a functional is said to have the Global Comparison Property with respect to $x$ if $F(u)\leq F(v)$ for any pair of functions $u,v\in \mathcal{K}$ such that $u(y)\leq v(y)$ for all $y$ and $u(x)=v(x)$ --we will say in such a case that $v$ touches $u$ from above at $x$.
\end{definition}

The following two auxiliary functions will be useful throughout the section: Fix $\phi_0:\mathbb{R}\to\mathbb{R}$, a nondecreasing $C^\infty$ function such that $0\leq \phi_0 \leq 1$, $\phi_0(x)=0$ for $x\leq 0$, $\phi_0(x)=1$ for $x\geq 1$. Then, given $r,R>0$ we define the functions
\begin{align}
  \phi_{r,R}(y) & := \phi_0\left ( \frac{|y|-R}{r} \right )  \label{equation:phi sub r R}\\
  \psi_{r,R}(y) & := 1-\phi_{r,R}(y) \label{equation:psi sub r R}
\end{align}
The following Proposition was first proved in \cite[Lemma 4.15, Corollary 4.16]{GuSc-2016MinMaxNonlocalarXiv}, we review the proof here for the reader's convenience.

\begin{proposition}\label{proposition:GCP implies weak localization}
  Suppose that $F:\mathcal{K}\subset C^\beta_b(\mathbb{R}^d)\to \mathbb{R}$ is a Lipschitz functional which has the $GCP$ with respect to $x$. Fix $\rho>0$. There is a constant $C(F,\rho)$ such that given $R>0$, $r\in (0,1)$, and $u,v \in \mathcal{K}_\rho$, then
  \begin{align*}
    |F(u)-F(v)| \leq C(F,\rho)r^{-\beta}\left ( \|u-v\|_{C^\beta(B_{R+r}(x))}+\|u-v\|_{L^\infty(\mathbb{R}^d\setminus B_{R}(x))} \right ).
  \end{align*}	  
\end{proposition}

\begin{remark}\label{remark:weak localization assumption versus proposition}
  It is worth comparing Proposition \ref{proposition:GCP implies weak localization} with Assumption \ref{assumption:tightness bound}. In the latter, one is interested in how $I(u,x)$ depends very little on the values of $u$  far away from $x$ (so, as $r\to \infty$), whereas the former deals with a weak version of this property that holds only for $r\in (0,1)$ but which follows alone from the GCP without the need for further assumptions on $F$.	
\end{remark}

\begin{proof}
  Take $\phi \in C^2_b(\mathbb{R}^d)$, such that $0\leq \phi \leq 1$ and $\phi(x)=0$. Then, for any $y$ we have
  \begin{align*}
    u(y) \leq w(y):= u(y)+ \phi(y) \left ( \|u-v\|_{L^\infty(\spt(\phi))}-(u(y)-v(y))\right ),
  \end{align*} 
  with the above being an equality for $y=x$. Now, let $\rho_0$ be chosen so that 
  \begin{align*}
    2\|\phi\|_{C^2(\mathbb{R}^d)}\rho_0 \leq \rho.
  \end{align*}
  Then, let us suppose that $u,v\in \mathcal{K}_{\rho}$ are such that $\|u-v\|_{C^\beta_b(\mathbb{R}^d)} \leq \rho_0$. In this case, we have $w\in \mathcal{K}$ since $u \in \mathcal{K}_\rho$ and in this case the GCP says that
  \begin{align*}
   F(u) \leq  F(w).
  \end{align*}
  Moreover, $F(w)\leq F(v)+L\|w-v\|_{C^\beta}$ and $w-v = (1-\phi) (u-v)+ \phi \|u-v\|_{L^\infty(\spt(\phi))}$, thus
  \begin{align*}
   F(u)-F(v) \leq L\|(1-\phi) (u-v)\|_{C^\beta}+L\|u-v\|_{L^\infty(\spt(\phi))}\|\phi\|_{C^\beta}.
  \end{align*}
  Consider the function $\phi(y) = \phi_{r,R}(y-x)$. Thanks to $r\in (0,1)$, the following estimates hold
  \begin{align*}
    \|\phi\|_{C^\beta} & \leq Cr^{-\beta},\\
    \|(1-\phi)(u-v)\|_{C^\beta} & \leq Cr^{-\beta} \|u-v\|_{C^\beta(B_{R+r})}.
  \end{align*}
  Substituting these in the inequality for $F(u)-F(v)$, the desired inequality follows when $\|u-v\|_{C^\beta}$ is no larger than $\rho_0$. Otherwise, $\|u-v\|_{C^\beta}\geq \rho_0$ and iterating the inequality in the previous case one obtains that
  \begin{align*}
    |F(u)-F(v)| \leq C(F,\rho)r^{-\beta}\left ( \|u-v\|_{C^\beta(B_{R+r}(x))}+\|u-v\|_{L^\infty(\mathbb{R}^d\setminus B_{R}(x))} \right ).
  \end{align*}	  
\end{proof}

\begin{lemma}\label{lemma:generalized gradients inherit GCP and weak localization}
  Let $F:\mathcal{K}\subset C^\beta_b(\mathbb{R}^d)\to\mathbb{R}$ be a Lipschitz functional which has the GCP with respect to $x$. Then, for every $\ell\in\partial F$ we have
  \begin{align*}
    \langle \ell,v\rangle\leq 0 \textnormal{ if } v\leq 0 \textnormal{ everywhere and } v(x)=0.	  
  \end{align*}	  
  In other words, if $F$ has the GCP with respect to $x$, then any $\ell$ arising as a generalized gradient of $F$ also has the GCP with respect to $x$. Furthermore, for any such $\ell$ and $r\in (0,1)$ we have
  \begin{align*}
    |\langle \ell,v\rangle | \leq C r^{-\beta}\left (  \|v\|_{C^\beta(B_r)} + \|v\|_{L^\infty(\mathbb{R}^d)} \right ).	  
  \end{align*}	  
\end{lemma}

\begin{proof}
  Let $u\in \mathcal{K}$, and let $v\in C^\beta_b(\mathbb{R}^d)$ be such that
  \begin{align*}
    v\leq 0\textnormal{ in } \mathbb{R}^d, v(x)=0.  
  \end{align*}
  Then, $u_t = u+tv$ touches $u$ from below at $x$ for each small $t$, therefore $F(u_t)\leq F(u)$ for every $t$, and
  \begin{align*}
    F^0(u,v) = \limsup\limits_{t\to 0} \frac{F(u+tv)-F(u)}{t} \leq 0.
  \end{align*}	  
  Since,
  \begin{align*}
    \max \limits_{\ell\in \partial F(u)} \langle \ell,v\rangle = F^0(u,v),	  
  \end{align*}
  it follows that $\langle \ell,v\rangle\leq 0$ for any $\ell\in \partial F(u)$, and the first part of the Lemma is proved. For the second part, one argues similarly, except that instead of invoking the GCP, one applies Proposition \ref{proposition:GCP implies weak localization} in order to pass the same estimate for any $\ell \in \partial F$.
\end{proof}

Fix a functional $\ell$ having the GCP with respect to $x$. Then, define $C_{\ell}$ by
\begin{align}\label{eqn:C sub ell definition}
  C_{\ell} := \langle \ell,1\rangle.
\end{align}
This associates a constant $C_{\ell}$ to any $\ell$ having the GCP. Likewise, we shall associate a vector $B_{\ell}$ and positive semi-definite matrix $A_{\ell}$. First, let us introduce some notation, 
\begin{align}\label{eqFunctionalsGCP:DefOfSet-S}
  \mathcal{S} := \{ \phi \in C^2_c(B_2(0))\;\mid\; \phi \equiv 1 \textnormal{ in a neighborhood of } 0,\; 0 \leq \phi \leq 1 \textnormal{ in all of } \mathbb{R}^d \}.
\end{align}
Given $\phi,\eta\in\mathcal{S}$, define the function
\begin{align}\label{eqFunctionalsGCP:DefOfTalyorP}
  P_{\phi,\eta,u,x}(\cdot) = \left \{ \begin{array}{rl} 
    u(x)+\phi(\cdot-x)(\nabla u(x),\cdot-x)+\tfrac{1}{2}\eta(\cdot-x)(D^2u(x)(\cdot-x),\cdot-x) & \textnormal{ if } \beta \in [2,3) ,\\
    u(x)+\phi(\cdot-x)(\nabla u(x),\cdot-x) & \textnormal{ if } \beta \in [1,2) ,\\
    u(x) & \textnormal{ if } \beta \in (0,1).	\end{array}\right.
\end{align}
For $x=0$ we will simply write $P_{\phi,\beta,u}$. Observe that, for example, if $\beta=2$ then $P_{\phi,\eta,u,x}$ is a smooth function which, in a neighborhood of $x$, coincides with the second order Taylor polynomial of the function $u$ at the point $x$.

\begin{definition}\label{definition:auxiliary A_ell,eta and B_ell,phi}
  Given any $\phi \in \mathcal{S}$ let $B_{\ell,\phi}$ be the vector defined by
  \begin{align*}
    (B_{\ell,\phi},e) = \langle \ell, \phi(\cdot)(\cdot,e)\rangle,\;\;\forall\; \textnormal{ vectors } e.
  \end{align*}
  At the same time, given $\eta \in \mathcal{S}$ let $A_{\ell,\eta}$ be the symmetric matrix defined by
  \begin{align*}
    \textnormal{tr}(A_{\ell,\eta}M) = \langle \ell, \eta(\cdot)\tfrac{1}{2}(M(\cdot),\cdot)\rangle,\;\;\forall\; \textnormal{ symmetric matrices } M.
  \end{align*}
\end{definition}

The following lemmas will characterize all of functionals having the GCP with respect $0$ (compare with Courrege's original proof \cite{Courrege-1965formePrincipeMaximum}, see also \cite{GuSc-2016MinMaxNonlocalarXiv}).

\begin{lemma}\label{lemma:preliminary Courrege theorem for a linear functional}
  Let $\ell:C^\beta_b(\mathbb{R}^d)\to \mathbb{R}$ be a bounded linear functional which has the GCP with respect to $0$, and $\phi,\eta \in \mathcal{S}$ (defined in (\ref{eqFunctionalsGCP:DefOfSet-S})). There is a positive measure $\mu_{\ell}$ on $\mathbb{R}^d\setminus \{0\}$ with 
  \begin{align*}
    \int_{\mathbb{R}^d\setminus\{0\} } \min\{1,|y|^{\beta}\}\;\mu_{\ell}(dy) \leq C\|\ell\|,
  \end{align*}	  
  such that for any $u \in C^\beta_b(\mathbb{R}^d)$ we have the following representation, 
  \begin{align*}
	  &\text{for}\ \beta\geq 2,\ \text{and}\ u\in C^\beta_b(\real^d)\intersect C^2(0),\\
    &\ \ \ \ \ \langle \ell,u\rangle = C_{\ell}u(0)+(B_{\ell,\phi},\nabla u(0))+\tr(A_{\ell,\eta}D^2u(0))+\int_{\mathbb{R}^d}u(y)-P_{\phi,\eta,u}(y)\;\mu_\ell(dy),\\
	&\text{for}\ \beta\in[1,2),\ \text{and}\ u\in C^\beta_b(\real^d)\intersect C^1(0),\\
    &\ \ \ \ \ \langle \ell,u\rangle  = C_{\ell}u(0)+(B_{\ell,\phi},\nabla u(0))+\int_{\mathbb{R}^d}u(y)-P_{\phi,\eta,u}(y)\;\mu_\ell(dy),\\
	&\text{for}\ \beta\in(0,1),\ \text{and}\ u\in C^\beta_b(\real^d),\\
    &\ \ \ \ \ \langle \ell,u\rangle  = C_{\ell}u(0)+\int_{\mathbb{R}^d}u(y)-u(0)\;\mu_\ell(dy).
  \end{align*}	 
  (The notation, $C^2(0)$ and $C^1(0)$, appears in Definition \ref{def:PointwiseC1C2}.) 
\end{lemma}

\begin{remark}
	We want to note that the dependence of $\mu$ only on $\ell$ is not a typo.  Even though the vector $B_{\ell,\phi}$ and matrix $A_{\ell,\eta}$ clearly depend on the functions $\phi$ and $\eta$, the reader can see in the proof in (\ref{eqFunctionalsGCP:DefOfMuEll}) that $\mu_\ell$ does not depend on $\phi$ or $\eta$. 
\end{remark}

\begin{proof}
  It suffices to prove the representation formula for $u \in C^2_b(\mathbb{R}^d)$ (even if $\beta \neq 2$), as it trivially extends to all of $C^\beta_b(\mathbb{R}^d)$ by approximation. We fix $u \in C^2_b(\mathbb{R}^d)\intersect C^2(0)$.  We recall $P_{\phi,\eta,u}$ is defined in (\ref{eqFunctionalsGCP:DefOfTalyorP}). Since $P_{\phi,\eta,u} \in C^\beta_b(\mathbb{R}^d)$ for each fixed $\phi,\eta$, we may write
  \begin{align*}
    u & = u-P_{\phi,\eta,u}+P_{\phi,\eta,u},
  \end{align*}
  and linearity gives
  \begin{align*}
    \langle \ell,u\rangle = \langle \ell, P_{\phi,\eta,u}\rangle+\langle \ell,u-P_{\phi,\eta,u}\rangle
  \end{align*}
  Let us study each of these two terms. Using the definition of $C_\ell,B_{\ell,\phi},$ and $A_{\ell,\eta}$, we have for $\beta\geq2$
  \begin{align*}
    \langle \ell,P_{\phi,\eta,u}\rangle & = u(0) \langle \ell,1\rangle + \sum \limits_{i=1}^d\partial_i u(0)\langle \ell, x_i \phi(x)\rangle +\frac{1}{2}\sum \limits_{i,j=1}^d \partial_{ij}^2u(0) \langle \ell, \eta(x)x_ix_j\rangle\\
      & = C_{\ell}u(0)+(B_{\ell,\phi},\nabla u(0))+\tfrac{1}{2}\textnormal{tr}(A_{\ell,\eta}D^2u(0)),
  \end{align*}
  as well as the corresponding expressions in the other cases when $\beta<2$. Next, we analyze the second term in the expression for $\langle \ell,u\rangle$ above, that is
  \begin{align*}   
    \langle \ell, u-P_{\phi,\eta,u}\rangle.
  \end{align*}
  First take the case $\beta \neq 1$. Given $w \in C^\beta_b(\mathbb{R}^d)$, define $\tilde w$ by
  \begin{align*}
    \tilde w(x) & := w(x) \frac{|x|^\beta}{1+|x|^\beta}.
  \end{align*}
  Observe that since $\beta \neq 1$, the function $\tilde 1 = |x|^\beta(1+|x|^\beta)^{-1}$ belongs to $C^\beta_b(\mathbb{R}^d)$. The linear transformation $w \mapsto \tilde w$ defines a linear functional $\tilde \ell$ via the relation
  \begin{align*}
    \langle \tilde \ell, w\rangle := \langle \ell, \tilde w \rangle.
  \end{align*}
  This clearly defines a bounded functional on $C^\beta_b(\mathbb{R}^d)$. In fact, however, this functional extends uniquely to a bounded functional in $C^0_b(\mathbb{R}^d)$: since $\tilde w$ is touched from above at $0$ by the function $\|w\|_{L^\infty}\tilde 1$, the GCP guarantees that 
  \begin{align*}
    |\langle \tilde \ell, w\rangle| \leq \|w\|_{L^\infty}\langle \ell,  \tfrac{|x|^\beta}{1+|x|^\beta} \rangle.
  \end{align*}
  This shows $\tilde \ell$ is a uniquely defined continuous functional on $C_b^0(\mathbb{R}^d)$ whose norm as a functional on $C_b^0(\mathbb{R}^d)$ is no larger than $\|\ell\| \| \tfrac{|x|^\beta}{1+|x|^\beta}\|_{C^\beta}$. It follows there is a measure $\tilde \mu$ such that
  \begin{align}\label{eqFunctionalsGCP:DefOfMuEll}
    \langle \tilde \ell, w\rangle = \int_{\mathbb{R}^d} w(y)\;\tilde \mu(dy).
  \end{align}
  Moreover, since $\langle \tilde \ell,w\rangle \geq 0$ whenever $w\geq 0$, $\tilde \mu(dy)$ is a non-negative measure. Now, since $u \in C^2_b(\mathbb{R}^d)$, we have that the function 
  \begin{align*}
    w(x) := (u(x)-P_{\phi,\eta,u}(x)) \frac{1+|x|^\beta}{|x|^\beta},
  \end{align*}
  remains continuous as $x \to 0$, so $w \in C^0_b(\mathbb{R}^d)$ and thus $\langle \tilde \ell,w\rangle$ is well defined. In this case, we have
  \begin{align*}
    \langle \ell, u-P_{\phi,\eta,u} \rangle = \langle \tilde \ell,w \rangle,
  \end{align*}
  and we obtain the formula
  \begin{align*}
    \langle \ell, u-P_{\phi,\eta,u}\rangle = \int_{\mathbb{R}^d}\left ( u(y)-P_{\phi,\eta,u}(y)\right ) \frac{1+|y|^\beta}{|y|^\beta}\;\tilde \mu(dy). 
  \end{align*}
  In particular, taking $\mu(dy) :=\tfrac{1+|y|^\beta}{|y|^\beta} \tilde \mu(dy)$, it follows that
  \begin{align*}
    \int_{\mathbb{R}^d\setminus\{0\}}\min\{1,|y|^\beta\}\mu(dy) \lesssim \|\ell\| \| \tfrac{|x|^\beta}{1+|x|^\beta}\|_{C^\beta}<\infty,
  \end{align*}
  and 
  \begin{align*}
    \langle \ell, u-P_{\phi,\eta,u}\rangle = \int_{\mathbb{R}^d\setminus\{0\} }u(y)-P_{\phi,\eta,u}(y)\;\mu(dy). 	  
  \end{align*}	  
  Revisiting the expression of $\ell$, we have when $\beta\geq2$
  \begin{align*}
    \langle \ell,u\rangle = C_{\ell}u(0)+(B_{\ell,\phi},\nabla u(0))+\tfrac{1}{2}\textnormal{tr}(A_{\ell,\eta}D^2u(0)) + \int_{\mathbb{R}^d\setminus\{0\} }u(y)-P_{\phi,\eta,u}(y)\;\mu(dy),
  \end{align*}  
  and the analogous formulas follow for the other cases where $\beta \neq 1$, per the change in definition of the function $P_{\phi,\eta,u}$ in (\ref{eqFunctionalsGCP:DefOfTalyorP}). It remains to consider the case $\beta = 1$.
  
  Since $|x|$ is not a $C^1$ function, we are going to approximate it by a more regular function. For every small $\varepsilon>0$ we repeat the argument above with $\beta = 1+\varepsilon$ and conclude that for some $\mu_\varepsilon$ we have the formula
  \begin{align*}
    \langle \ell,u\rangle = C_{\ell}u(0)+(B_{\ell,\phi},\nabla u(0))+ \int_{\mathbb{R}^d\setminus\{0\} }u(y)-P_{\phi,\eta,u}(y)\;\mu_{\varepsilon}(dy),
  \end{align*}  
  and this measure $\mu_\varepsilon$ is positive and satisfies the bound
  \begin{align*}
    \int_{\mathbb{R}^d\setminus\{0\}}\min\{1,|y|^\beta\}\mu(dy) \lesssim \|\ell\| \| \tfrac{|x|^{1+\varepsilon}}{1+|x|^{1+\varepsilon}}\|_{C^1}.
  \end{align*}
  Since
  \begin{align*}
    \sup \limits_{\varepsilon \in (0,1)} \| \tfrac{|x|^{1+\varepsilon}}{1+|x|^{1+\varepsilon}}\|_{C^1} < \infty,
  \end{align*}
  it follows that the respective finite measures $\{\tilde \mu_{\varepsilon} \}_{\varepsilon \in (0,1)}$ have uniformly bounded mass. Therefore, it is not difficult to show (using $\ell$ to get tightness for the $\tilde \mu_\varepsilon$) that along a subsequence $\varepsilon \to 0$ we can find a limit $\tilde \mu$, and if we let $\mu:= (1+|y|)|y|^{-1}\tilde \mu$ then
  \begin{align*}
    \int_{\mathbb{R}^d\setminus \{0\}}\min\{1,|y|\}\mu(dy)<\infty, 
  \end{align*}	  
  and again, for any $u\in C^2_b(\mathbb{R}^d)$,
  \begin{align*}
    \langle \ell,u\rangle = C_{\ell}u(0)+(B_{\ell,\phi},\nabla u(0))+ \int_{\mathbb{R}^d\setminus\{0\} }u(y)-P_{\phi,\eta,u}(y)\;\mu(dy),
  \end{align*}  
  
\end{proof}

We consider the following special functions. For $\delta>0$, define (see \eqref{equation:psi sub r R} for definition of $\psi_{r,R}$)
\begin{align}
  \phi_\delta(x) & :=	\psi_{\delta,1-2\delta}, \label{equation:auxiliary phi_delta}\\
  \eta_\delta(x) & := \psi_{\delta,\delta}(x).\label{equation:auxiliary eta_delta}
\end{align}
Note that $\phi_\delta \equiv 1$ inside $B_{1-2\delta}$ and $\phi_\delta \equiv 0$ outside $B_{1-\delta}$, while $\eta_\delta \equiv 1$ inside $B_\delta$ and $\eta_\delta \equiv 0$ outside $B_{2\delta}$. Furthermore, we note that $\delta \leq \delta'$ implies that $\eta_{\delta}\leq \eta_{\delta'}$.

\begin{lemma}\label{lemma:A_ell and B_ell existence}
  Assume that $\beta\in[0,3)$, $l:C^{\beta}_b(\real^d)\to\real$ is a bounded linear functional with the GCP with respect to $0$, and that $A_{\ell,\eta}$, $B_{\ell,\phi}$ are as in Definition \ref{definition:auxiliary A_ell,eta and B_ell,phi}.  Taking $\eta_\delta$ as in \eqref{equation:auxiliary eta_delta}, the limit
  \begin{align*}  	
    A_{\ell} & := \lim\limits_{\delta \searrow 0} A_{\ell,\eta_\delta},
  \end{align*}   
  exists for all $\beta \in [0,3)$, and $A_\ell\equiv0$ if $\beta<2$. Moreover, if $\phi_\delta$ is as in \eqref{equation:auxiliary phi_delta}, there is a sequence $\delta_k\searrow 0$ such that the following limit exists
  \begin{align*}
    B_{\ell} := \lim\limits_{k\to \infty} B_{\ell,\phi_{\delta_k}}.	
  \end{align*}
\end{lemma}

\begin{proof}
  Let $\eta_1,\eta_2 \in \mathcal{S}$ and such that $\eta_1\leq \eta_2$. Then for any positive semi-definite $M$ we have
  \begin{align*}
    \tfrac{1}{2}\eta_1(x)(Mx,x)\leq \tfrac{1}{2}\eta_2(x)(Mx,x),\;\textnormal{ with equality at } x=0.
  \end{align*}
  Since $\ell$ has the GCP with respect to $0$, it follows that
  \begin{align*}
    \langle \ell, \tfrac{1}{2}\eta_1(x)(Mx,x)\rangle \leq \langle \ell,\tfrac{1}{2}\eta_2(x)(Mx,x)\rangle.
  \end{align*}
  From this monotonicity and the elementary inequality $|\langle \ell, \tfrac{1}{2}\eta(x)(Mx,x)\rangle| \leq C|M|\max_{ij}\|\eta x_ix_j \|_{C^\beta}$ we conclude that the following limit exists for every positive semi-definite $M$
  \begin{align*}
    \lim \limits_{\delta \searrow 0} \langle \ell, \tfrac{1}{2}\eta_{\delta}(x)(Mx,x)\rangle.
  \end{align*}	  
  At the same time, when $\beta<2$ we have $\|\eta_{\delta}x_ix_j\|_{C^\beta} \to 0$ as $\delta \searrow 0$ for all $i,j$, so in this case the limit is zero. Now, given a symmetric matrix $M$, write $M = M^+-M^-$, where both $M^+$ and $M^-$ are positive semi-definite. Then, we also have that the limit
  \begin{align*}
    \lim \limits_{\phi\in\mathcal{S},\eta\searrow 0} \langle \ell, \tfrac{1}{2}\eta(x)(Mx,x)\rangle
  \end{align*}
  exists for any symmetric matrix $M$. It is clear then that this limit is linear as a function of $M$, and therefore, there is a unique symmetric matrix $A_{\ell}$ such that 
  \begin{align}\label{eqn:A sub ell definition}
    \tr(A_{\ell} M) = \lim\limits_{\eta \searrow 0}\langle \ell, \frac{1}{2}\eta(x)(Mx,x)\rangle.
  \end{align}
  Moreover, this matrix $A_{\ell}$ is positive semi-definite and $A_{\ell,\eta_\delta} \to A_{\ell}$ as $\delta \searrow 0$, and $A_{\ell} = 0$ when $\beta<2$. It remains to analyze the limit of $B_{\ell,\phi_{\delta}}$ along a subsequence. For every $\delta \in (0,1)$
  \begin{align*}
    (B_{\phi_\delta})_i = \langle \ell,\phi_\delta x_i\rangle.
  \end{align*}
  Now, recall the estimate from Lemma \ref{lemma:generalized gradients inherit GCP and weak localization}, which implies
  \begin{align*}
    |\langle \ell,\phi_\delta x_i\rangle| & \leq C(\|\phi_\delta x_i\|_{C^\beta(B_{1/2})}+\|\phi_\delta x_i\|_{L^\infty(\mathbb{R}^d)}).
  \end{align*}	
  A direct computation shows that
  \begin{align*}
    \sup\limits_{0<\delta<1}\|\phi_\delta x_i\|_{C^\beta(B_{1/2})}<\infty.
  \end{align*}
  It follows that
  \begin{align*}
    \sup \limits_{0<\delta<1} |B_{\phi_\delta}| < \infty,
  \end{align*}	 
  and by compactness, there must be a subsequence $\delta_k \to 0$ for which $\{B_{\ell,\phi_{\delta_k}}\}_k$ converges.

\end{proof}

\begin{lemma}\label{lemma:Courrege theorem for a linear functional}
  Assume that $\beta\in[0,3)$. Let $\ell:C^\beta_b(\mathbb{R}^d)\to \mathbb{R}$ be a bounded linear functional which has the GCP with respect to $0$. For $\beta\geq 2$ and any $u \in C^\beta_b(\mathbb{R}^d)\intersect C^2(0)$, we have the representation
  \begin{align*}
    \langle \ell,u\rangle & = C_{\ell}u(0)+(B_{\ell},\nabla u(0))+\tr(A_{\ell}D^2u(0))+\int_{\mathbb{R}^d}u(y)-u(0)-\chi_{B_1(0)}(\nabla u(0),y)\;\mu_\ell(dy).
  \end{align*}	
  This representation is unique. This means that if there were $\tilde C$, $\tilde B$, $\tilde A$ and $\tilde \mu$ a measure in $\mathbb{R}^d\setminus \{0\}$ all such that
  \begin{align*}
    \langle \ell,u\rangle & = \tilde Cu(0)+(\tilde B,\nabla u(0))+\tr(\tilde AD^2u(0))+\int_{\mathbb{R}^d}u(y)-u(0)-\chi_{B_1(0)}(\nabla u(0),y)\;\tilde \mu(dy).
   \end{align*}
   for all $u$, then $\tilde C = C_\ell$, $\tilde B = B_\ell$, $\tilde A = A_\ell$, and $\tilde \mu = \mu_\ell$. Furthermore, if $\beta<2$ and $u\in C^\beta(\real^d)\intersect C^1(0)$, then $A_{\ell}=0$, and if $\beta<1$, then $B_{\ell} = 0$ and the integrand on the right can be replaced with just $u(y)-u(0)$.
 
\end{lemma}

\begin{proof}  
  Let $\delta,\delta'\in(0,1)$. Applying Lemma \ref{lemma:preliminary Courrege theorem for a linear functional} with the functions $\phi_{\delta}$ and $\eta_{\delta'}$,
  \begin{align*}
    \langle \ell,u\rangle & = C_{\ell}u(0)+(B_{\ell,\phi_{\delta}},\nabla u(0))+\tr(A_{\ell,\eta_{\delta'}}D^2u(0))+\int_{\mathbb{R}^d}u(y)-P_{\phi_{\delta},\eta_{\delta'},u}(y)\;\mu_\ell(dy).
  \end{align*}	  
  Since $\min\{1,|y|^\beta\}$ is integrable against $\mu_\ell$, it follows that  
  \begin{align*}
    \lim \limits_{\delta' \searrow 0}\int_{\mathbb{R}^d\setminus\{0\}} \eta_{\delta'}(y)(D^2u(0)y,y)\;\mu_\ell(dy) = 0.
  \end{align*}
  Therefore, 
  \begin{align*}
    \lim \limits_{\delta' \searrow 0}\int_{\mathbb{R}^d\setminus\{0\}} u(y)-P_{\phi_{\delta},\eta_{\delta'},u}(y)\;\mu_\ell(dy) = \int_{\mathbb{R}^d\setminus\{0\}} u(y)-u(0)-\phi_{\delta}(y)(\nabla u(0),y)\;\mu_\ell(dy).
  \end{align*}
  Then, thanks to Lemma \ref{lemma:A_ell and B_ell existence}, the formula for $\langle \ell,u\rangle$ becomes (for every fixed $\delta\in (0,1)$)
  \begin{align*}
    \langle \ell,u\rangle = C_{\ell}u(0)+(B_{\ell,\phi_\delta},\nabla u(0))+\tr(A_{\ell}D^2u(0))+\int_{\mathbb{R}^d\setminus \{0\}} u(y)-u(0)-\phi(y)(\nabla u(0),y)\;\mu_\ell(dy). 
  \end{align*}
  Now, let $\delta_k\searrow 0$ be chosen so that $B_{\ell \phi_{\delta_k}}\to B_{\ell}$ (which can be done thanks to Lemma \ref{lemma:A_ell and B_ell existence}). From the definition of $\phi_{\delta}$, we have that
  \begin{align*}
    u(y)-u(0)-\phi_{\delta_{k}}(y)(\nabla u(0),y) \textnormal{ is monotone in } k.
  \end{align*}
  At the same time, for every $y\in \mathbb{R}^d$ we have
  \begin{align*}
    \lim \limits_{k\to\infty} \phi_{\delta_k}(y) = \chi_{B_1(0)}.
  \end{align*}
  Therefore, by  monotone convergence we have
  \begin{align*}
    \lim\limits_{k\to \infty} \int_{\mathbb{R}^d\setminus \{0\} } u(y)-u(0)-\phi_{\delta_k}(y)(\nabla u(0),y) \;\mu_\ell(dy) = \int_{\mathbb{R}^d\setminus \{0\} } u-u(0)-\chi_{B_1}(y)(\nabla u(0),y) \;\mu_\ell(dy). 	
  \end{align*}
  From where it follows that
  \begin{align*}
    \langle \ell,u\rangle = C_{\ell}u(0)+(B_{\ell},\nabla u(0))+\tr(A_{\ell}D^2u(0))+\int_{\mathbb{R}^d\setminus \{0\}} u(y)-u(0)-\chi_{B_1(0)}(y)(\nabla u(0),y)\;\mu_\ell(dy),
  \end{align*}
  as claimed. It remains to prove the uniqueness part. For this, it is enough to show that if for all $u$ we have $\langle \ell,u \rangle = 0$ and 
  \begin{align*}
    \langle \ell,u\rangle & = C_{\ell}u(0)+(B_{\ell},\nabla u(0))+\tr(A_{\ell}D^2u(0))+\int_{\mathbb{R}^d}u(y)-u(0)-\chi_{B_1(0)}(\nabla u(0),y)\;\mu_\ell(dy),
  \end{align*}	
  then $C_\ell = 0, B_\ell = 0, A_\ell = 0$ \and $\mu_\ell = 0$. First, consider any $u$ with compact support which is disjoint from $\{0\}$, for such a $u$ we have
  \begin{align*}
    \langle \ell,u\rangle & = \int_{\mathbb{R}^d}u(y)\;\mu_\ell(dy),
  \end{align*}	
  Since $u$ can be any function with compact support in $\mathbb{R}^d\setminus \{0\}$, it follows that $\mu_\ell = 0$. Evaluating $\ell$ at the function $u(x)\equiv 1$ we obtain $C_\ell = 0$. Lastly, evaluating $\ell$ at all of the functions of the form $(x,e)$, $e\in\mathbb{R}^d$ and $(Mx,x)$, $M$ symmetric matrix, we see that $B_\ell \cdot e= 0$ for any vector $e$ and $\tr(AM) = 0$  for any symmetric matrix $M$, so that $B_\ell = 0$ and $A_\ell = 0$.

\end{proof}

By a simple change of variables, Lemma \ref{lemma:Courrege theorem for a linear functional} implies the following.

\begin{corollary}\label{corollary:Courrege theorem for a linear functional, general base point}
  Assume that $x$ is fixed, $\beta\in[0,3)$, and let $\ell:C^\beta_b(\mathbb{R}^d)\to \mathbb{R}$ be a bounded linear functional which has the GCP with respect to $x$. For $\beta\geq 2$ any $u \in C^\beta_b(\mathbb{R}^d)\intersect C^2(x)$ we have the representation
  \begin{align*}
    \langle \ell,u\rangle & = C_{\ell}u(x)+(B_{\ell},\nabla u(x))+\tr(A_{\ell}D^2u(x))+\int_{\mathbb{R}^d}u(x+y)-u(x)-\chi_{B_1(0)}(\nabla u(x),y)\;\mu_\ell(dy).
  \end{align*}	
  As before, this representation is unique, and when $\beta<2$ and $u\in C^\beta_b(\real^d)\intersect C^1(x)$, we have $A_{\ell} = 0$, while for $\beta<1$ we have $B_{\ell} = 0$ and the integrand can be replaced with just $u(x+y)-u(x)$.
\end{corollary}

%%%%%%%%%%%%%%%%%%%%%%%%%%%%%%%%%
%%%%%%%%%%%%%%%%%%%%%%%%%%%%%%%%%
%%%%%%%%%%%%%%%%%%%%%%%%%%%%%%%%%

\subsection{Proofs of Theorems \ref{theorem:MinMax Euclidean} and \ref{theorem:MinMax Translation Invariant}}\label{sec:TheoremsWithoutWhitney}

With Lemmas \ref{lemma:generalized gradients inherit GCP and weak localization} and \ref{lemma:Courrege theorem for a linear functional} and Corollary \ref{corollary:Courrege theorem for a linear functional, general base point} in hand, we can now prove Theorems \ref{theorem:MinMax Translation Invariant} and \ref{theorem:MinMax Euclidean}.

\begin{proof}[Proof of Theorem \ref{theorem:MinMax Translation Invariant}]
  Consider the functional,
  \begin{align*}
    F(u) := I(u,0).
  \end{align*}
  Now, by Theorem \ref{theorem:MinMax for scalar functionals}, we have that
  \begin{align*}
    F(u) = \min_a\max_b \{ f_{ab}+\langle \ell_{ab},u\rangle \}. 	  
  \end{align*}	  
  By Lemma \ref{lemma:generalized gradients inherit GCP and weak localization}, each $\ell_{ab}$ is a linear operator having the GCP with respect to $0$, in which case Lemma \ref{lemma:Courrege theorem for a linear functional} says that for $u\in C^\beta_b(\real^d)\intersect C^2(0)$,
  \begin{align*}
    \langle \ell_{ab},u\rangle & = \tr(A_{ab}D^2u(0))+B_{ab}\cdot \nabla u(0)+C_{ab}u(0)+\int_{\mathbb{R}^d}u(y)-u(0)-\chi_{B_1}(0)(\nabla u(0),y)\;\mu_{ab}(dy).  
  \end{align*}	  The translation invariance of $I$ boils down to the identity
  \begin{align*}
    I(u,x) = F(\tau_x u).
  \end{align*}
  Therefore,
  \begin{align*}
    I(u,x) = \min\limits_{a} \max\limits_{b} \{ f_{ab} + \langle \ell_{ab},\tau_x u\rangle\}
  \end{align*}
  However, $\langle \ell_{ab},\tau_x u\rangle$ has a simple expression, namely
  \begin{align*}
    \tr(A_{ab}D^2u(x))+B_{ab}\cdot \nabla u(x)+C_{ab}u(x)+\int_{\mathbb{R}^d}u(x+y)-u(x)-\ind_{B_1(0)}\nabla u(x)\cdot y\;\mu_{ab}(dy),  	  
  \end{align*}	  
  and this proves the theorem.
\end{proof}

\begin{proof}[Proof of Theorem \ref{theorem:MinMax Euclidean}]

  The beginning of the proof is similar to that of the previous one. For each $x\in\mathbb{R}^d$, define a functional
  \begin{align*}
    F_x(u) := I(u,x),\;\;\forall\; u\in C^\beta_b(\mathbb{R}^d).  
  \end{align*}
  Applying Theorem \ref{theorem:MinMax for scalar functionals}, it follows that
  \begin{align*}
    F_x(u) := \min\limits_{v\in C^\beta_b(\mathbb{R}^d)} \max \limits_{ \ell \in \partial F_x}  \{ F_x(v)+\langle \ell,u-v\rangle \}.
  \end{align*}
  Applying Lemma \ref{lemma:generalized gradients inherit GCP and weak localization}, it follows that for any $\ell \in \partial F_x $
  \begin{align*}
    & \langle \ell,u\rangle = Cu(x)+(B,\nabla u(x))+\tr(AD^2u(x))+\int_{\mathbb{R}^d} u(x+y)-u(x)-\chi_{B_1(0)}(\nabla u(x),y)\;\mu(dy).	  		  
  \end{align*}	  
  Since $F_x(v) = I(v,x)$ this proves the Theorem, with $\mathcal{K}(I)_x = \{ L \mid L(u) = \langle \ell,u\rangle \textnormal{ for } \ell \in \partial F_x\}$ .
  
\end{proof}

\begin{remark}
  It is worthwhile to compare the proof of Theorem \ref{theorem:MinMax Euclidean} above to the much longer and complicated one given in \cite{GuSc-2016MinMaxNonlocalarXiv}. The simplicity here is made possible by the use of a mean value theorem for Lipschitz functionals (Theorem \ref{theorem:Lebourg}) in the infinite dimensional setting, which suffices to prove Theorem \ref{theorem:MinMax Euclidean} as it involves a min-max formula in terms of linear functionals in $C^2_b$ and not linear operators from $C^2_b(\mathbb{R}^d)$ to $C^0_b(\mathbb{R}^d)$. The more complicated method from \cite{GuSc-2016MinMaxNonlocalarXiv} is however still of value, specially if one is interested in obtaining a min-max representation in terms of a family of linear operators from $C^2_b$ to $C_b^0$.  Moreover, it is by adapting the method from \cite{GuSc-2016MinMaxNonlocalarXiv} that we are able to prove Theorem \ref{theorem:MinMax Euclidean ver2}, after analyzing the spatial properties of the finite dimensional approximations (see in Section \ref{section:Analysis of finite dimensional approximations}).

\end{remark}

%%%%%%%%%%%%%%%%%%%%%%%%%%%%%%%%%
%%%%%%%%%%%%%%%%%%%%%%%%%%%%%%%%%
%%%%%%%%%%%%%%%%%%%%%%%%%%%%%%%%%
%%%%%%%%%%%%%%%%%%%%%%%%%%%%%%%%%
%%%%%%%%%%%%%%%%%%%%%%%%%%%%%%%%%
%%%%%%%%%%%%%%%%%%%%%%%%%%%%%%%%%
\section{Finite Dimensional Approximations to $C^\beta_b(\mathbb{R}^d)$}\label{section:Finite Dimensional Approximations}

%%%%%%%%%%%%%%%%%%%%%%%%%%%%%%%%%
%%%%%%%%%%%%%%%%%%%%%%%%%%%%%%%%%
\subsection{Graph approximations}

The following nested family of sets will be important in what follows
\begin{align*}
  G_n & := 2^{-n} \mathbb{Z}^d.	
\end{align*}	
It will be convenient to write $h_n := 2^{-n}$. Then, $h_n$ represents the maximum possible distance between $x\in \mathbb{R}^d$ and $G_n$, and in particular $\textnormal{dist}(x,G_n)\leq h_n$ for all $x \in \mathbb{R}^d$. Observe that
\begin{align*}
  G_1 \subset G_2 \subset G_3 \ldots, 
\end{align*}
and note also the union of the sets $G_n$ is dense in $\mathbb{R}^d$.

\begin{definition}\label{definition:discrete function spaces}
  We consider the following function spaces
  \begin{align*}
    C(G_n) & := \{ u:G_n \to \mathbb{R}^d\},\\
    C_*(G_n) & := \{ u \in C(G_n) \mid u(x) = 0 \textnormal{ if } x\not\in  [-2^n,2^n]^d\}.
  \end{align*}
  These spaces will be related to $C^\beta_b(\mathbb{R}^d)$ by restriction, which we think of as a map denoted by $T_n$ and given by 
  \begin{align*}
    T_n:C^\beta_b(\mathbb{R}^d)\to C(G_n),\;\; T_nu := u_{\mid G_n}.	  
  \end{align*}	  
\end{definition}

\begin{remark}\label{remark:C star G_n is finite dimensional}
  The space $C_*(G_n)$ is a finite dimensional vector space.

\end{remark}

%%%%%%%%%%%%%%%%%%%%%%%%%%%%%%%%%
%%%%%%%%%%%%%%%%%%%%%%%%%%%%%%%%%
\subsection{Cube decomposition and partition of unity}

In this section we shall apply the Whitney theory to extend functions in a grid $r \mathbb{Z}^d$ to all of $\mathbb{R}^d$. Since it is in our interest for the Whitney construction to be compatible with the grid structure, we shall do the usual cube decomposition making sure the resulting family of cubes is invariant under translations by vectors in $r \mathbb{Z}^d$, the resulting construction is illustrated in Figure \ref{figure:periodic cube decomposition}.

\begin{lemma}\label{lemma:cube decomposition of the complement of the lattice}
  For every $r>0$, there exists a collection of cubes $\{Q_k\}_k$ such that
  \begin{enumerate}
    \item The cubes $\{Q_k\}_k$ have pairwise disjoint interiors.
    \item The cubes $\{Q_k\}_k$ cover $\mathbb{R}^d \setminus r\mathbb{Z}^d$	
    \item $c_1 \textnormal{diam}(Q_k)\leq \textnormal{dist}(Q_k,\mathbb{Z}^d) \leq c_2 \textnormal{diam}(Q_k).$
    \item For every $h \in r\mathbb{Z}^d$, there is a bijection $\sigma_h:\mathbb{N}\to\mathbb{N}$ such that $Q_k+h = Q_{\sigma_h k}$ for every $k\in\mathbb{N}$. 
  \end{enumerate}
\end{lemma}

  \begin{center}
    \includegraphics[scale=0.5]{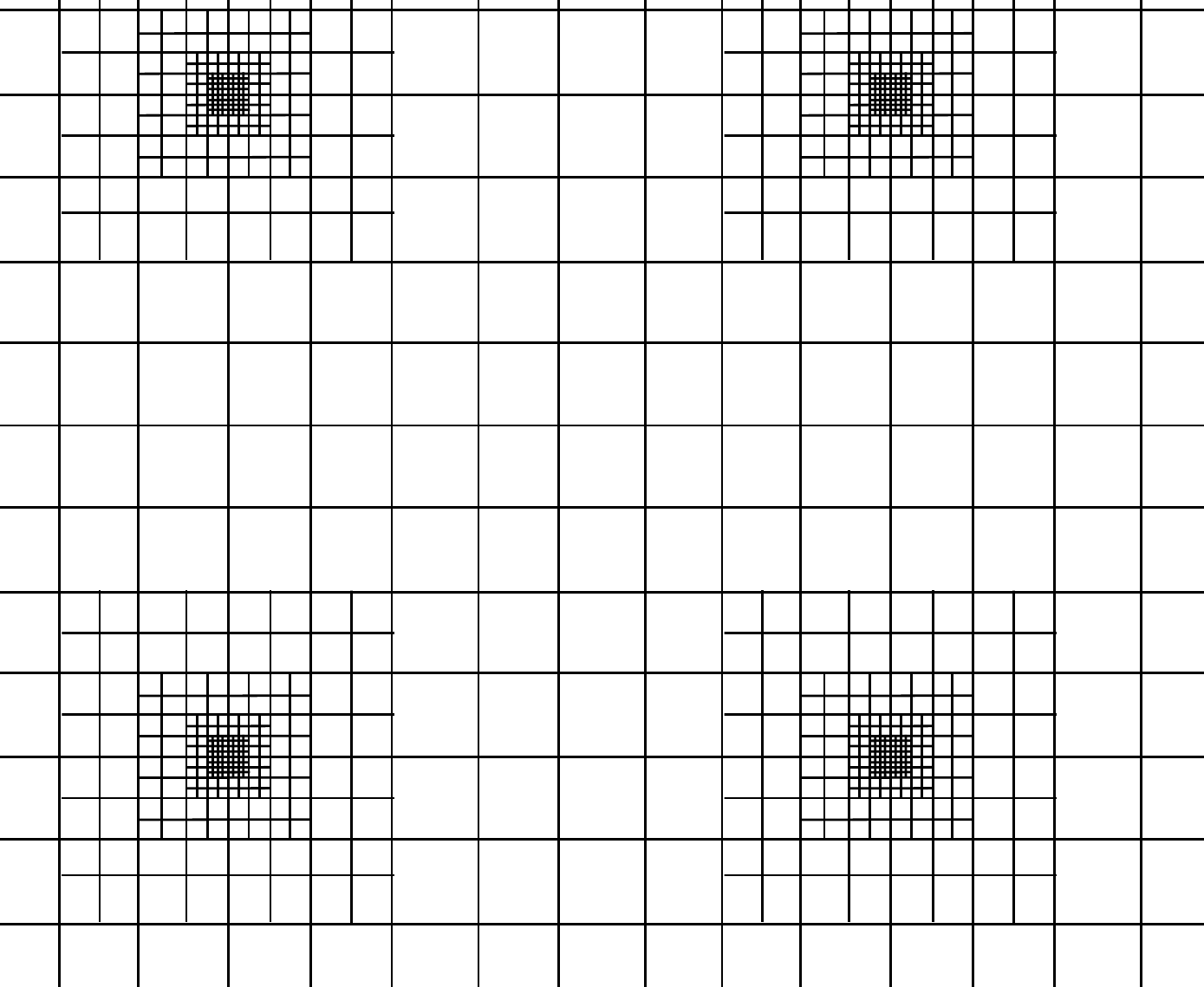}
  \captionof{figure}{A (periodic) cube decomposition of $\mathbb{R}^d\setminus \mathbb{Z}^d$}\label{figure:periodic cube decomposition}   
  \end{center}

\begin{proof}
  We consider the case $r=1$, once the collection of cubes is $\{Q_k\}_k$ obtained in this case, the general case follows via scaling by taking the family $\{rQ_k\}_k$ . 
  
  Consider the cube $Q_0 = [-1/2,1/2]^d$, let $\mathcal{M}_0$ denote the family of $2^d$ equal size cubes obtained from $Q_0$ by bisecting each of its sides. Let $\mathcal{M}_k$ denote the family of cubes obtained from applying this same procedure to each of the cubes in $\mathcal{M}_{k-1}$. Note that the side length of each cube in $\mathcal{M}_k$ is just $2^{-k}$. Now, we construct a family $\mathcal{F}_0$ as follows, with $R_k := \{ 2\sqrt{d}2^{-k}\leq |x|\leq 2\sqrt{d}2^{-(k-1)}\}$ for each $k\in \mathbb{N}$, then
  \begin{align*}
    \mathcal{F}_0 := \bigcup \limits_k \{ Q \in \mathcal{M}_k \; : \; Q\cap R_k \neq \emptyset \}.
  \end{align*}
  Observe that if $Q \in \mathcal{F}_0$ then $Q\in\mathcal{M}_k$ for some $k$ and there is some $x\in Q$ such that $2\sqrt{d}2^{-k}\leq |x|$ and $|x|\leq 2\sqrt{d} 2^{-(k-1)}$. This means, 
  \begin{align*}
    \sqrt{d}2^{-k}= 2\sqrt{d}2^{-k}-\textnormal{diam}(Q) \leq \textnormal{dist}(Q,0) \leq 2\sqrt{d}2^{-k},
  \end{align*}
  and since $\textnormal{diam}(Q) = \sqrt{d}2^{-k}$, we conclude that
  \begin{align*}
    \textnormal{diam}(Q) \leq \textnormal{dist}(Q,0) \leq 4\textnormal{diam}(Q) \;\;\forall\;Q\in\mathcal{F}_0.
  \end{align*}
  On the other hand, we have that
  \begin{align*}
    \bigcup \limits_{Q \in \mathcal{F}_0} Q = [-1/2,1/2]^d \setminus \{0\}.  
  \end{align*}
  
  If $\mathcal{F}$ denotes the subfamily of maximal cubes in $\mathcal{F}_0$, it follows that: the union of these cubes is still $[-1/2,1/2]^d \setminus \{0\}$, the inequality $\textnormal{diam}(Q) \leq \textnormal{dist}(Q,0) \leq 4\textnormal{diam}(Q)$ holds for each $Q\in\mathcal{F}$, and the cubes have pairwise disjoint interiors. 
  	   
  Denote by $\{Q_k\}_k$ an enumeration of the family of cubes of the form $Q+z$, where $Q\in \mathcal{F}$ and $z \in \mathbb{Z}^d$. It is clear that $\{Q_k\}_k$ covers all of $\mathbb{R}^d\setminus \mathbb{Z}^d$ and that these cubes have pairwise disjoint interiors. Furthermore, for any $h\in \mathbb{Z}^d$ the map $Q \to Q+h$ gives a bijection of the set $\{Q_k\}_k$ onto itself, therefore one can represent it via a bijection $\sigma_h:\mathbb{N}\to\mathbb{N}$ so that $Q_k +h = Q_{\sigma_h k}$. Last but not least, as each cube of the form $Q+z$ is closest to $z$ than to any other point in $\mathbb{Z}^d$, property (3) follows from the respectively inequality for the family $\mathcal{F}$.

\end{proof}

\begin{remark}\label{remark:cube centers and hat ynk}
  We apply Lemma \ref{lemma:cube decomposition of the complement of the lattice} with $r=2^{-n}$, for some $n\in\mathbb{N}$, and for the rest of the section shall refer to the resulting cubes as $\{Q_{n,k}\}_k$.
	
  Furthermore, for every $n$ and $k$, we will denote the center of $Q_{n,k}$ by $y_{n,k}$, and for each $n$ and $k$ we will denote by $\hat y_{n,k}$ the unique point in $G_n$ such that 
  \begin{align*}
    \textnormal{dist}(y_{n,k},G_n) = |y_{n,k}-\hat y_{n,k}|,
  \end{align*}	  
  (note that there is only one since by construction not a single center $y_{n,k}$ lies at equidistance to two different lattice points). 
  
  In particular, for each of the bijections $\sigma_h : \mathbb{N}\to \mathbb{N}$ from Lemma \ref{lemma:cube decomposition of the complement of the lattice} we have
  \begin{align*}
    y_{n,k}+h  = y_{n,\sigma_h k}, \; \hat y_{n,k} +h = \hat y_{n,\sigma_h k},\;\forall\;n,k.	
  \end{align*}
 
\end{remark}

\begin{remark}\label{remark:maximum number of overlapping cubes}
  In all what follows, given a cube $Q$, we shall denote by $Q^*$ the cube with same center as $Q$ but whose sides are increased by a factor of $9/8$. Observe that for every $n$ and $k$, we have $Q_{n,k}^* \subset \mathbb{R}^d\setminus 2^{2-n}\mathbb{Z}^d$, and that any given $x$ lies in at most some number $C(d)$ of the cubes $Q_k^*$.
\end{remark}

\begin{proposition}\label{proposition:partition of unity properties}
  For every $n$, there is a family of functions $\phi_{n,k}(x)$ such that
  \begin{enumerate}
    \item $0\leq \phi_{n,k}(x)\leq 1$ for every $k$ and $\phi_{n,k} \equiv 0$ outside $Q_{n,k}^*$ (using the notation in Remark \ref{remark:maximum number of overlapping cubes})
    \item $\sum_k \phi_{n,k}(x) =1 $ for every $x \in \mathbb{R}^d\setminus G_n$.
    \item There is a constant $C$, independent of $n$ and $k$, such that
    \begin{align*}
      |\nabla^{i}\phi_{n,k}(x)| \leq \frac{C}{\diam(Q_{n,k})^i}.
    \end{align*}		
    \item For every $z \in G_n$, we have
    \begin{align*}
      \phi_{n,k}(x-z) = \phi_{n,\sigma_zk}(x),\;\;\forall\;k,\;x,
    \end{align*}		
    where $\sigma_z$ are the bijections introduced above. 
  \end{enumerate}

\end{proposition}

\begin{proof}
  Fix a $C^\infty$ function $\phi$ such that
  \begin{align*}
    & 0\leq \phi \leq 1,\\ 
    & \phi \equiv 1 \textnormal{ in } Q_0 = [-1/2,1/2]^d,\\
    & \phi \equiv 0 \textnormal{ outside } Q_0^*.
  \end{align*}
  Let $\ell(Q)$ denote the common length for the sides of $Q_{n,k}$, and with $y_{n,k}$ as given in Remark \ref{remark:cube centers and hat ynk} we define	
  \begin{align*}
    \tilde \phi_{n,k} := \phi \left ( \frac{x-y_{n,k}}{\ell_{n,k}}\right ).
  \end{align*}
  Consider the function 
  \begin{align*}
    \Phi(x) = \sum \limits_{k} \tilde \phi_{n,k}(x).
  \end{align*}
  It follows from Remark \ref{remark:maximum number of overlapping cubes} that given any $x$ ,at most $C(d)$ of the terms appearing in the sum are non-zero in a neighborhood of $x$, and therefore $\Phi$ is a smooth function. Then, define
  \begin{align*}
    \phi_{n,k}(x) := \tilde \phi_{n,k}(x) \Phi(x)^{-1}.
  \end{align*}
  It is clear that the functions $\{\phi_{n,k}\}_{k}$ satisfy properties (1) and (2). Property (3) follows easily from the chain rule, using the differentiability of the function $\phi$. It remains to check property (4), let $z \in G_n$, then 
  \begin{align*}
    \phi_{n,k}(x-z) & = \phi \left ( \frac{x-(y_{n,k}+z)}{\ell(Q_{n,k})}\right )\Phi(x-z)^{-1}\\
	  & = \phi \left ( \frac{x-y_{n,\sigma_z k}}{\ell(Q_{n,\sigma_z k})}\right ) \Phi(x)^{-1} = \phi_{n,\sigma_z k}(x),
  \end{align*}
  where we used that $\ell(Q_{n,k}) = \ell(Q_{n,\sigma_z k})$, which follows clearly from the definition of $\sigma_z$.
\end{proof}

%%%%%%%%%%%%%%%%%%%%%%%%%%%%%%%%%
%%%%%%%%%%%%%%%%%%%%%%%%%%%%%%%%%
\subsection{Discrete derivatives}

In what follows, it will be in our interest to approximate the first and second derivatives of a function $u \in C^\beta_b(\mathbb{R}^d)$ (see \eqref{equation:Cbeta definition} for our convention regarding the meaning of $C^\beta_b$) at a point $x\in G_n$ using only information about the values of $u$ on $G_n$. This motivates the following two definitions (we recall that $h_n = 2^{-n}$).

\begin{definition}\label{definition:finite difference operators 1}
  The vector $(\nabla_n)^1u(x)$ is defined via the system of equations ($k=1,\ldots,d$)
  \begin{align*}
    (\nabla_n)^1u(x),e_k) := (2h_n)^{-1}[u(x+h_ne_k)-u(x-h_ne_k)]
  \end{align*}

\end{definition}

\begin{definition}\label{definition:finite difference operators 2}
  The matrix $(\nabla_n)^2u(x)$ is defined via the system of equations ($k,\ell=1,\ldots,d$),
  \begin{align*}
    ( (\nabla_n)^2u(x) e_k,e_\ell) := h_n^{-2	}\left [ u(x+h_n e_k+h_n e_\ell) - u(x+h_n e_k) - u(x+h_n e_\ell) + u(x) \right ]
  \end{align*}
  
\end{definition}

\begin{remark}\label{remark:finite difference translation invariance} From the definition it is clear that these discrete derivatives commute with translations with respect to a vector $z\in G_n$. That is, given a function $u$ and $z\in G_n$ then for every $x\in G_n$ we have
  \begin{align*}
    ((\nabla_n)^1 \tau_z u)(x) = ((\nabla_n)^1 u)(x+z)
  \end{align*}

\end{remark}

Depending on how regular the function $u$ is, these discrete derivative operators enjoy quantitative ``continuity estimates'' as functions on $G_n$. An important point being that these estimates are uniform in $n$ once $u$ is fixed. 
\begin{proposition}\label{proposition:discrete derivatives converge to real derivatives}
  There is a universal constant $C$ such that for $u \in C^{\beta}_b(\mathbb{R}^d)$ and $x \in G_n$, 
  \begin{align*}
    & | (\nabla_n)^1u(x)-\nabla u(x)| \leq C\|u\|_{C^{\beta}}h_n^{\beta-1},\;\textnormal{ if } \beta \in [1,2],\\
    & | (\nabla_n)^2u(x)-D^2 u(x)| \leq C\|u\|_{C^{\beta}}h_n^{\beta-2},\;\textnormal{ if } \beta \in [2,3).	
  \end{align*}

\end{proposition}

\begin{proof} See appendix.

\end{proof}

\begin{proposition}\label{proposition:discrete derivatives regularity}
  Fix $u \in C^{\beta}_b(\mathbb{R}^d)$. Then, given $x_1,x_2 \in G_n$, we have
  \begin{align*}
    | u(x_1)-u(x_2)| & \leq C\|u\|_{C^{\beta}}|x_1-x_2|^{\beta},\;\textnormal{ if } \beta \in [0,1],\\
    | (\nabla_n)^1u(x_1)-(\nabla_n)^1u(x_2)| & \leq C\|u\|_{C^{\beta}}|x_1-x_2|^{\beta-1},\;\textnormal{ if } \beta \in [1,2],\\
    | (\nabla_n)^2u(x_1)-(\nabla_n)^2 u(x_2)| & \leq C\|u\|_{C^{\beta}}|x_1-x_2|^{\beta-2},\;\textnormal{ if } \beta \in [2,3].	
  \end{align*}

\end{proposition}

\begin{proof} See appendix.

\end{proof}

%%%%%%%%%%%%%%%%%%%%%%%%%%%%%%%%%
%%%%%%%%%%%%%%%%%%%%%%%%%%%%%%%%%
\subsection{The Whitney Extension and Projection operators.}

\begin{definition}\label{definition:interpolating polynomials}

\begin{align*}
  p^\beta_{u,k}(x) := \left \{ \begin{array}{ll}
    u (\hat y_{n,k}) & \textnormal{ if } \beta\in [0,1)\\
    u (\hat y_{n,k})+(\nabla^1_n u(\hat y_{n,k}),x-\hat y_{n,k}) & \textnormal{ if } \beta \in [1,2)\\
    u (\hat y_{n,k})+(\nabla^1_n u(\hat y_{n,k}),x-\hat y_{n,k})+\tfrac{1}{2}\left ( \nabla^2_n u(\hat y_{n,k})(x-\hat y_{n,k}),(x-\hat y_{n,k})\right )	 & \textnormal{ if } \beta \in [2,3)
  \end{array}\right.
\end{align*}

\end{definition}

We are now ready to define the Whitney extension operator.
\begin{align}
  E^\beta_n(u,x) := \left \{ \begin{array}{ll}
    u(x) & \textnormal{ if } x\in G_n,\\
    \sum \limits_{k} p^\beta_{u,k}(x)\phi_{n,k}(x) & \textnormal{ if } x\not\in G_n.	
  \end{array}\right.
\end{align}

The projector operator $\pi_n^\beta:C^\beta_b(\mathbb{R}^d)\to C^\beta_b(\mathbb{R}^d)$ is given by
\begin{align}\label{equation:pi_n^beta definition}
  \pi_n^\beta := E_n^\beta \circ T_n,
\end{align}
where we recall that $T_nu = u_{\mid G_n}$ (Definition \ref{definition:discrete function spaces}).

\begin{theorem}\label{theorem:Whitney Extension Is Bounded}
  There is a constant $C$ such that for any $n$ and any $u\in C^\beta_b(\mathbb{R}^d)$ we have
  \begin{align*}
    \|\pi_n^\beta u \|_{C^\beta(\mathbb{R}^d)} \leq C\|u\|_{C^\beta(\mathbb{R}^d)}.
  \end{align*}
\end{theorem}

\begin{proof}
  This follows arguing exactly as in \cite[Chapter VI, Theorem 3 and 4]{Stei-71}, making use of the regularity estimates in Proposition \ref{proposition:discrete derivatives regularity}. Since this is a standard argument, we omit the details. 
\end{proof}

\begin{proposition}\label{proposition:Whitney extension translation invariance}
  Let $z \in G_n$ and $u \in C^\beta_b$, then.
  \begin{align*}
    \pi_n^\beta(\tau_{z}u) = \tau_{z} \pi_n^\beta (u).
  \end{align*}
\end{proposition}

\begin{proof}
  Let us show that $\pi_n^\beta(\tau_{z}u)(x) = \tau_{z} \pi_n^\beta (u)(x)$ for every $x \in \mathbb{R}^d$ and $z\in G_n$. Note that if $x\in G_n$ then the equality is trivial, so let us take $x\in \mathbb{R}^d\setminus G_n$ and $z \in G_n$, then we have
  \begin{align*}	
    \pi_n^\beta(\tau_{z}u)(x) = \sum \limits_{k} p^\beta_{\tau_{z}u,k}(x)\phi_{n,k}(x). 	
  \end{align*}	
  Furthermore, it is not difficult to check that (see Remark \ref{remark:finite difference translation invariance})
  \begin{align*}
    p^\beta_{\tau_{z}u,k}(x) = p^\beta_{u,\sigma_z k}(x+z),
  \end{align*}
  while part (4) of Proposition \ref{proposition:partition of unity properties} implies that 
  \begin{align*}
    \phi_{n,k}(x) = \phi_{n,\sigma_z k}(x+z).
  \end{align*}
  From these two identities we conclude that
  \begin{align*}
    \pi_n^\beta(\tau_{z}u)(x) = \sum \limits_{k} p^\beta_{u,\sigma_z k}(x+z)\phi_{n,\sigma_z k}(x+z) = \sum \limits_{k} p^\beta_{u,k}(x+z)\phi_{n,k}(x+z) = \tau_{z} \pi_n^\beta (u)(x),
  \end{align*}
  where we used that $\sigma_z$ is bijective, this proves the proposition.
\end{proof}

\begin{remark}\label{remark:approximation of directions}
  Given $\varepsilon \in (0,1)$ there is a $C>1$ such that for every $n\in \mathbb{N}$, $x_0 \in G_n$, and unit vector $x_* \in \mathbb{R}^d$ there is some $x_1 \in G_n$ and $s>0$ such that
  \begin{align*}
    |s x_* -(x_1-x_0)| \leq h_n,\; C^{-1}h_n^{\varepsilon}\leq |x_1-x_0| \leq C h_n^{\varepsilon}.    	  
  \end{align*}	  
  Indeed, this follows from the fact that $h_n^{\varepsilon} x_* \in [-h_n^{\varepsilon},h_n^\varepsilon]^d$ and that $[-h_n^{\varepsilon},h_n^\varepsilon]^d \cap (G_n-x_0)$ is a $h_n$-net in $[-h_n^{\varepsilon},h_n^\varepsilon]^d$, so there is $x_1 \in [-h_n^{\varepsilon},h_n^\varepsilon]^d \cap (G_n-x_0)$ such that $|h_n^\varepsilon x_*-(x_1-x_0)| \leq h_n$. Then, the inequalities for $|x_1-x_0|$ follow from two applications of the triangle inequality and the fact that $\varepsilon<1$  and $h_n \leq 1/2$ for all $n\geq 1$.
\end{remark}

\begin{proposition}\label{proposition:discrete derivative estimate at a local minimum}
  Let $w \in C^{\beta}_b(\mathbb{R}^d)$ be such that $w(x)\geq 0$ for every $x\in G_n$ and such that $w(x_0) = 0$ at some $x_0 \in G_n$. Then, there is a universal $C$ such that
  \begin{align*}
    |\nabla \pi^\beta_n w(x_0)| & \leq C\|w\|_{C^{\beta}} h_n^{\min\{2,\beta\}-1},\;\textnormal{ if } \beta \geq 1,\\	
    |(\nabla^2 \pi^\beta_n w(x_0))_-| & \leq C\|w\|_{C^{\beta}} h_n^{(\min\{3,\beta\}-2)/2},\;\textnormal{ if } \beta \geq 2.	  
  \end{align*}	  
  Here, for a given symmetric matrix $D$, $D_{-}$ denotes it's negative part. 
\end{proposition}

\begin{proof}
  Fix any $x\in G_n$. Thanks to Proposition \ref{proposition:discrete derivatives converge to real derivatives} and the fact that $|x-x_0|\geq h_n$ we have
  \begin{align*}
    |w(x)-w(x_0)- (\nabla \pi^\beta_n w(x_0),x-x_0)| \leq C\|w\|_{C^{\beta}}|x-x_0|^{\min\{2,\beta\}}.
  \end{align*}
  Since $w(x_0)=0$, and $w(x)\geq 0$ by assumption, 
  \begin{align*}
    0 \leq (\nabla \pi^\beta_n w(x_0),x-x_0) + C\|w\|_{C^{\beta}}|x-x_0|^{\min\{2,\beta\}}.
  \end{align*}
  It is easy to see there is some $x_1 \in G_n$ such that $|x_0-x_1| = h_n$ and
  \begin{align*}
    (\nabla \pi^\beta_n w(x_0),x_1-x_0) = -|\nabla \pi^\beta_n w(x_0)|_{\ell^\infty}|x_1-x_0|,
  \end{align*}
  and therefore, 
  \begin{align*}
    (\nabla \pi^\beta_n w(x_0),x_1-x_0) \leq -C_d^{-1}|\nabla \pi^\beta_n w(x_0)||x_1-x_0|.
  \end{align*}
  Combining these inequalities and recalling Theorem \ref{theorem:Whitney Extension Is Bounded} it follows that
  \begin{align*}
    |\nabla \pi^\beta_n w(x_0)|\leq C\|w\|_{C^{\beta}}h_n^{\min\{2,\beta\}-1}.
  \end{align*}
  This proves the estimate for the gradient when $\beta\geq 1$. Now assume $\beta \geq 2$, the beginning of the argument in this case goes along similar lines. For any $x\in G_n$ we have that
  \begin{align*}
    |w(x)-w(x_0)- (\nabla \pi^\beta_n w(x_0),x-x_0)-\tfrac{1}{2}(\nabla^2 \pi^\beta_n w(x_0)(x-x_0),x-x_0 )| \leq C\|w\|_{C^{\beta}}|x-x_0|^{\min\{3,\beta\}},
  \end{align*}
  where we have once again used Theorem \ref{theorem:Whitney Extension Is Bounded}. Thus, since $w(x_0) = 0$ and $w(x)\geq 0$ for $x \in G_n$,
  \begin{align*}
    (\nabla \pi^\beta_n w(x_0),x-x_0)+\tfrac{1}{2}(\nabla^2 \pi^\beta_n w(x_0)(x-x_0),x-x_0 ) + C\|w\|_{C^{\beta}}|x-x_0|^{\min\{3,\beta\}} \geq 0.
  \end{align*}
  Now, since we are on a lattice, it is obvious that for any $x\in G_n$ we have that $x' := 2x_0-x \in G_n$. In this case we can add up the inequalities for $x$ and $x'$, and conclude that
  \begin{align*}
    & (\nabla \pi^\beta_n w(x_0),x-x_0)+\tfrac{1}{2}(\nabla^2 \pi^\beta_n w(x_0)(x-x_0),x-x_0 )\\
    & +(\nabla \pi^\beta_n w(x_0),x'-x_0)+\tfrac{1}{2}(\nabla^2 \pi^\beta_n w(x_0)(x'-x_0),x'-x_0 ) + 2C\|w\|_{C^{\beta}}|x-x_0|^{\min\{3,\beta\}} 	\geq 0.
  \end{align*}
  Since $x'-x_0 = -(x-x_0)$, we conclude that
  \begin{align*}
    & (\nabla^2 \pi^\beta_n w(x_0)(x-x_0),x-x_0 )+ 2C\|w\|_{C^{\beta}}|x-x_0|^{\min\{3,\beta\}} \geq 0,\;\forall\;x\in G_n.
  \end{align*}
  Let $x_* \in \mathbb{R}^d$ be a unit vector such that 
  \begin{align*}
    -(\nabla^2 \pi^\beta_n w(x_0)x_*,x_* )  = |(\nabla^2 \pi^\beta_n w(x_0))_-|
  \end{align*}
  According to Remark \ref{remark:approximation of directions}, there is $x_1 \in G_n$ and $s>0$ such that 
  \begin{align*}
    |sx_*-(x_1-x_0)| \leq  h_n,\;\; C^{-1} h_n^{\varepsilon}\leq |x_1-x_0| \leq Ch_n^{\varepsilon}.
  \end{align*}
  For this $x_1$ we have
  \begin{align*}
     |(\nabla^2 \pi^\beta_n w(x_0))_-|s^2 & = -(\nabla^2 \pi^\beta_n w(x_0)x_*,x_*)s^2 \\ 
	 & \leq -(\nabla^2 \pi^\beta_n w(x_0)(x_1-x_0),x_1-x_0 ) + C\|w\|_{C^\beta} |sx_*-(x_1-x_0)|.
  \end{align*}
  This, together with the previous step, shows that 
  \begin{align*}
    C^{-2}|(\nabla^2 \pi^\beta_n w(x_0))_-|(h_n^{\varepsilon})^2 \leq 2C\|w\|_{C^\beta}h_n^{\min\{3,\beta\}\varepsilon} + C\|w\|_{C^\beta} h_n,
  \end{align*}
  again having used Theorem \ref{theorem:Whitney Extension Is Bounded}. Simplifying, this becomes
  \begin{align*}
    |(\nabla^2 \pi^\beta_n w(x_0))_-| \leq C\|w\|_{C^{\beta}}(h_n^{(\min\{3,\beta\}-2)\varepsilon}+h_n^{1-\varepsilon}).
  \end{align*}
  Choosing $\varepsilon = 1/2$, and noting $\min\{3,\beta\}-2)\leq 1$, we conclude that
  \begin{align*}
    |(\nabla^2 \pi^\beta_n w(x_0))_-| \leq C\|w\|_{C^{\beta}}h_n^{(\min\{3,\beta\}-2)/2}.
  \end{align*}

\end{proof}

We fix an auxiliary function $\eta_0:[0,\infty) \to\mathbb{R}_+$, with $\eta_0 \in C^\infty(\mathbb{R}_+)$, and 
\begin{align}\label{equation:eta_0 approximation to min of 1 and y squared}
  \;0\leq \eta_0 \leq 1,\;\eta_0'(t) \geq 0 \textnormal{ for all }t, \eta_0(t) = t \textnormal{ for } t\leq 1/2,\;\eta_0(t) = 1 \textnormal{ for } t \geq 1.
\end{align}
The function $\eta_0$, as well as the following two estimates, will be useful in the next section. Essentially, $\eta_0(t)$ should be thought of as a smooth replacement for $\min\{1,t\}$.

\begin{lemma}\label{lemma:Whitney Extension Is Almost Order Preserving}
  Let $1\leq \beta<\beta_0<3$, and consider $w\in C^{\beta_0}_b(\mathbb{R}^d)$ and $x_0\in G_n$ such that
  \begin{align*}
    w \geq 0 \textnormal{ in } G_n \textnormal{ and } w(x_0) = 0.	  
  \end{align*}	  
  Then, there is a function $R_{\beta_0,n,w,x_0}$ such that $R(x_0) = 0$, and
  \begin{align*}
    & \pi_n^\beta w(x)+R_{\beta_0,n,w,x_0}(x) \geq 0,\;\;\forall\;x\in\mathbb{R}^d,\\
    & \|R_{\beta_0,n,w,x_0}\|_{C^\beta(\mathbb{R}^d)} \leq Ch_n^{\gamma} \|w\|_{C^{\beta_0}(\mathbb{R}^d)},
  \end{align*}	 
  for some constant $\gamma = \gamma(\beta,\beta_0) \in (0,1)$.
  
\end{lemma}

\begin{remark}\label{remark:Whitney extension is order preserving for beta<1}
  For $\beta\in(0,1)$, it is straightforward that $w\geq 0$ in $G_n$ guarantees that $\pi_n^\beta w\geq 0$ everywhere, that is, the Whitney extension for $\beta \in (0,1)$ is order preserving. Accordingly, Lemma \ref{lemma:Whitney Extension Is Almost Order Preserving} is only needed for $\beta>1$.

\end{remark}

\begin{proof}
  We consider the cases $1\leq \beta<2$ and $\beta\geq 2$ separately. First suppose $\beta \in [1,2)$. Let $\phi_0(t)$ be a smooth function such that $0\leq \phi_0(t)\leq 1$ for all $t$, $\phi_0(t)=1$ for $t\leq 1/4$ and $\phi_0(t)=0$ for $t\geq 1$. Then set
  \begin{align*}
    \tilde w(x) = \pi^\beta_n w(x) - ( \nabla \pi^\beta_nw(x_0), x-x_0) \phi_0(x-x_0).
  \end{align*}
  For each $x\in \mathbb{R}^d$, let $\hat x$ denote a point in $G_n$ such that $|x-\hat x| = \textnormal{dist}(x,G_n)\leq h_n$. Then, since $w(\hat x)\geq 0$ for any $\hat x$ (from the assumption), we have
  \begin{align*}
    \tilde w(x)  & = \tilde w(\hat x) + (\tilde w(x)-\tilde w(\hat x))\\
	  & \geq - ( \nabla \pi^\beta_n w (x_0),x-x_0) \phi_0(x-x_0) - C\|\tilde w\|_{C^{\beta_0}} |\hat x-x|\\
	  & \geq - ( \nabla \pi^\beta_nw (x_0),x-x_0) \phi_0(x-x_0) - C\|\tilde w\|_{C^{\beta_0}} h_n.
  \end{align*}
  By Proposition \ref{proposition:discrete derivative estimate at a local minimum}, we have $|\nabla \pi^\beta_nw (x_0)| \leq C\|w\|_{C^{\beta_0}}h_n$ when $\beta_0>1$, therefore, 
  \begin{align*}
    \tilde w(x) \geq -C\|w\|_{C^{\beta_0}}h_n,\;\;\forall\;x\in\mathbb{R}^d,
  \end{align*}
  where we have used Theorem \ref{theorem:Whitney Extension Is Bounded} to bound $\|\pi_n^\beta w\|_{C^\beta_0}$. On the other hand, since $\beta_0 >1 $ and $\nabla \tilde w(x_0)=0$, we have
  \begin{align*}
    \tilde w(x) & \geq -\|\tilde w\|_{C^{\beta_0}}|x-x_0|^{\beta_0},\\
	  & \geq -C\|w\|_{C^{\beta_0}}|x-x_0|^{\beta_0}\;\;\forall\;x\in\mathbb{R}^d,
  \end{align*}
  Now, we take $\eta_0$ as in \eqref{equation:eta_0 approximation to min of 1 and y squared} and define the function
  \begin{align*}
    \tilde R(x) := 2C\|w\|_{C^{\beta_0}}h_n \eta_0 \left ( \frac{|x-x_0|^{\beta_0}}{h_n} \right ).
  \end{align*}
  If $|x-x_0|^{\beta_0} \geq h_n/2$, then 
  \begin{align*}
    \tilde w(x)+\tilde R(x) & =  \tilde w(x) + C\|w\|_{C^{\beta_0}}h_n \geq 0.
  \end{align*}
  If on the contrary, $|x-x_0|^{\beta_0} \leq h_n/2$, then
  \begin{align*}
    \tilde w(x)+\tilde R(x) & =  \tilde w(x) + C\|w\|_{C^{\beta_0}}|x-x_0|^{\beta_0} \geq 0.
  \end{align*}
  We conclude that
  \begin{align*}
    \tilde w(x)+\tilde R(x) \geq 0,\;\forall\;x\in\mathbb{R}^d.
  \end{align*}
  On the other hand, an elementary computation (see the Appendix) shows that
  \begin{align*}
    \|\tilde R\|_{C^{\beta}} \leq C h_n^{\gamma} \|w\|_{C^{\beta_0}}.
  \end{align*}
  Finally, let
  \begin{align*}
    R_{\beta_0,n,w,x_0}(x) := \tilde R(x) - ( \nabla \pi^\beta_n(x_0), x-x_0) \phi_0(x-x_0).
  \end{align*}
  We conclude that $\|R_{\beta_0,n,w,x_0}\|_{C^\beta} \leq C h_n^{\gamma}\|w\|_{C^{\beta_0}}$ and
  \begin{align*}
    \pi^\beta_n w(x)+R_{\beta_0,n,w,x_0}(x) \geq 0,\;\forall\;x\in\mathbb{R}^d.
  \end{align*}
  This proves the Proposition when $\beta \in [1,2)$. The argument for $\beta\geq 2$ is similar, we only highlight the main differences. This time, we subtract not just the first order part of $w$ near $x_0$, but also the second order part, namely we consider the function
  \begin{align*}
    \tilde{\tilde w} := \pi_n^\beta w(x) - ( \nabla \pi^\beta_n(x_0), x-x_0) \phi_0(x-x_0) -  \tfrac{1}{2}( (\nabla^2 \pi^\beta_n(x_0))_-(x-x_0), x-x_0) \phi_0(x-x_0).
  \end{align*}	  
  Then, one applies again Proposition \ref{proposition:discrete derivative estimate at a local minimum} and use the regularity of $w$ to obtain (in analogy to the previous case)
  \begin{align*}
    \tilde{\tilde w}(x) \geq -C\|w\|_{C^{\beta_0}} \max \{ h_n, |x-x_0|^{\beta_0}\}
  \end{align*}	  
  The respective function $\tilde{\tilde R}$ is defined exactly as $\tilde R$ and one argues as in the previous case. 
\end{proof}

\begin{remark}\label{remark:Whitney extension is almost order preserving L infinity version}
  The argument in the proof provides -after small modifications- a closely related result: if instead of $w \in C^{\beta}_b(\mathbb{R}^d)$ we assume that $w \in C^0_b(\mathbb{R}^d)$ and that for some $M>0$ and $\beta_0>\beta$ we have
  \begin{align*}
    |w(x)| \leq M|x-x_0|^{\beta_0},\;\forall\;x\in \mathbb{R}^d,
  \end{align*}	  
  then there is as before a function $\hat R_{\beta_0,n,w,x_0}$ such that $\hat R_{\beta_0,n,w,x_0}(x_0)=0$ and $ \pi_n^\beta w(x)+R_{\beta_0,n,w,x_0}(x)\geq 0$ for all $x$, but this time the $C^{\beta}$ estimate for $\hat R_{\beta_0,n,w,x_0}$ is 
  \begin{align*}
    \|\hat R_{\beta_0,n,w,x_0}\|_{C^{\beta}} \leq Ch_n^{\gamma}(\|w\|_{L^\infty}+M).  
  \end{align*}	  
\end{remark}

The following proposition will be useful later in the proof of Proposition \ref{proposition:touching from above function with decay at point operator estimate}.
\begin{proposition}\label{proposition:touching from above function with decay at a point}
  Let $1\leq \beta < \beta_0<3$ or $\beta \in (0,1)$ and $\beta_0 = \beta$. Fix $f \in C^\infty_c(\mathbb{R}^d)$, and let $\eta_0$ be as in \eqref{equation:eta_0 approximation to min of 1 and y squared}. Let $x_0 \in G_n$ and $w(x) = f(x-x_0) \eta_0 (|x-x_0|^{\beta_0})$, then
  \begin{align*}
    \pi^\beta_n(w,x) & \leq C\|f\|_{L^\infty} \eta_0(|x-x_0|^{\beta_0}),\;\forall\;x\in\mathbb{R}^d, \textnormal{ if } \beta\in (0,1),\\
    \pi^\beta_n(w,x) & \leq C\|f\|_{L^\infty} \eta_0(|x-x_0|^{\beta_0}) + \hat R_{\beta_0,n,w,x_0}(x) ,\;\forall\;x\in\mathbb{R}^d, \textnormal{ if } \beta \in [1,2],
  \end{align*}
  for some function $\hat R_{\beta_0,n,w,x_0}$ such that $\hat R_{\beta_0,n,w,x_0}(x_0) = 0$ and
  \begin{align*}
    \|\hat R_{\beta_0,n,w,x_0}\|_{C^{\beta}} \leq C\|f\|_{L^\infty}h_n^{\gamma},
  \end{align*}	  
  where $\gamma$ is as in Lemma \ref{lemma:Whitney Extension Is Almost Order Preserving}.
\end{proposition}

\begin{proof}
  Define the function $\tilde w(x) := (\|f\|_{L^\infty}-f(x-x_0))\eta_0(|x-x_0|^{\beta_0})$. Then $\tilde w(x_0) = 0$ and
  \begin{align*}
    |\tilde w(x)| \leq 2\|f\|_{L^\infty}\eta_0(|x-x_0|^{\beta_0}),\;\forall\;x\in\mathbb{R}^d,
  \end{align*}
  while, since $\eta_0\geq 0$, we also have $\tilde w(x)\geq 0$ for every $x\in G_n$. If $\beta \in [1,2]$, using Lemma \ref{lemma:Whitney Extension Is Almost Order Preserving} and the function $\hat R_{\beta_0,n,w,x_0}$ from Remark \ref{remark:Whitney extension is almost order preserving L infinity version}, we have
  \begin{align*}
    \pi^\beta_n(\tilde w,x) + \hat R_{\beta_0,n,w,x_0}(x) \geq 0,\;\;\forall\;x,
  \end{align*}
  This inequality, after some rearranging, yields (for $\beta\in[1,2]$)
  \begin{align*}
    \pi^\beta_n(w,x)\leq \|f\|_{L^\infty} \pi^\beta_n(\eta_0(|\cdot-x_0|^{\beta_0}),x) + \hat R_{\beta_0,n,w,x_0}(x) ,\;\;\forall\;x\in\mathbb{R}^d.
  \end{align*}
  Since we also have $\|\tilde w\|_{L^\infty} \leq C\|f\|_{L^\infty}$, we have again by Remark \ref{remark:Whitney extension is almost order preserving L infinity version}
  \begin{align*}
    \|\hat R_{\beta_0,n,w,x_0}\|_{C^{\beta}} \leq C\|f\|_{L^\infty}h_n^{\gamma}, 
  \end{align*}	  
  and the Proposition is proved in this case. For $\beta \in (0,1)$ we argue along similar lines, using Remark \ref{remark:Whitney extension is order preserving for beta<1} instead of Lemma \ref{lemma:Whitney Extension Is Almost Order Preserving}.

\end{proof}

%%%%%%%%%%%%%%%%%%%%%%%%%%%%%%%%%
%%%%%%%%%%%%%%%%%%%%%%%%%%%%%%%%%
\subsection{Convergence of the projection operators}

\begin{lemma}\label{lemma:projection operators convergence}
  Let $0<\beta<\beta_0<3$, there is a constant $C$ such that if $u\in C^{\beta_0}_b(\mathbb{R}^d)$, then 
  \begin{align*}
    \|\pi_n^\beta u - u \|_{C^{\beta}} \leq Ch_n^{\gamma}\|u\|_{C^{\beta_0}}.
  \end{align*}
  Here, $\gamma= \gamma(\beta_0,\beta) \in (0,1)$.
\end{lemma}

\begin{proof}
  For notational simplicity let us write $f(x) = \pi_n^\beta u(x)$ throughout the proof. 
  
  Since $u=f$ throughout $G_n$, for an arbitrary $x\in G_n$ we have (with $\hat x$ denoting a point in $G_n$ such that $\textnormal{dist}(x,G_n) = |x-\hat x|$), with $\alpha := \min\{1,\beta_0\}$
  \begin{align*}
    |u(x)-f(x)| & \leq |f(x)-f(\hat x)| + |u(\hat x)-u(x)|\\
	  & \leq |x-\hat x|^\alpha [f]_{C^{\alpha}} + |x-\hat x|^{\alpha} [u]_{C^{\alpha}}\\
	  & \leq C\|u\|_{C^{\beta_0}}h_n^\alpha \leq C\|u\|_{C^{\beta_0}}h_n^\alpha,
  \end{align*}
  where we made use of Theorem \ref{theorem:Whitney Extension Is Bounded} to obtain $[f]_{C^\alpha} \leq C\|u\|_{C^\beta}$. This shows that $\|u-f\|_{L^\infty}$ goes to zero at some rate determined by $\beta_0$ and the size of $\|u\|_{C^{\beta_0}}$. To prove the lemma we need to also bound the H\"older seminorm of $u-f$ and its derivatives, according to $\beta_0$.
  
  \textbf{The case $\beta, \beta_0 \in [0,1)$}. Fix $x_1,x_2 \in \mathbb{R}^d$. First, suppose that $|x_1-x_2| \leq \max\{|x_1-\hat x_1|,|x_2-\hat x_2|\}$, then
  \begin{align*}
    |f(x_1)-u(x_1)-(f(x_2)-u(x_2))| \leq [f-u]_{C^{\beta_0}}|x_1-x_2|^{\beta_0} \leq C\|u\|_{C^{\beta_0}}|x_1-x_2|^{\beta_0}.
  \end{align*}
  In this case, and since $0\leq \beta <\beta_0<1$, we have that $|x_1-x_2|^{\beta_0-\beta} \leq \max \{|x_1-\hat x_1|^{\beta_0-\beta},|x_2-\hat x_2|^{\beta_0-\beta}\} \leq h_n^{\beta_0-\beta}$. Then, using Theorem \ref{theorem:Whitney Extension Is Bounded}
  \begin{align*}
    |f(x_1)-u(x_1)-(f(x_2)-u(x_2))| \leq [f-u]_{C^\beta}|x_1-x_2|^{\beta} \leq C\|u\|_{C^{\beta_0}}h_n^{\beta_0-\beta}|x_1-x_2|^{\beta}.
  \end{align*}
  Next, suppose that $|x_1-x_2| > \max\{|x_1-\hat x_1|,|x_2-\hat x_2|\}$. In this case
  \begin{align*}
    |f(x_1)-u(x_1)-(f(x_2)-u(x_2))| & \leq \|f\|_{C^{\beta_0}}|x_1-\hat x_1|^{\beta_0}+\|u\|_{C^{\beta_0}}|x_2-\hat x_2|^{\beta_0}\\
      & \leq C\|u\|_{C^{\beta_0}}h_n^{\beta_0-\beta}|x_1-x_2|^{\beta}, 	
  \end{align*}
  where once again Theorem \ref{theorem:Whitney Extension Is Bounded} was used. Combining these two estimates, we conclude that
  \begin{align*}
    [f-u]_{C^\beta} = \sup \limits_{x_1 \neq x_2} \frac{|f(x_1)-u(x_1)-(f(x_2)-u(x_2))|}{|x_1-x_2|^{\beta}} \leq C\|u\|_{C^{\beta_0}}h_n^{\beta_0-\beta}.
  \end{align*}
  Then, using that $h_n \leq 1$ for all $n\geq 1$, we have
  \begin{align*}
    \|f-u\|_{C^\beta} \leq Ch_n^{\gamma}\|u\|_{C^{\beta_0}}.
  \end{align*}
  
  \textbf{The case $\beta, \beta_0 \in [1,2)$}. In this case we trivially have the same estimates from the previous case, and  only need the bounds for first derivative. This is done as follows, first
  \begin{align*}
    |\nabla f(x)-\nabla u(x)| \leq |\nabla f(x)-\nabla f(\hat x)| + |\nabla f(\hat x)-\nabla u(\hat x)| + |\nabla u(x)-\nabla u(\hat x)|.
  \end{align*}
  Then, using Theorem \ref{theorem:Whitney Extension Is Bounded}, we have
  \begin{align*}
    |\nabla f(x)-\nabla u(x)| & \leq [\nabla f]_{C^{\beta_0-1}}h_n^{\beta_0-1} + |\nabla f(\hat x)-\nabla u(\hat x)| + [\nabla u]_{C^{\beta_0-1}}h_n^{\beta_0-1}\\
	  & \leq C\|u\|_{C^{\beta_0}}h_n^{\beta_0-1}+ |\nabla f(\hat x)-\nabla u(\hat x)|.
  \end{align*}
  Recall that $\nabla f(\hat x) = (\nabla_n)^1 u(\hat x)$, and use Proposition \ref{proposition:discrete derivatives converge to real derivatives} to conclude that 
  \begin{align*}
    |\nabla f(x)-\nabla u(x)| & \leq C\|u\|_{C^{\beta_0}}h_n^{\beta_0-1}+ C\|u\|_{C^{\beta_0}}h_n^{\beta_0-1}.
  \end{align*}
  The H\"older seminorm $[\nabla f-\nabla u]_{C^\beta}$ is bounded with the same argument used to bound $[f-u]_{C^{\beta}}$ in the previous case, we omit the details.
  
  \textbf{The case $\beta = 2, \beta_0 \in (2,3)$}. Right as before, we note that
  \begin{align*}
    |D^2 f(x)-D^2 u(x)| \leq |D^2f(x)-D^2f(\hat x)| + |D^2 f(\hat x)-D^2 u(\hat x)| + |D^2 u(x)-D^2 u(\hat x)|.
  \end{align*}
  Then, applying Theorem \ref{theorem:Whitney Extension Is Bounded} and Proposition \ref{proposition:discrete derivatives converge to real derivatives} as in the previous case, we have
  \begin{align*}
    |D^2 f(x)-D^2 u(x)| & \leq [D^2 f]_{C^{\beta_0-2}}h_n^{\beta_0-2} + |D^2 f(\hat x)-D^2 u(\hat x)| + [D^2 u]_{C^{\beta_0-2}}h_n^{\beta_0-2}\\
	  & \leq 2C\|u\|_{C^{\beta_0}}h_n^{\beta_0-2} + |\nabla f(\hat x)-\nabla u(\hat x)|\\
	  & \leq 3C\|u\|_{C^{\beta_0}}h_n^{\beta_0-2}.
  \end{align*}
  For the H\"older seminorm, we repeat the argument used in the case $\beta \in (0,1)$, again we leave the details to the reader.
 \end{proof}

\begin{remark}\label{remark:projection operators C0 convergence}  
  If $u\in C^0_b(\mathbb{R}^d)$, then the same argument from Lemma \ref{lemma:projection operators convergence} can be used to show
  \begin{align*} 
    \lim \limits_{n\to \infty} \|u-\pi_n^0(u)\|_{L^\infty(\mathbb{R}^d)} = 0,  
  \end{align*}
  the rate of convergence being determined by the modulus of continuity of $u$.

\end{remark}

%%%%%%%%%%%%%%%%%%%%%%%%%%%%%%%%%
%%%%%%%%%%%%%%%%%%%%%%%%%%%%%%%%%
%%%%%%%%%%%%%%%%%%%%%%%%%%%%%%%%%
%%%%%%%%%%%%%%%%%%%%%%%%%%%%%%%%%
%%%%%%%%%%%%%%%%%%%%%%%%%%%%%%%%%
%%%%%%%%%%%%%%%%%%%%%%%%%%%%%%%%%
\section{Analysis of $I(u,x)$ via the finite dimensional approximations}\label{section:Analysis of finite dimensional approximations}

In this section we introduce a sequence of operators $I_n$ which approximate $I$. The operators $I_n$ behave like operators in a finite dimensional vector space in the sense that they arise from a composition between linear maps with a Lipschitz map from a finite dimensional space onto itself. This allows us to prove a min-max formula for $I_n(u,x)$ at least when $x\in G_n$ by using Clarke's idea of a generalized gradient \cite{Cla1990optimization}. More precisely, we use the fact that $I_n$ factorizes via a map between finite dimensional vector spaces (which is what the spaces $C_*(G_n)$ were introduced for), where the generalized gradient can be used, and then lift this to corresponding maps from $C^\beta_b(\mathbb{R}^d)$ to $C_b^0(\mathbb{R}^d)$ using the Whitney extension. The majority of the section is concerned with deriving estimates and regularity properties for the linear operators arising in the min-max formula for $I_n$, and ultimately concluding such linear operators are pre-compact, which leads to a min-max formula for the original operator.

%%%%%%%%%%%%%%%%%%%%%%%%%%%%%%%%%
%%%%%%%%%%%%%%%%%%%%%%%%%%%%%%%%%
\subsection{The operators $I_n$ and their min-max representation}

We are going to approximate the operator $I(\cdot,x)$ via ``finite dimensional approximations'', this referring to maps $I_n:C^\beta_b\to C^0_b$, which factorize through a finite dimensional space (see \eqref{equation:I n via little i n} below).

We introduce a modification of the projection operator $\pi_n^0$ defined in \eqref{equation:pi_n^beta definition}. First, we define
\begin{align*}
  \textnormal{Pr}_n:C(G_n) \to C_*(G_n),\;\; \textnormal{Pr}_n(u)(x) := u(x) \chi_{[-2^{n},2^{n}]^d}(x).
\end{align*}
That is, given $u\in C(G_n)$, we define $\textnormal{Pr}_n(u)$ as the function obtained by restricting $u$ to $G_n \cap [-2^{n},2^{n}]^d$ and then extending it to the rest of $G_n$ by zero. Then, we define the modified Whitney extension,
\begin{align*}
  \hat E^\beta_n  := E_n^\beta \circ  \textnormal{Pr}_n,
\end{align*}
and the modified projection operator
\begin{align*}
  \hat \pi_n^\beta := \hat E^\beta_n \circ T_n.
\end{align*}
These are, respectively, bounded linear maps from $C(G_n)$ to $C^\beta_b(\mathbb{R}^d)$ and from $C^0_b(\mathbb{R}^d)$ to $C^\beta_b(\mathbb{R}^d)$. Now we are ready to introduce the finite dimensional approximations to the operator $I$, define
\begin{align}\label{equation:I_n definition}
  I_n &= \hat \pi_n^0 \circ I \circ \hat \pi_n^\beta,\;\;I_n:C^\beta_b(\mathbb{R}^d) \to C_b^0(\mathbb{R}^d).
\end{align}
That is, to compute $I_n(u,x)$, we first compute the modified projection $\hat \pi_n^\beta u$, and compute $I(\hat \pi_n^\beta u)$, to which we later apply the modified projection $\hat \pi_n^0$. In particular, $I_n$ only depends on the values of $u$ on $G_n \cap [-2^n,2^n]^d$. Associated to this, we introduce a map, $i_n$, defined as follows
\begin{align}\label{equation:little i n via I n} 
  i_n:C_*(G_n)\to C_*(G_n),\ \   i_n & = \textnormal{Pr}_n \circ T_n \circ I \circ E_n^\beta.
\end{align}
From the definition of $I_n$, we have $I_n = E_n^\beta \circ \textnormal{Pr}_n \circ T_n \circ I \circ E_n^\beta \circ \textnormal{Pr}_n \circ T_n$, thus we see $I_n$ and $i_n$ are themselves related by
\begin{align}\label{equation:I n via little i n} 
  I_n & = E_n^0 \circ i_n \circ \textnormal{Pr}_n \circ T_n.
\end{align}
The situation for both $I_n$ and $i_n$ is represented in the following two diagrams,
\begin{equation*}
  \begin{tikzcd}
    C^\beta_b(\mathbb{R}^d) \arrow{r}{I_n} \arrow[swap]{d}{\hat \pi_n^\beta} & C^0_b(\mathbb{R}^d)\\
    C^\beta_b(\mathbb{R}^d) \arrow{r}{I} & C^0_b(\mathbb{R}^d)  \arrow{u}{\hat \pi_n^0} 
  \end{tikzcd} \quad\quad\quad \begin{tikzcd}
    C_*(G_n) \arrow{r}{i_n} \arrow[swap]{d}{E_n^\beta} & C_*(G_n) \\
    C^\beta_b(\mathbb{R}^d) \arrow{r}{I} & C^0_b(\mathbb{R}^d) \arrow{u}{\textnormal{Pr}_n \circ T_n} 
  \end{tikzcd}
\end{equation*}
Now, the space $C_*(G_n)$ is finite dimensional (Remark \ref{remark:C star G_n is finite dimensional}), and the map $i_n:C_*(G_n)\to C_*(G_n)$ is Lipschitz continuous. Therefore, tools available for Lipschitz functions in the finite dimensional setting can be applied to $i_n$ and then related to $I_n$ via \eqref{equation:I n via little i n}.

We recall the generalized derivative of $i_n$ in the sense of Clarke \cite[Section 2.6]{Cla1990optimization}.
\begin{definition}\label{definition:Clarke differential}
  Let $V$ be a Banach space, and $T:V\to V$ a Lipschitz continuous function. We define the set of generalized derivatives of $T$, by 
  \begin{align*}
    \mathcal{D}T := \textnormal{c.h.}\{ L:V\to V \mid L = \lim \limits_{k} L_k \textnormal{ where } L_k= DT(x_k),\; T \textnormal{ is differentiable at } x_k \;\forall\; k \}.
  \end{align*}

\end{definition}

  By Rademacher's theorem, the set $\mathcal{D}T$ is not empty when $V$ is finite dimensional. Applying this to $i_n:C_*(G_n)\to C_*(G_n)$, we have, first, that $\mathcal{D}i_n$ is non-empty, and secondly that $\mathcal{D}I_n$ is non-empty as well, this is proved in Lemma \ref{lemma:characterization of D I sub n}, where we describe the relationship between $\mathcal{D}i_n$ to $\mathcal{D}I_n$. The following Lemma is the mean value theorem for nonsmooth Lipschitz functions between finite dimensional spaces (note the similarity with Theorem \ref{theorem:Lebourg}).
   
\begin{lemma}\label{lemma:min max finite dimensional map}
  Assume that $I: C^\beta_b(\real^d)\to C^0_b(\real^d)$ is Lipschitz. For any $u,v \in C_*(G_n)$, there is a $L \in \mathcal{D}i_n$ such that 
  \begin{align*}
    i_n(u,x) - i_n(v,x) = L(u-v,x). 
  \end{align*}

\end{lemma}

\begin{proof}
  We refer the reader to \cite[Proposition 2.6.5]{Cla1990optimization} for a proof of the lemma.
\end{proof}

The second lemma is basically the chain rule. 
\begin{lemma}\label{lemma:characterization of D I sub n}
  Assume that $I: C^\beta_b(\real^d)\to C^0_b(\real^d)$ is Lipschitz.  The set $\mathcal{D}I_n$ is non-empty, and for any $L \in \mathcal{D}I_n$ there is a $\tilde L \in \mathcal{D}i_n$ such that
  \begin{align*}
    L =E_n^0 \circ \tilde L \circ T_n,
  \end{align*}
  conversely, any $L$ defined in this way for some $\tilde L \in \mathcal{D}i_n$ belongs to $\mathcal{D}I_n$.
\end{lemma}

\begin{proof}
  Note that $I_n$ is differentiable at a point $u$ if and only if $i_n$ is differentiable at $\tilde u = T_nu$, a fact which follows applying the chain rule to the identities \eqref{equation:little i n via I n}  and \eqref{equation:I n via little i n}. Furthermore, at such $u$'s we have
  \begin{align*}
    DI_n(u) = E_n^* \circ Di_n(\tilde u) \circ T_n.
  \end{align*}
  If $u_k$ is a sequence along which $I_n$ is differentiable, and $L_k:=DI_n(u_k)$ converges to some $L$, then the sequence $\tilde L_k := Di_n(\tilde u_k)$ has a limit $\tilde L$, and $L = E_n^* \circ \tilde L \circ T_n$, taking the convex hull and by the linearity of $E_n^*$ and $T_n$, the lemma follows.
\end{proof}

The following remark will not be of any relevance until the proof of Theorem \ref{theorem:MinMax Euclidean ver2} at the end of this section, but we include it here to illustrate how Lemmas \ref{lemma:min max finite dimensional map} and \ref{lemma:characterization of D I sub n} immediately yield a min-max formula for $I_n(u,x)$ (for $x\in G_n$).
\begin{remark}\label{remark:MinMax for In}
  Fix $n$ and let $x\in G_n$. Then for any $u \in C^\beta_b(\mathbb{R}^d)$ we have 
  \begin{align}\label{equation:minmax for I_n}
    I_n(u,x) \leq \max \limits_{L\in \mathcal{D}I_n} \{ I_n(v,x)+L(u-v,x)\},\;\;\forall\;x\in G_n, u,v\in C^\beta_b(\mathbb{R}^d). 
  \end{align}	  
  Indeed, according to Lemma \ref{lemma:min max finite dimensional map} given $u$ and $v$ says there is some $\tilde L \in\mathcal{D}i_n$ such that	
  \begin{align*}
    i_n(u)-i_n(v) = \tilde L(u-v).
  \end{align*}	  
  In this case, we have $E_n^0(i_n(u))-E_n^0(i_n(v)) = E_n^0(\tilde L(u-v))$, and thus setting $L := E_n^0 \circ \tilde L \circ T_n \in \mathcal{D}I_n$, we have
  \begin{align*}
    I_n(u) = I_n(v) + L(u-v),
  \end{align*}
  and \eqref{equation:minmax for I_n} immediately follows.
\end{remark}
Next we make an elementary observation regarding the nature of the operators $L \in \mathcal{D}I_n$. This observation is merely a consequence of the factorization of $I_n$ through the space $C(G_n)$.
\begin{remark}\label{remark:structure of D little i n kernel}
  For each $L \in \mathcal{D}I_n$ there is a function $K = K_L$, $K: G_n \times G_n \to \mathbb{R}$ such that 
  \begin{align}\label{equation:L in DI_n formula}
    Lu(x) = \sum \limits_{y\in G_n} K(x,y)u(x+y),\;\;\forall\;u\in C^\beta_b(\mathbb{R}^d).
  \end{align}
  Indeed, simply let us use the basis functions $\{e_y\}_{y\in G_n} \subset C(G_n)$ given by
  \begin{align*}
    e_y(x) = \left \{ \begin{array}{rl}
	  1 & \textnormal{ if } x =y,\\
	  0 & \textnormal{ if } x\neq y.
	  \end{array}\right.
  \end{align*}	  
  Observe that for any $u \in C^{\beta}_b(\mathbb{R}^d)$ the function $T_nu$ has finite support, and in particular $T_nu = \sum_{y\in G_n}u(y)e_y$ as the sum on the right has at most a finite number of non-zero terms. Thanks to Lemma \ref{lemma:characterization of D I sub n}, there is some $\tilde L \in \mathcal{D}i_n$ such that $L = E^0_n \circ \tilde L \circ T_n$ and therefore,
  \begin{align*}
    Lu(x) = \sum \limits_{y\in G_n} (\tilde Le_y)(x)u(y) = \sum \limits_{y\in G_n-x} (\tilde Le_{x+y})(x)u(x+y),\;\forall\;x\in G_n. 
  \end{align*}
  Then, defining $K_L(x,y) = (\tilde Le_{x+y})(x)$ for $x,y\in G_n$  the identity \eqref{equation:L in DI_n formula} follows.
\end{remark}

For the rest of this section we analyze the operators $I_n$ and the sets $\mathcal{D}I_n$ and obtain in the limit a min-max formula for $I_n$.  We shall focus on operators satisfying Assumption \ref{assumption:coefficient regularity}. As we see below this property is inherited --to some extent-- by the operators $I_n$, and by any operator $L \in \mathcal{D}I_n$, this fact is covered in the next two propositions. In the subsections that follow, we will use the spatial regularity afforded by Assumption \ref{assumption:coefficient regularity} to show that the operators in the family $\mathcal{D}I_n$ have coefficients enjoying some regularity, which in the limit yields regular coefficients.

\begin{proposition}\label{proposition:coefficient regularity inherited by In}
  Let $I$ be Lipschitz and satisfy Assumption \ref{assumption:coefficient regularity}. Let $x_1,x_2 \in G_n$ and $h = x_1 - x_2$, and $r\geq 2^{4-n}$. Then,   for any $u,v \in C^\beta_b(\mathbb{R}^d)$ we have
  \begin{align*}
    & | I_n(v+\tau_{-h}u,x_1)-I_n(v,x_1) -\left ( I_n(v+u,x_2)-I_n(v,x_2)\right ) |\\
    & \leq \omega(|h|)C(2r)\left ( \|u\|_{C^\beta(B_{4r}(x_2))} + \|u\|_{L^\infty(\mathcal{C}B_r(x_2))}\right ).
  \end{align*}
  where $\omega(\cdot)$ is the modulus of continuity and $C(\cdot)$ the function given by Assumption \ref{assumption:coefficient regularity}.
\end{proposition}

\begin{proof}
  Observe that 
  \begin{align*}
    I_n(v+\tau_{-h}u,x_1)-I_n(v,x_1) = I(\pi_n^\beta v + \pi_n^\beta (\tau_{-h}u),x_1) - I_n(\pi_n^\beta,x_1),
  \end{align*}
  and recall that Proposition \ref{proposition:Whitney extension translation invariance} says that $\pi_n^\beta (\tau_{-h}u) = \tau_{-h}\pi_n^\beta (u)$ when $G_n+h = G_n$.

  Therefore, applying the bound in Assumption \ref{assumption:coefficient regularity} with $\tfrac{3}{2}r$,
  \begin{align*}
    & | I_n(v+\tau_{-h}u,x_1)-I_n(v,x_1) -\left ( I_n(v+u,x_2)-I_n(v,x_2)\right ) |\\
    & = |  I(\pi_n^\beta v +  \tau_{-h}(\pi_n^\beta u),x_1) - I_n(\pi_n^\beta,x_1)-\left ( I(\pi_n^\beta v+\pi_n^\beta u,x_2)-I(\pi_n^\beta v,x_2)\right ) |\\	
    & \leq \omega(|x_1-x_2|)C(3r/2)\left ( \|\pi_n^\beta u\|_{C^\beta(B_{3r}(x))} + \|\pi_n^\beta u\|_{L^\infty(\mathcal{C}B_{3r/2}(x))}\right ).
  \end{align*}
  Now, provided $r \geq 2^{4-n}$, we have
  \begin{align*}
    \|\pi_n^\beta u\|_{C^\beta(B_{3r}(x))} & \leq C \|u\|_{C^\beta(B_{4r}(x))},\\	  
    \|\pi_n^\beta u\|_{L^\infty(\mathcal{C}B_{3r/2}(x))} & \leq C\|u\|_{L^\infty(\mathcal{C}B_{r}(x))}, 
  \end{align*}
  the proposition follows.
  
\end{proof}

\begin{proposition}\label{proposition:coefficient regularity for L} 
  Let $I$ be Lipschitz and satisfy Assumption \ref{assumption:coefficient regularity}. Given $L\in \mathcal{D}I_n$, $x_1,x_2 \in G_n$, $r\geq 2^{4-n}$ and $u \in C^\beta_b(\mathbb{R}^d)$, we have the inequality
  \begin{align}\label{eqn:coefficient regularity}
    |L(\tau_{-h}u ,x_1)-L(u,x_2)| \leq \omega(|h|) C(2r) \left (\|u\|_{C^\beta(B_{4r}(x_2))}+ \|u\|_{L^\infty(\mathcal{C}B_{r}(x_2))}  \right ).
  \end{align}
  Here, $h = x_1-x_2$ and $\omega(\cdot)$ and $C(\cdot)$ are given by Assumption \ref{assumption:coefficient regularity}.
\end{proposition}

\begin{proof}	
  Consider any $v \in C^\beta_b(\mathbb{R}^d)$ such that $I_n$ is differentiable at $v$ with derivative $L$. Then,
  \begin{align*}
    L(\tau_{-h}u ,x_1) & = \lim \limits_{s\to 0} \frac{1}{s} \left ( I_n(v+s \tau_{-h}u,x_1)-I_n(v,x_1)\right ),\\
    L(u ,x_2) & = \lim \limits_{s\to 0} \frac{1}{s} \left ( I_n(v+s u,x_2)-I_n(v,x_2)\right ).	
  \end{align*}
  By Proposition \ref{proposition:coefficient regularity inherited by In}, we have
  \begin{align*}
     & | L(\tau_{-h}u ,x_1) - L(u ,x_2) | \\
	& = \limsup \limits_{s\to 0} \frac{1}{s} \left | I_n(v+s \tau_{-h}u,x_1)-I_n(v,x_1) - ( I_n(v+s u,x_2)-I_n(v,x_2))\right |,\\
	& \leq \omega(|h|)C(2r) \limsup \limits_{s\to 0} \frac{1}{s} \left ( \|s u\|_{C^\beta(B_{2r}(x))}  + \|s u\|_{L^\infty(\mathcal{C} B_{r}(x))} \right ),\\
	& = \omega(|h|)C(2r) \left ( \| u\|_{C^\beta(B_{2r}(x))}  + \| u\|_{L^\infty(\mathcal{C} B_{r}(x))} \right ).
  \end{align*}
  This proves the desired inequality for those $L \in \mathcal{D}I_n$ which happen to be the derivative of $I_n$ at a point of differentiability. This property is clearly preserved under limits and convex combinations, so it follows any $L \in \mathcal{D}I_n$ has the desired property.
\end{proof}

The following proposition is directly related to Proposition \ref{proposition:touching from above function with decay at a point}.
\begin{proposition}\label{proposition:touching from above function with decay at point operator estimate}
  Assume that $I$ is Lipschitz and satisfies Assumption \ref{assumption:GCP}.  For $f\in C^\infty_c(\mathbb{R}^d)$ let $w(x) = f(x-x_0)\eta_0(|x-x_0|^\beta)$ with $\eta_0$ as in \eqref{equation:eta_0 approximation to min of 1 and y squared}, then
  \begin{align*}
    I(\pi_n^\beta u + \pi_n^\beta w,x)-I(\pi_n^\beta u,x) \leq C\|f\|_{L^\infty}.
  \end{align*}
  If instead we have $w(x) = f(x-x_0)\eta_0(|x-x_0|^{\beta_0})$ with $f$ non-negative and some $\beta_0>\beta$, then 
  \begin{align*}
    I(\pi_n^\beta u + \pi_n^\beta w,x)-I(\pi_n^\beta u,x) \geq -C\|f\|_{L^\infty} h_n^{\gamma},
  \end{align*}
  for some constant $\gamma= \gamma(\beta_0,\beta) \in (0,1)$.

\end{proposition}

\begin{proof}
  We apply Proposition \ref{proposition:touching from above function with decay at a point}, and we have with $\hat R_{\beta, n, w, x_0}$ from the same proposition, we have
  \begin{align*}
    \pi_n^\beta w(x) \leq \hat w(x):= C\|f\|_{L^\infty}\left ( \eta_0(|x-x_0|^\beta) + \hat R_{\beta, n,w,x_0}(x) \right ),\;\forall\;x\in\mathbb{R}^d,
  \end{align*}
  with equality holding for $x=x_0$. It follows that $\pi_n^\beta u + \pi_n^\beta w$ is touched from above at $x_0$ by $\pi^\beta_n u+ \hat w$. Then, since $I(\cdot,x)$ has the GCP,
  \begin{align*}
    I(\pi_n^\beta u + \pi_n^\beta w,x) \leq I(\pi_n^\beta u + \hat w,x)
  \end{align*}
  This means that
  \begin{align*}
    I(\pi_n^\beta u + \pi_n^\beta w,x_0)-I(\pi_n^\beta u,x_0) \leq I(\pi_n^\beta u + \hat w,x_0)-I(\pi_n^\beta u,x_0) \leq C\|\hat w\|_{C^\beta}.
  \end{align*}
  Since $\|\hat w\|_{C^\beta} = \|f\|_{L^\infty} \| \eta_0(|\cdot-x_0|^\beta) + \hat R_{\beta, n,w,x_0}\|_{C^\beta} \leq C\|f\|_{L^\infty}$ the first inequality is proved. For the second inequality, we apply Remark \ref{remark:Whitney extension is almost order preserving L infinity version} directly, and use that $I$ has the GCP to conclude that 
  \begin{align*}
    I(\pi_n^\beta u + \pi_n^\beta w + \hat R_{\beta_0,n,w,x_0},x_0) \geq I(\pi_n^\beta u,x_0).
  \end{align*}
  Then, using the Lipschitz property of $I$ we conclude that
  \begin{align*}
    I(\pi_n^\beta u + \pi_n^\beta w,x_0) - I(\pi_n^\beta u,x_0) \geq -C\| \hat R_{\beta_0,n,w,x_0}\|_{C^{\beta}} \geq -Ch_n^{\gamma} \|f\|_{L^\infty},
  \end{align*}
  where we used that $|w(x)| \leq C\|f\|_{L^\infty}\min\{1,|x-x_0|^{\beta_0}\}$ and Remark \ref{remark:Whitney extension is almost order preserving L infinity version} to obtain the last inequality.
  
\end{proof}

\begin{proposition}\label{proposition:tightness bound for I_n}
  Let $I$ be Lipschitz and satisfy Assumption \ref{assumption:tightness bound}. Let $R\geq 1$ and $w\in C^\beta_b(\real^d)$ with $w\equiv 0$ in $B_{3R}(x_0)$, then for any $x \in \cap B_{R}(x_0)$ we have
  \begin{align*}
    |I(\pi_n^\beta u + \pi_n^\beta w,x)-I(\pi_n^\beta u,x)| \leq \rho(R)\|w\|_{L^\infty(\mathbb{R}^d)},
  \end{align*}
  where $\rho$ is the rate coming from Assumption \ref{assumption:tightness bound}.
  
\end{proposition}

\begin{proof}
  If $w\equiv 0$ in $B_{3R}(x_0)$, then $\pi^\beta_n \equiv 0$ in $B_{2R}(x_0)$. In other words, $\pi_n^\beta u$ and $\pi_n^\beta u + \pi_n^\beta w$ are identically equal in $B_{2R}(x_0)$. Therefore, Assumption \ref{assumption:tightness bound} says that 
  \begin{align*}
    |I(\pi_n^\beta u + \pi_n^\beta w,x)-I(\pi_n^\beta u,x)| \leq \rho(R)\|\pi_n^\beta  w\|_{L^\infty(\mathbb{R}^d)},\;\;\forall x \in B_{R}(x_0).
  \end{align*}
  By Proposition \ref{proposition:touching from above function with decay at a point},  $\|\pi_n^\beta  w\|_{L^\infty(\mathbb{R}^d)} \leq \|w\|_{L^\infty(\mathbb{R}^d)}$, the proposition is proved.

\end{proof}

%%%%%%%%%%%%%%%%%%%%%%%%%%%%%%%%%
%%%%%%%%%%%%%%%%%%%%%%%%%%%%%%%%%
\subsection{Properties of $\mathcal{D}I_n$}

For each $L\in \mathcal{D}I_n$ and $x\in G_n$ we define a Borel measure $\mu_L(x,dy)$ (which is possibly signed) as follows
\begin{align}\label{equation:definition of mu sub L}	
  \mu_L(x,dy) := \sum \limits_{y \in G_n\setminus \{0\}} K_L(x,y) \delta_{x+y}.	  
\end{align}
where $K_L(x,y)$ is as in Remark \ref{remark:structure of D little i n kernel}. From its definition, it is immediate that given $\phi \in C^{\beta}$ and $x\in G_n$ then
\begin{align*}
  L(\phi,x) = \int_{\mathbb{R}^d} \phi(x+y)\;d\mu_L(x,dy).
\end{align*}

\begin{proposition}\label{proposition:DI_n Borel measures integrability}
Assume that $I$ is Lipschitz and satisfies Assumption \ref{assumption:GCP}.  For each $L\in \mathcal{D}I_n$ and $x\in G_n$, and $\eta_0(t)$ the function in \eqref{equation:eta_0 approximation to min of 1 and y squared},
  \begin{align*}
    \sup \limits_{n}\sup \limits_{x\in G_n} \int_{\mathbb{R}^d} f(y)\eta_0(|y|^{\beta})\;\mu_L(x,dy) & \leq C\|f\|_{L^\infty},\;\;\forall\;f\in C^\infty_c(\mathbb{R}^d).
  \end{align*}	  
\end{proposition}

\begin{proof}
  Fix $x_0 \in G_n$. Let us assume first that $\beta \neq 1$. Let $w(x) = f(x-x_0)\eta_0(|x-x_0|^\beta)$, then
  \begin{align*}
    L(w,x_0) = \int_{\mathbb{R}^d} \phi(y) \eta_0(|y|^\beta) \;\mu_L(x_0,dy).
  \end{align*}	  
  Therefore it suffices to show there is a universal constant such that  
  \begin{align*}
    L(w,x_0) \leq C\|f\|_{L^\infty},\;\;\forall\;L\in\mathcal{D}I_n.
  \end{align*}
  Let us prove this when $L$ arises as the derivative of $I_n$ at some $v\in C^\beta_b$, namely, that
  \begin{align*}
    L(\phi,x_0) = \lim \limits_{s\to 0} (I_n(v+s\phi,x_0) -I_n(v,x_0))/s.
  \end{align*}	  
  In this case, we can apply Proposition \ref{proposition:touching from above function with decay at point operator estimate} to the expression on the right and conclude that
  \begin{align*}
    \lim \limits_{s\to 0} (I_n(v+sw,x_0) -I_n(v,x_0))/s \leq C\|f\|_{L^\infty},
  \end{align*}	  
  where we used that when $\beta \neq 1$ the function $\eta_0(|\cdot-x_0|^\beta)$ belongs to $C^\beta_b(\mathbb{R}^d)$ and the norm $\|\eta_0(|\cdot-x_0|^\beta)\|_{C^\beta}$ is bounded in terms of $\beta,d,$ and the function $\eta_0$. This the desired estimate for such $L$. Since this property is clearly preserved under limits and convex combinations, it follows that the property holds for all elements of $\mathcal{D}I_n$.
  
  The case $\beta = 1$ proceeds similarly, except one first fixes $\varepsilon \in (0,1)$ and considers the function $\eta_0(|x-x_0|^{\beta+\varepsilon})$ instead. After proceeding as in the previous case, we obtain the estimate
  \begin{align*}
    \int_{\mathbb{R}^d} f(y)\eta_0(|y|^{\beta+\varepsilon})\;\mu_L(x_0,dy) & \leq C\|f\|_{L^\infty},
  \end{align*}	  
  for every $L \in \mathcal{D}I_n$ and $x_0 \in G_n$. The constant $C$ is independent of $\varepsilon \in (0,1)$, since $\|\eta_0(|\cdot-x_0|^\beta)\|_{C^1}$ is independent of $\varepsilon$ when $\varepsilon>0$. Letting $\varepsilon \searrow 0$ for the integral on the left (and using the special form of $\mu_{L}(x_0,dy)$) one obtains the estimate in the case $\beta =1$.

\end{proof}

\begin{proposition}\label{proposition:DI_n Borel measures negative part}
Assume that $I$ is Lipschitz and satisfies Assumption \ref{assumption:GCP}.  Let $f\in C^\infty_c(\mathbb{R}^d)$ be a non-negative function. There is a constant $C=C(I,d,\beta,\beta_0)$ such that given $\beta_0>\beta$ then for each $L\in \mathcal{D}I_n$ and $x\in G_n$,
  \begin{align*}
    \inf \limits_{n} \inf \limits_{x \in G_n} \int_{\mathbb{R}^d} f(y)\eta_0(|y|^{\beta_0})\;\mu_L(x,dy) & \geq -Ch_n^{\gamma}\|f\|_{L^\infty}.
  \end{align*}	  
  As before, $\eta_0$ is the function in \eqref{equation:eta_0 approximation to min of 1 and y squared}, and $\gamma = \gamma(\beta,\beta_0)$.
\end{proposition}

\begin{proof}
  As in the proof of the previous proposition, we note that if $x_0 \in G_n$, $w(x) := f(x-x_0)\eta_0(|x-x_0|^{\beta_0})$, and $L \in \mathcal{D}I_n$, then
  \begin{align*}
    L(w,x_0) = \int_{\mathbb{R}^d} f(y) \eta_0(|y|^\beta) \;\mu_L(x_0,dy).
  \end{align*}
  As in the previous Proposition, it suffices to show that $L(w,x_0) \geq -C\|f\|_{L^\infty} h_n^{\gamma}$, and from $\mathcal{D}I_n$'s definition, it suffices to show this for those $L's$ in $\mathcal{D}I_n$ which are the derivative of $I_n$ at some $u \in C^{\beta}_b(\mathbb{R}^d)$. In this case, given that $f\geq 0$, we may apply the second part of Proposition \ref{proposition:touching from above function with decay at point operator estimate} to obtain
  \begin{align*}
    L(w,x_0) = \lim\limits_{s\to 0} \frac{I(\pi_n^\beta u + \pi_n^\beta(sw),x_0)-I(\pi_n^\beta u,x_0)}{s} \geq \lim\limits_{s\to 0} -\frac{C\|s f\|_{L^\infty }h_n^\gamma}{s} = -Ch_n^\gamma\|f\|_{L^\infty},
  \end{align*}	  
  and the proposition is proved.
\end{proof}

Let us recall the function
\begin{align*}
  P_{\phi,\eta,u,x}(\cdot) = u(x)+\phi(\cdot-x)(\nabla u(x),\cdot-x)+\tfrac{1}{2}\eta(\cdot-x)(D^2u(x)(\cdot-x),(\cdot-x)).
\end{align*}
In this section we introduce a variation on this function. This modification takes into account the geometry of the grid $G_n$ as well as the regularity exponent $\beta$, and will be used in a way analogous to the previous section. 
\begin{align*}
  P_{\phi,\eta,u,x}^{(n)}(\cdot) = \left \{ \begin{array}{lr}
      u(x) & \textnormal{ if } \beta \in(0,1),\\
      u(x) + \phi(\cdot-x)((\nabla_n)^1 u(x),\cdot-x) & \textnormal{ if } \beta \in [1,2),\\
      u(x) + \phi(\cdot-x)((\nabla_n)^1 u(x),\cdot-x)+ \tfrac{1}{2} \eta(\cdot-x)((\nabla_n)^2u(x)(\cdot-x),\cdot-x) & \textnormal{ if } \beta \in [2,3).	  
  \end{array}\right.
\end{align*}
Associated with this, we introduce functions in $G_n$ taking (respectively) scalar, vector, and matrix values.  

First, some notation. To functions $\eta,\phi \in \mathcal{S}$ we associate the following family of functions
\begin{align*}
  \phi_{i}(y) = \phi(y)y_i,\;i=1,\ldots,d,\;\;\eta_{ij}(y) = \eta(y)y_iy_j,\;i,j=1,\ldots,d.
\end{align*}

Then, for $L\in\mathcal{D}I_n$ and $\eta,\phi \in \mathcal{S}$ we define a symmetric matrix $A_{L,\eta}$, a vector $B_{L,\phi}$, and a scalar $C_{L}$. These are functions in $G_n$ defined by the formulas,
\begin{align}
  {(A_{L,\eta}(x))}_{ij} & = L(\tau_{-x}\eta_{ij},x), \;i,j=1,\ldots,d, \label{eqn:AsubL definition}\\
  {(B_{L,\phi}(x))}_i & = L(\tau_{-x}\phi_i,x), \;i=1,\ldots,d, \label{eqn:BsubL definition}\\
  C_L(x) & = L(1,x). \label{eqn:CsubL definition}.
\end{align}
The functions $A_{L,\eta},B_{L,\phi},C_{L},$ and $\mu_L$ give us a representation for $L(u,x)$ for $x\in G_n$.

\begin{proposition}\label{proposition:L in DIn first representation formula}
  Assume that $I$ is Lipschitz. Let $L\in \mathcal{D}I_n$, then for $\beta\in[2,3)$ and $u\in C^\beta_b(\real^d)$ we may write it as 
  \begin{align*}
    L(u,x) & = C_L(x)u(x)+B_{L,\phi}(x)\cdot (\nabla_n)^1 u(x)+\tr(A_{L,\eta}(x)(\nabla_n)^2u(x))\\
	  & \;\;\;\;+\int_{\mathbb{R}^d} u(x+y)-P_{\phi,\eta,u,x}^{(n)}(x+y) \;\mu_L(x,dy).  
  \end{align*}	  
  For $\beta \in [1,2)$
  \begin{align*}
    L(u,x) & = C_L(x)u(x)+B_{L,\phi}(x)\cdot (\nabla_n)^1 u(x)+\int_{\mathbb{R}^d} u(x+y)-P_{\phi,\eta,u,x}^{(n)}(x+y) \;\mu_L(x,dy),
  \end{align*}	  
  and for $\beta \in [0,1)$
  \begin{align*}
    L(u,x) & = C_L(x)u(x)+\int_{\mathbb{R}^d} u(x+y)-u(x)\;\mu_L(x,dy).
  \end{align*}	  
  
\end{proposition}

\begin{proof}
  We do the case $\beta \geq 2$ explicitly, as the others are identical. Let us compute $L(u,x)$ by adding and subtracting $L(P_{\phi,\eta,u,x}^{(n)},x)$,
  \begin{align*}
    L(u,x) & = L(u-P_{\phi,\eta,u,x}^{(n)},x)+L(P^{(n)}_{\phi,\eta,u,x},x).
  \end{align*}
  From Remark \ref{remark:structure of D little i n kernel}, \eqref{equation:definition of mu sub L}, we have that
  \begin{align*}
    L(u-P_{\phi,\eta,u,x}^{(n)},x) = \int_{\mathbb{R}^d} u(x+y)-P_{\phi,\eta,u,x}^{(n)}(x+y)\;\mu_L(x,dy)
  \end{align*}
  As for the other term, we observe that
  \begin{align*}
    L(P_{\phi,\eta,u,x}^{(n)},x) & = u(x)L(1,x) + \sum \limits_{i=1}^d (\nabla_1u)^n_i(x) L(\tau_{-x}\phi_i,x)+ \tfrac{1}{2}\sum \limits_{i,j=1}^d (\nabla_n)^2_{ij}u(x)L(\tau_{-x}\eta_{ij},x).	
  \end{align*}
  Rewriting the terms on the right and gathering the terms, we conclude that
  \begin{align*}
    L(P_{\phi,\eta,u,x}^{(n)},x) & = C_L(x)u(x) + (B_{L,\phi}(x),(\nabla_n)^1u(x)) + \tr(A_{L,\eta}(x)(\nabla_n)^2u(x)).	
  \end{align*}
  The remaining cases of $\beta$ follow from the corresponding definition of $P^{(n)}_{\phi,\eta,u}$ in those cases.

\end{proof}

The next two propositions say that the terms appearing Proposition \ref{proposition:L in DIn first representation formula} satisfy a uniform continuity in $G_n$. The first refers to the measure $\mu_L$.

\begin{proposition}\label{proposition:TV norm continuity and tightness estimate for discrete Levy measures}
  Assume $I$ satisfies Assumptions \ref{assumption:GCP}, \ref{assumption:tightness bound}, and \ref{assumption:coefficient regularity}, as stated for $C^\beta_b(\real^d)$. Let $L \in D I_n$, $x_1,x_2 \in G_n$, and $r\geq 2^{4-n}$. There is a constant $C(r)$ such that for any $\zeta \in C_c(\mathbb{R}^d)$ such that $\zeta \equiv 0$ in $B_r$,
  \begin{align*}
    \left |\int_{\mathcal{C}B_r} \zeta(y)\;\mu_L(x_1,dy)-\int_{\mathcal{C} B_r} \zeta(y)\;\mu_L(x_2,dy) \right | \leq C(r)\|\zeta\|_{L^\infty}\omega(|x_1-x_2|),
  \end{align*}
  where $\omega$ is the modulus from Assumption \ref{assumption:coefficient regularity}. In particular, 
  \begin{align*}
    \left \| \mu_L(x_1,dy)-\mu_L(x_2,dy)  \right \|_{\textnormal{TV}(\mathcal{C}B_r )} \leq C(r)\omega(|x_1-x_2|).
  \end{align*}
  On the other hand, if $\zeta \in C^0(\mathbb{R}^d)$ is such that $\zeta\equiv 0$ in $B_{3R}(0)$ for some $R>1$, then for any $x_0 \in G_n$ we have
  \begin{align*}
    \int_{\mathbb{R}^d} \zeta(y) \;\mu_L(x_0,dy) \leq \rho(R)\|\zeta\|_{L^\infty(\mathbb{R}^d)},
  \end{align*}
  where $\rho(\cdot)$ is the function from Assumption \ref{assumption:tightness bound}.
\end{proposition}
 
\begin{proof}
  From the fact that $\tau_{-x_1}\zeta$ and $\tau_{-x_2}\zeta$ vanish in, respectively, $B_r(x_1)$ and $B_r(x_2)$, we have
  \begin{align*}
     L( \tau_{-x_1}\zeta,x_1) - L(\tau_{-x_2}\zeta,x_2) & = \int_{\mathbb{R}^d} \zeta(y)\;d\mu(x_1,dy)-\int_{\mathbb{R}^d} \zeta(y)\;d\mu(x_2,dy) \\
	 &=  \int_{\mathcal{C}B_r} \zeta(y)\;d\mu(x_1,dy)-\int_{\mathcal{C}B_r} \zeta(y)\;d\mu(x_2,dy).
  \end{align*}
  Since $\zeta \equiv 0$ in $B_r$, Proposition \ref{proposition:coefficient regularity for L} says that, as long as $r\geq 2^{4-n}$
  \begin{align*}
    & \left | \int_{\mathcal{C}B_r} \zeta(y)\;d\mu(x_1,dy)-\int_{\mathcal{C}B_r}  \zeta(y)\;d\mu(x_2,dy) \right | \leq \omega(|x_1-x_2|)C(r)\|\zeta\|_{L^\infty(\mathcal{C} B_{r})}.	
  \end{align*}
  This proves the first estimate, for the second one, fix $\zeta$ and $x_0 \in G_n$, and define $w(x) = \tau_{-x_0}\zeta$, then 
  \begin{align*}
    L(w,x_0) = \int_{\mathbb{R}^d} \zeta(y) \;\mu_L(x_0,dy).
  \end{align*}
  Therefore, as before, it suffices for us to bound $L(w,x_0)$ for every $L \in \mathcal{D}I_n$, and from the definition of $\mathcal{D}I_n$ it suffices to prove the bound for those $L$ such that $L = DI_n(v)$ at some $v$. In this case, Proposition \ref{proposition:tightness bound for I_n} says that
  \begin{align*}
    L(w,x_0) = \lim\limits_{s\to 0} \frac{1}{s} (I_n(v+sw,x_0)-I_n(v,x_0)) \leq \rho(R)\|w\|_{L^\infty(\mathbb{R}^d)} =  \rho(R)\|\zeta\|_{L^\infty(\mathbb{R}^d)}
  \end{align*}
\end{proof}

The following notation will be useful in what follows,
\begin{align*}
  \alpha(r,\eta) & := C(2r)\left ( \max\limits_{1\leq i,j\leq d}\|\eta_{ij} \|_{C^\beta(B_{4r})} +\max\limits_{1\leq i,j\leq d}\|\eta_{ij} \|_{L^\infty(\mathcal{C}B_r)}\right ),\\
  \beta(r,\phi) & := C(2r)\left ( \max\limits_{1\leq i\leq d}\|\phi_i \|_{C^\beta(B_{4r})} +\max\limits_{1\leq i\leq d}\|\phi_i \|_{L^\infty(\mathcal{C}B_r)}\right ),	 	
\end{align*}	  
where $C(r)$ is as in Assumption \ref{assumption:coefficient regularity} (see also Proposition \ref{proposition:coefficient regularity inherited by In}).

\begin{proposition}\label{proposition:discrete coefficient regularity for L in DI_n}
 Assume $I$ satisfies Assumptions \ref{assumption:GCP}, \ref{assumption:tightness bound}, and \ref{assumption:coefficient regularity}, as stated for $C^\beta_b(\real^d)$. Let $L \in \mathcal{D}I_n$, $r\geq 2^{4-n}$, and $x_1,x_2 \in G_n$, then
  \begin{align*}
    |A_{L,\eta}(x_1)-A_{L,\eta}(x_2)| & \leq \alpha(r,\eta)\omega(|x_1-x_2|),\\
    |B_{L,\phi}(x_1)-B_{L,\phi}(x_2)| & \leq \beta(r,\phi)\omega(|x_1-x_2|),\\
    |C_{L}(x_1)-C_{L}(x_2)| & \leq C(r)\omega(|x_1-x_2|).
  \end{align*}
\end{proposition}

\begin{proof}
  Fix $x_1,x_2 \in G_n$ and let $h = x_2-x_1$. Applying Proposition \ref{proposition:coefficient regularity for L} to $x=x_1$ and $h$, with the functions $1$, $\phi_i$, and $\eta_{ij}$, we see that for $r\geq 2^{4-n}$
  \begin{align*}
    |L(\tau_{-x_2}\eta_{ij},x_2)-L(\tau_{-x_1}\eta_{ij},x_1)| & \leq \alpha(\eta,r)\omega(|x_1-x_2|),\\
    |L(\tau_{-x_2}\phi_i,x_2)-L(\tau_{-x_1}\phi,x_1)| & \leq \beta(\phi,r)\omega(|x_1-x_2|),\\	
    |L(1,x_2)-L(1,x_1)| & \leq C\omega(|x_1-x_2|).
  \end{align*}	  
  These inequalities respectively amount to the stated estimate for $A_{L,\eta}$, $B_{L,\phi}$, and $C_L$.
  
\end{proof}

%%%%%%%%%%%%%%%%%%%%%%%%%%%%%%%%%
%%%%%%%%%%%%%%%%%%%%%%%%%%%%%%%%%
\subsection{Properties of $\mathcal{D}_I$ }

Now, we define the set $\mathcal{D}_I$, which plays the role the Clarke differential played for $I_n$ (we recall that c.h. stands for ``convex hull'').
\begin{align}\label{equation:DI definition}
  \mathcal{D}_I := \textnormal{c.h.}\{ L \mid \exists \{L_{n_k}\}, n_k\to\infty,\; L_{n_k} \in \mathcal{D}I_{n_k} \textnormal{ s.t } L(u,\cdot) = \lim\limits_{k} L_{n_k}(u,\cdot)\;\forall\;u \}.
\end{align}

\begin{remark}\label{remark:Clarke differential in the limit}
  We would like to note a point about notation and definitions, namely why above we have $\mathcal{D}_I$ with $I$ as a subscript. This is to avoid confusion (or perhaps, to promote it) by distinguishing it from the generalized derivative in the sense of Clarke from Definition \ref{definition:Clarke differential}. The objects are closely related, and in fact one would hope that $\mathcal{D}_I = \mathcal{D}I$, but we are not concerned with whether this is actually the case as the above definition works for our purposes.
\end{remark}

The following is an important Lemma that says --among other things-- that $\mathcal{D}_I$ is non-empty.

\begin{lemma}\label{lemma:DI_n sequences subconverge}
 Assume $I$ satisfies Assumptions \ref{assumption:GCP}, \ref{assumption:tightness bound}, and \ref{assumption:coefficient regularity}, as stated for $C^\beta_b(\real^d)$.  Given a sequence $n_k \to\infty$ and operators $L_{n_k}$ with $L_{n_k} \in \mathcal{D}I_{n_k}$ for every $k$, and $\phi,\eta \in \mathcal{S}$ we have the following
  \begin{enumerate}
    \item There is a subsequence $\bar n_k$ and functions $A(x),B(x),$ and $C(x)$ defined on $\mathbb{R}^d$ and taking values respectively in $\mathbb{S}(d)$, $\mathbb{R}^d$, and $\mathbb{R}$, such that if $x \in G_n$ for some $n$ then we have the convergence
    \begin{align*}
      A_{L_{\bar n_k},\eta}(x) \to A(x),\; B_{L_{\bar n_k},\phi}(x) \to B(x),\; C_{L_{\bar n_k}}(x) \to C(x).
    \end{align*}
    \item There is a function $\mu(x)$ in $\mathbb{R}^d$, taking values on the space of L\'evy measures in $\mathbb{R}^d$, such that for every $r>0$, and every $x$ as before we have the convergence
    \begin{align*}
      \lim \limits_{k\to \infty } \| \mu_{L{\bar n_k}}(x)-\mu(x)\|_{\textnormal{TV}(\mathcal{C}B_r)} = 0.	
    \end{align*}	
    \item The functions $A,B,C,$ all have a modulus of continuity $C\omega(2(\cdot))$, while for each $r>0$ we have the estimate,
    \begin{align}\label{equation:continuity estimate mu_L}	
      \|\mu(x_1)-\mu(x_2)\|_{\textnormal{TV}(\mathcal{C}B_r)} \leq C(r)\omega(2|x_1-x_2|).
    \end{align}
    \item If we define $L$ by
    \begin{align*}
      L(u,x) & := \textnormal{tr}(A(x)D^2u(x))+B(x)\cdot \nabla u(x) + C(x)u(x)\\
	    & \;\;\;\;+ \int_{\mathbb{R}^d} u(x+y)-P_{\phi,\eta,u,x}(x+y) \;\mu(x,dy)	
    \end{align*}	
    Then, $L \in \mathcal{D}_I$.
    \item Moreover, if $\beta<2$, then we have $A(x)\equiv0$. Furthermore, if $\beta<1$ then $B(x)\equiv 0$ and $L$ takes the form
    \begin{align*}
      L(u,x) = C(x)u(x) + \int_{\mathbb{R}^d} u(x+y)-u(x) \;\mu(x,dy).	
    \end{align*}

  \end{enumerate}
 
\end{lemma}

\begin{proof}
Let us fixe $\eta$ and $\phi$.	First of all, we invoke Proposition \ref{proposition:L in DIn first representation formula} to obtain the collection of $A_{L_{n_k},\eta}$, $B_{L_{n_k},\phi}$, $C_{L_{n_k}}$, and $\mu_{L_{n_k}}$.  Furthermore, already as a result of Proposition \ref{proposition:L in DIn first representation formula}, we have item (5) of the lemma.

  \emph{Step 1.} (Extension)  We have a sequence of functions defined on varying, monotone increasing sets $G_n$. One way to show they converge (along a subsequence) to a function in $\mathbb{R}^d$ is by extending them to all of $\mathbb{R}^d$ and check whether the resulting sequences are pre-compact. 
  
  With this idea in mind, for each $n\in\mathbb{N}$ we apply the Whitney extension to $A_{L_n,\eta}$, $B_{L_n,\eta}$, $C_{L_n,\eta}$,
  \begin{align*}
    \hat A_{L_n,\eta}(x) := E_n^0 (A_{L_n,\eta})(x),\; \hat B_{L_n,\phi}(x) := E_n^0 (B_{L_n,\phi})(x),\;\hat C_{L_n}(x) := E_n^0 (C_{L_n})(x).
  \end{align*}
  We repeat the same for $\mu_{L_{n}}$, resulting in a map $\hat \mu_{L_n}$ from $\mathbb{R}^d$ to the space of L\'evy measures, given by the formula
  \begin{align*}
    \hat \mu_{L_n}(x,dy) = \sum \limits_{k=1}^\infty \phi_{n,k}(x)\mu(x_k,dy),
  \end{align*}	
  where $\{\phi_{k}\}_k$ is the partition of unity from Proposition \ref{proposition:partition of unity properties}. The functions $\hat A_{L_{n},\eta}$, $\hat B_{L_{n},\phi}$, and $\hat C_{L_{n}}(x)$ all have modulus of continuity $C\omega(2(\cdot))$, thanks to Proposition \ref{proposition:discrete coefficient regularity for L in DI_n} and the properties of the Whitney extension operator, see \cite[Chapter VI, Theorem 3]{Stei-71}. The same proof from reference \cite{Stei-71} can be applied with minor modifications to show that for every $r>0$ we have
  \begin{align*}
    \|\hat \mu_{L_n}(x_1)-\hat \mu_{L_n}(x_2)\|_{\textnormal{TV}(\mathcal{C}B_r)} \leq C(r) \omega(2|x_1-x_2|).
  \end{align*}
  Furthermore, for every $x$, by Proposition \ref{proposition:TV norm continuity and tightness estimate for discrete Levy measures},
  \begin{align*}
    |\hat \mu_{L_n}(x)|(\mathcal{C}B_R) \leq \rho(R),
  \end{align*}
  where $\rho(R) \to 0$ as $R\to \infty$. This shows that for each $r>0$, the functions $\{ \hat \mu_{L_n}\mid_{\mathcal{C}B_r}\}_n$ are an equicontinuous family of functions taking values inside the space of measures $\nu$ which are supported in $\mathcal{C}B_r$ and such that $\nu(\mathcal{C}B_R) \leq \rho(R)$ for all $R\geq r$. This space, equipped with the total variation distance, is a compact metric space.
  
  \emph{Step 2.} (Cantor diagonalization) We now use a standard Cantor diagonalization argument to obtain locally uniform convergence along a subsequence. We construct a family nested sequences $\tilde n^m_k$ in the following recursive manner. First, $\tilde n^1_k$ is a subsequence of $n_k$ along which the functions converge uniformly in $B_1$ to functions $A^1(x),B^1(x)$, and $C^1(x)$) defined in $B_1$. Next, suppose that for $m \in \mathbb{N}$ we have build a nested family of sequences $\tilde n^1_k,\ldots,\tilde n^m_k$ such that the functions $A_{L_{\tilde n^m_k},\eta},\ldots$, etc converge uniformly in $B_m(0)$ to functions $A^m(x)\ldots$, etc. In this case, we choose $\tilde n^{m+1}_k$ to be a subsequence of $\tilde n^m_k$ along which $A_{L_{\tilde n^{m+1}_k},\eta},\ldots$ converge uniformly in $B_{m+1}$ to functions $A^{m+1}(x)\ldots$ and so on.  	
		
  Having constructed these $\tilde n^m_k$, we define the sequence $\tilde n_k$ as  $\tilde n_k := n^k_k$. The resulting sequences converge locally uniformly, respectively, to $A(x),B(x)$, and $C(x)$.

  \emph{Step 3.} (Cantor diagonalization continued)
  
  As noted at the end of Step 1, for every $r>0$, the sequence $\{\hat \mu_{L_{\tilde n_k}}\}_k$ is an equicontinuous family of functions taking values in a compact metric space. Therefore, we can apply the Arzela-Ascoli type theorem found in \cite[p. 202]{GreenValentine1961arzela} to obtain a subsequence $\bar n^1_k$ of $\tilde n_k$ and a measure $\mu^1$ such that
  \begin{align*}
    \lim \limits_{k\to \infty}\sup \limits_{x\in B_1} \| \hat \mu_{L_{\bar n_k^1}}(x)-\mu^1( x)\|_{\textnormal{TV}(\mathcal{C}B_{1/2})} = 0.
  \end{align*}
  Now, suppose we have repeated this $m$ times: we have $\bar n^m_k$ (a subsequence of $\bar n^{m-1}_k$), as well as a measure $\mu^m$ such that
  \begin{align*}
    \lim \limits_{k\to \infty} \sup \limits_{x\in B_m} \| \hat \mu_{L_{\bar n_k^m}}(x)-\mu^m( x)\|_{\textnormal{TV}(\mathcal{C}B_{1/2^m})} = 0.
  \end{align*}
  Then, using again the compactness theorem in \cite[p. 202]{GreenValentine1961arzela} we pick a subsequence $\bar n^{m+1}_k$ of $\bar n^m_k$ and a measure $\mu^{m+1}$ such that
  \begin{align*}
    \lim \limits_{k\to \infty} \sup \limits_{x\in B_{m+1}} \| \hat \mu_{L_{\bar n_k^m}}(x)-\mu^{m+1}( x)\|_{\textnormal{TV}(\mathcal{C}B_{1/2^{m+1}})} = 0.
  \end{align*}
  Observe that the measures $\{ \mu^{m}\}$ are such that $\mu^{m+1}_{\mid \mathcal{C}B_{1/2^m}}(x) = \mu^m(x)$ for all $x\in B_m$, which uniquely defines a direct limit measure $\mu(x)$ for each $x\in \mathbb{R}^d\setminus  \{0\}$. Letting $\bar n_k := \bar n_k^k$ we see that for every $R>0$ and $r>0$ we have
  \begin{align*}
    \lim \limits_{k\to \infty} \sup \limits_{x\in B_{R}} \| \hat \mu_{L_{\hat n_k^k}}(x)-\mu( x)\|_{\textnormal{TV}(\mathcal{C}B_{r})} = 0.
  \end{align*}
  Since $\bar n_k$ is a subsequence of $\tilde n_k$, we still have convergence of $A_{L_{\bar n_k},\eta},\ldots$ to $A(x),\ldots$. Moreover, the continuity estimates in the previous step all pass to the limit to give respective estimates for $A(x),B(x),C(x),$ and $\mu(x)$ in the respective metrics.
  
  Last but not least, we note that while $\{\mu_{L_{\bar n_k}}\}_k$ are a sequence of signed measures, their limit $\mu$ will be a measure, which follows at once from Proposition \ref{proposition:DI_n Borel measures negative part}.
  
  \emph{Step 4.} (Convergence)
  
  First, note that for fixed $u$, we have that as $n\to \infty$,
  \begin{align*}
    u(x+\cdot)-P^{(n)}_{\phi,\eta,u,x}(x+\cdot) \to u(x+\cdot)-P_{\phi,\eta,u,x}(x+\cdot) \textnormal{ in } L^\infty(\mathbb{R}^d),
  \end{align*}
  which in particular guarantees that, for every fixed $r>0$,
  \begin{align*}  
    \lim \limits_{k\to \infty} \int_{\mathcal{C}B_r}u(x+y)-P^{(n_k)}_{\phi,\eta,u,x}(x+y)\;\mu_{L_{n_k}}(x,dy) =  \int_{\mathcal{C}B_r}u(x+y)-P_{\phi,\eta,u,x}(x+y)\;\mu(x,dy). 
  \end{align*}
  Then, by the bound in Proposition \ref{proposition:DI_n Borel measures integrability}, we conclude that 
  \begin{align*}  
    \lim \limits_{k\to \infty} \int_{\mathbb{R}^d}u(x+y)-P^{(n_k)}_{\phi,\eta,u,x}(x+y)\;\mu_{L_{n_k}}(x,dy) =  \int_{\mathbb{R}^d}u(x+y)-P_{\phi,\eta,u,x}(x+y)\;\mu(x,dy). 
  \end{align*}
  Therefore, and taking into account the convergence of $\hat A_{L_{\tilde n_k},\eta}, \hat B_{L_{\tilde n_k},\phi},$ and $\hat C_{L_{\tilde n_k}}$, and with $L(u,x)$ defined as in the statement of the Lemma, $x\in G_n$, and $u\in C^{\beta}_b(\mathbb{R}^d)$, we have
  \begin{align*}
    \lim \limits_{k \to \infty} L_{\tilde n_k}(x) & = \lim\limits_{k\to \infty} \big \{ \tr(\hat A_{L_{\tilde n_k},\eta}D^2u(x)) + \hat B_{L_{\tilde n_k},\phi}\cdot \nabla u(x) + \hat C_{L_{\tilde n_k}}(x)u(x) \big \} \\
    & \;\;\;+ \lim \limits_{k\to \infty}\int_{\mathbb{R}^d} u(x+y)-P^{(\tilde n_k)}_{\phi,\eta,u,x}(x+y) \;\hat \mu_{L_{n_k}}(x,dy)\\
    & =  \tr(AD^2u(x)) + B\cdot \nabla u(x) + C(x)u(x)\\
    & \;\;\;+ \int_{\mathbb{R}^d}u(x+y)-P_{\phi,\eta,u,x}(x+y) \;\hat \mu(x,dy),	     
  \end{align*}	 
  and we conclude that $L \in \mathcal{D}_I$.

\end{proof}

It is to be expected that every $L \in \mathcal{D}_I$ satisfies the GCP, and thus, it has to be an operator of L\'evy type.  This is proved in the lemma below, and further, we show that the coefficients in the operator inherit a modulus of continuity from Assumption \ref{assumption:coefficient regularity}.

\begin{lemma}\label{lemma:L in DI representation formula first form}
 Assume $I$ satisfies Assumptions \ref{assumption:GCP}, \ref{assumption:tightness bound}, and \ref{assumption:coefficient regularity}, as stated for $C^\beta_b(\real^d)$.  Given $L \in \mathcal{D}_I$, and any $\phi, \eta \in \mathcal{S}$, the operator $L$ can be represented as
  \begin{align*}
    L(u,x) & = C_L(x)u(x)+B_{L,\phi}(x)\cdot \nabla u(x)+\tr(A_{L,\eta}(x)D^2u(x))\\
      & \;\;\;\;+ \int_{\mathbb{R}^d}u(x+y)-P_{\phi,\eta,u,x}(x+y)\;\mu(x,dy).	
  \end{align*}
 Here, $\mu_L(x,dy)$ is a L\'evy measure satisfying the continuity estimate \eqref{equation:continuity estimate mu_L}, and 
  \begin{align*}
    (A_{L,\eta})_{ij}(x) & = L( \tau_{-x}\eta_{ij},x),\\
    (B_{L,\phi})_i(x) & = L( \tau_{-x}\phi_{i},x),\\
    C_{L}(x) & = L(1,x),
  \end{align*}
  all have modulus of continuity $C\omega(2(\cdot))$.
\end{lemma}

\begin{proof}
  Fix $\phi,\eta \in \mathcal{S}$. Assume first that $L$ is the limit of a sequence $L_{n_k}$ with $L_{n_k} \in \mathcal{D}I_{n_k}$. Then, by Lemma \ref{lemma:DI_n sequences subconverge} there is a subsequence $\tilde n_k$ as well as (matrix, vector, scalar, measure)-valued functions $A,B,C$, and $\mu$, all such that 
  \begin{align*}
    C_{L_{\tilde n_k}}(x) \to C(x),\; B_{L_{\tilde n_k},\phi_k}(x) \to B(x),\; A_{L_{\tilde n_k},\eta_k}(x) \to A(x),\;\mu_{L_{\tilde n_k}}(x,dy) \to \mu(x,dy).
  \end{align*}
  and, as a result, we have
  \begin{align*}
    L(u,x) & = \tr(A(x)D^2u(x))+B(x)\cdot \nabla u(x)+C(x)u(x)\\
	  & \;\;\;\;+ \int_{\mathbb{R}^d} u(x+y)-P_{\phi,\eta,u,x}(y)\;\mu(x,dy).
  \end{align*}
  The estimate in Proposition \ref{proposition:DI_n Borel measures integrability} in the limit as $n\to \infty$ implies that
  \begin{align*}
    \int_{\mathbb{R}^d} \eta_0(|y|^{\beta})\;\mu(x,dy) \leq C,
  \end{align*}
  for some constant $C$ independent of $x$ and $L$. Meanwhile, also the $n\to \infty$ limit of the estimate in Proposition \ref{proposition:DI_n Borel measures negative part} implies that $\mu(x,dy)$ is a non-negative measure in $\mathbb{R}^d\setminus \{0\}$. The positivity of $\mu$ means that the previous estimate is equivalent to
  \begin{align*}
    \int_{\mathbb{R}^d} \min\{1,|y|^{\beta}\}\;\mu(x,dy) \leq C.
  \end{align*}
  Since $L_{\tilde n_k}(u,x) \to L(u,x)$, for every $u$, we have in particular, for $x \in \bigcup G_k$
  \begin{align*}
    (A_{L_{\tilde n_k},\eta})_{ij}(x) = L_{\tilde n_k}(\tau_{-x} \eta_{ij},x) \to L(\tau_{-x}\eta_{ij},x).
  \end{align*}
  From where it follows that $(A_{L,\eta})_{ij}(x) = L( \tau_{-x}\eta_{ij},x)$ (and thus for all $x$, by continuity), the exact same argument yields that $ (B_{L,\phi})_i(x) = L( \tau_{-x}\phi_{i},x)$, and $C_{L}(x)  = L(1,x)$, and the lemma is proved.
    
\end{proof}

Let us now simplify things by doing away with the auxiliary functions $\phi$ and $\eta$. To accomplish this, we shall make use of the auxiliary functions from Section \ref{section:Functionals with the GCP}.
\begin{align}\label{equation:phi delta and psi delta recalled}
  \phi_{\delta}(x) = \psi_{\delta,1-\delta},\;\eta_{\delta}(x) =  \psi_{\delta,\delta}(x),
\end{align}
where we recall the two-parameter of functions $\psi_{r,R}(x)$ was defined in \eqref{equation:psi sub r R}. An important property of these one-parameter families is the bound
\begin{align}\label{equation:auxiliary phi and eta functions are bounded uniformly in delta}
  \sup \limits_{\delta \in (0,1)} \{ \|\phi_{\delta}\|_{C^\beta(B_{1/2})} + \|\phi_\delta\|_{L^\infty(\mathbb{R}^d)} + \max \limits_{ij}\|\eta_{\delta}x_ix_j\|_{C^\beta(\mathbb{R}^d)} \} < \infty.
\end{align}

\begin{corollary}\label{corollary:L in DI representation formula final form}
  Assume $I$ satisfies Assumptions \ref{assumption:GCP}, \ref{assumption:tightness bound}, and \ref{assumption:coefficient regularity}, as stated for $C^\beta_b(\real^d)$. Then, any $L \in \mathcal{D}_I$ has the form, 
  \begin{align*}
    L(u,x) & = C(x)u(x)+B(x)\cdot \nabla u(x)+\tr(A(x)D^2u(x))\\
      & \;\;\;\;+ \int_{\mathbb{R}^d}u(x+y)-u(x)-\chi_{B_1(0)}(y)\nabla u(x)\cdot y\;\mu(x,dy).	
  \end{align*}
  Moreover, $A,B,$ and $C$ each have modulus of continuity $C\omega(2(\cdot))$, and for every $r>0$ and any $x_1,x_2\in\mathbb{R}^d$ we have
  \begin{align*}
    \|\mu_L(x_1)-\mu_{L}(x)\|_{\textnormal{TV}(\mathcal{C} B_r)} \leq C(r) \omega(2|x_1-x_2|).
  \end{align*}
  If $\beta<2$, then $A\equiv 0$, while if $\beta<1$ then $B \equiv 0$ and the integrand with respect to $\mu(x,dy)$ in the formula above is replaced with $u(x+y)-u(x)$.
\end{corollary}

\begin{proof}
  Take a decreasing sequence $\delta_k$ such that $\delta_k \to 0$, and let us take the functions $\phi_{\delta_k}$ and $\eta_{\delta_k}$, as defined in \eqref{equation:phi delta and psi delta recalled}. Then for each $k$, $L$ has the representation
  \begin{align*}
    L(u,x) & = C_L(x)u(x)+B_{L,\phi_{\delta_k}}(x)\cdot \nabla u(x)+\tr(A_{L,\eta_{\delta_k}}(x)D^2u(x))\\
      & \;\;\;\;+ \int_{\mathbb{R}^d}u(x+y)-P_{\phi_{\delta_k},\eta_{\delta_k},u,x}(x+y)\;\mu(x,dy),	
  \end{align*}
  where $A_{L,\eta_{\delta_k}}$, $B_{L,\phi_{\delta_k}},$ and $C_{L}$ are as in Lemma \ref{lemma:L in DI representation formula first form}. Now, $L$ satisfies the estimate
  \begin{align*}
    |L(\tau_{-x_1}(\eta_{\delta_k})_{ij},x_1)-L(\tau_{-x_2}(\eta_{\delta_k})_{ij},x_2)| \leq \alpha(1,\eta_{\delta_k})\omega(2|x_1-x_2|)
  \end{align*}
  Thanks to \eqref{equation:auxiliary phi and eta functions are bounded uniformly in delta}, it follows that $\alpha(1,\eta_{\delta_k}) \leq C$ for all $k$. It follows that $\{A_{L,\eta_{\delta_k}}\}_k$ has a uniform modulus of continuity. The same argument yields a modulus of continuity for $\{B_{L,\phi_{\delta_k}}\}_k$ and for the function $C(x)$, all given by $C\omega(2|x_1-x_2|)$, with $C$ independent of $k$ and $\omega$ being the modulus from Assumption \ref{assumption:coefficient regularity}.
  This equicontinuity means these sequences of functions are pre-compact at least when restricted to any compact subset of $\mathbb{R}^d$, by the Arzela-Ascoli theorem. Therefore, after a Cantor diagonalization argument we see that along some subsequence $m_k \to \infty$ these functions converge locally uniformly in $\mathbb{R}^d$ to functions $A(x)$, $B(x)$, respectively.  Of course, the functions $A,B,$ and $C$ all inherit the modulus of continuity $C\omega(2(\cdot))$. The respective TV-norm continuity estimate for $\mu_L$ follows by applying Proposition \ref{proposition:TV norm continuity and tightness estimate for discrete Levy measures} and passing to the limit (always recalling that, $\mathcal{D}_I$ is the convex hull of such limit points).
   
  With the convergence established, we have 
  \begin{align*}
    \lim\limits_{k\to \infty}\big ( B_{L,\phi_{\delta_{m_k}}}(x)\cdot\nabla u(x) + \tr(A_{L,\eta_{\delta_{m_k}}}(x)D^2u(x)) \big ) = B(x)\cdot\nabla u(x) + \tr(A(x)D^2u(x)),
  \end{align*}
  and so, for every $u$ we have the formula
  \begin{align*}
    L(u,x) & = C(x)u(x)+B(x)\cdot \nabla u(x)+\tr(A(x)D^2u(x))\\
      & \;\;\;\;+ \lim\limits_{k \to \infty}\int_{\mathbb{R}^d}u(x+y)-P_{\phi_{\delta_{k}},\eta_{\delta_k},u,x}(x+y)\;\mu(x,dy),	
  \end{align*} 
  It remains to compute the limit of the integral, observe that
  \begin{align*}
    \int_{\mathbb{R}^d}\eta_{\delta_k}(y)(D^2u(x)y,y)\;\mu(x,dy) = \int_{B_{\delta_k}}\eta_{\delta_k}(y)(D^2u(x)y,y)\;\mu(x,dy),
  \end{align*}
  which means that
  \begin{align*}
    \left | \int_{\mathbb{R}^d}\eta_{\delta_k}(y)(D^2u(x)y,y)\;\mu(x,dy) \right | \leq C|D^2u(x)| \int_{B_{\delta_k}} |y|^2 \;d\mu(x,dy).
  \end{align*}
  Therefore,
  \begin{align*}
    \lim\limits_{k\to 0}\int_{\mathbb{R}^d}\eta_{\delta_k}(y)(D^2u(x)y,y)\;\mu(x,dy) = 0.
  \end{align*}
  On the other hand, for every $y$ we have
  \begin{align*}
    \lim \limits_{k\to \infty} \Big ( u(x+y)-P_{\phi_{\delta_k},\eta_{\delta_k},u,x}(y)\Big) =   u(x+y)-u(x)-\chi_{B_1}(y)\nabla u(x)\cdot y,
  \end{align*}  
  and the limit is monotone. Therefore, by monotone convergence we conclude that
  \begin{align*}
    \lim\limits_{k\to \infty} \int_{\mathbb{R}^d}u(x+y)-P_{\phi_{\delta_k},\eta_{\delta_k},u,x}(y)\;\mu(x,dy) = \int_{\mathbb{R}^d}u(x+y)-u(x)-\chi_{B_1}(y)\nabla u(x)\cdot y\;\mu(x,dy).
  \end{align*}
  and with this the Corollary is proved.
\end{proof}

%%%%%%%%%%%%%%%%%%%%%%%%%%%%%%%%%
%%%%%%%%%%%%%%%%%%%%%%%%%%%%%%%%%
\subsection{Limits of $I_n$}

\begin{lemma}\label{lemma:I_n converges to Iu for nice u}
  Assume that $I:C^\beta_b(\real^d)\to C^0_b(\real^d)$ is Lipschitz.     Let $K>0$ and $0<\beta<\beta_0<3$. If $u \in C^{\beta_0}_b(\mathbb{R}^d)$ is supported in $B_K$, and $2^{n-2}\geq K$, then
  \begin{align*}
    \|I_nu-Iu\|_{L^\infty(B_K \cap G_n)} \leq C2^{-n \gamma} \|u\|_{C^{\beta_0}(\mathbb{R}^d)},
  \end{align*}
  for a universal constant $C$ and $\gamma= \gamma(\beta_0,\beta) \in (0,1)$. Furthermore, we have
  \begin{align*}
    \lim \limits_{n\to\infty} \|I(u)-I_n(u)\|_{L^\infty(B_K)} = 0.
  \end{align*}
\end{lemma}

\begin{proof}
  Let $u$ be compactly supported in $B_K$, and be such that $\|u\|_{C^{\beta_0}} \leq M$. First, note that since $2^{n-2} \geq K$, then we have
  \begin{align*}
    \hat \pi_n^\beta u = \pi_n^\beta u,	  
  \end{align*}	  
  thus,  $I_n(u) = \hat \pi_n^0\circ I \circ \pi_n^\beta (u)$. Keeping this in mind, using the Lipschitz property of $I$, we have
  \begin{align*}
    \|I(u)-I(\hat \pi_n^\beta u)\|_{L^\infty(\mathbb{R}^d)} \leq C\|u-\pi_n^\beta u\|_{C^\beta(\mathbb{R}^d)}.
  \end{align*}
  Since $2^{n-2}\geq K$ we have that $I(\hat \pi_n^\beta u) = \hat \pi_n^0 I(\hat \pi_n^\beta u)  = I_n(u)$ when restricted to $B_K \cap G_n$, which thanks to Lemma \ref{lemma:projection operators convergence} implies the first estimate. Next, Theorem \ref{theorem:Whitney Extension Is Bounded} guarantees that
  \begin{align*}
    \|\hat \pi_n^0 I(u)-\hat \pi_n^0 I(\hat \pi_n^\beta u)\|_{L^\infty(K)} \leq C\|I(u)-I(\hat \pi_n^\beta u)\|_{L^\infty(\mathbb{R}^d)} \leq C\|u- \pi_n^\beta u\|_{L^\infty(\mathbb{R}^d)}.
  \end{align*}	  
  Thus,
  \begin{align*}
    \|I_n(u)-I(u)\|_{L^\infty(K)} & \leq \|\hat \pi_n^0 I(u)-I_n(u)\|_{L^\infty(K)} + \|\hat \pi_n^0( I(u))-I(u)\|_{L^\infty(K)}\\
	& \leq C \|u- \pi_n^\beta u\|_{C^\beta(\mathbb{R}^d)} +  \|\hat \pi_n^0 (I(u))-I(u) \|_{L^\infty(K)}.
  \end{align*}
  Applying Lemma \ref{lemma:projection operators convergence} to the first term and Remark \ref{remark:projection operators C0 convergence} to the second, we conclude that 
  \begin{align*}
    \lim\limits_{n\to \infty}\|I_nu-Iu\|_{L^\infty(K)} =0.
  \end{align*}
  
\end{proof}

\begin{corollary}\label{corollary:I_n pointwise convergence to I}
Assume $I$ satisfies Assumptions \ref{assumption:GCP}, \ref{assumption:tightness bound}, and \ref{assumption:coefficient regularity}, as stated for $C^\beta_b(\real^d)$. Then for every $u \in C^\beta_b(\mathbb{R}^d)$ and every $R>0$,   \begin{align*}
    \lim \limits_{n\to \infty} \|I_nu-Iu\|_{L^\infty(B_R)} = 0.
  \end{align*}
\end{corollary}

\begin{proof}
  Fix $u \in C^\beta_b(\mathbb{R}^d)$ and $R,\varepsilon>0$. For $K>0$ (to be determined later), we may decompose $u$ as $u = u_0+u_1$, where $u_0$ is compactly supported in $B_{2K+1}$ and $u_1 \equiv 0$ in $B_{2K}$, all such that
  \begin{align*}
    \|u_i \|_{C^\beta(\mathbb{R}^d)} \leq C\|u\|_{C^\beta(\mathbb{R}^d)},\;i=1,2.
  \end{align*}
  The constant $C>1$ being independent of $K$. Now, by Assumption \ref{assumption:tightness bound} and since $u \equiv u_0$ in $B_{2K}$, we have
  \begin{align*}
    |I(u_0)-I(u)| \leq \rho(K)\|u-u_0\|_{L^\infty(\mathbb{R}^d)} \leq 2C\rho(K)\|u\|_{C^\beta(\mathbb{R}^d)}.  
  \end{align*}
  Choose $K$ large enough so that $K\geq 2R$ and $2C\rho(R)\|u\|_{C^\beta(\mathbb{R}^d)} \leq \varepsilon/2$. Then, with this $K$, we apply Lemma \ref{lemma:I_n converges to Iu for nice u} two times, and conclude that there is some $n_0>0$ such that 
  \begin{align*}
    |I_n(u_0)-I(u_0)| + |I_n(u_0)-I_n(u)| \leq \varepsilon/2 \textnormal{ whenever } n \geq n_0.
  \end{align*}
  On the other hand, in all $\mathbb{R}^d$ we have the pointwise inequality,
  \begin{align*}
     |I_n(u)-I(u)| \leq |I_n(u_0)-I(u_0)| + |I_n(u_0)-I_n(u)| + |I(u_0)-I(u)|,
   \end{align*}
  and it follows that, for $x \in B_{R}$ and $n\geq n_0$, that 
  \begin{align*}
    |I_n(u,x)-I(u,x)| \leq \varepsilon,
  \end{align*}
  and the corollary is proved. 

\end{proof}

%%%%%%%%%%%%%%%%%%%%%%%%%%%%%%%%%
%%%%%%%%%%%%%%%%%%%%%%%%%%%%%%%%%
%%%%%%%%%%%%%%%%%%%%%%%%%%%%%%%%%

\subsection{Proofs of Theorems \ref{theorem:MinMax Euclidean ver2} and \ref{theorem:minmax with beta less than 2}}\label{sec:TheoremsThatUseWhitney}

 We conclude this section with the proofs of the remaining theorems.

\begin{proof}[Proof of Theorem \ref{theorem:MinMax Euclidean ver2}]

Consider the set $\mathcal{D}_I$. The proof will boil down to showing that for any $u,v\in C^{\beta_0}_c(\mathbb{R}^d)$ and any $x\in \mathbb{R}^d$ there is some $L\in \mathcal{D}_I$ such that
\begin{align*}
  I(u,x)  \leq I(v,x)+L(u-v,x). 
\end{align*}
Fix $u,v$ and $x$. Then, by Remark \ref{remark:MinMax for In}, for every $n$ we have
\begin{align*}
  I_n(u,x) \leq \max \limits_{L_n\in \mathcal{D}I_n} \{ I_n(v,x)+L_n(u-v,x)\}.
\end{align*}
In particular, for every $n$, there is some $L_n \in \mathcal{D}I_n$ such that (with this same $u,v$ and $x$)
\begin{align*}
  I_n(u,x) \leq I_n(v,x)+L_n(u-v,x).	
\end{align*}
Let us obtain an inequality as we let $n\to \infty$ along some subsequence. Thanks to Corollary \ref{corollary:I_n pointwise convergence to I}, for every $x \in \mathbb{R}^d$ we have
\begin{align*}
  \lim \limits_{n\to \infty} I_n(u,x) = I(u,x),\;\;\lim\limits_{n\to\infty} I_n(v,x) = I(v,x).
\end{align*}
On the other hand, Lemma \ref{lemma:DI_n sequences subconverge} says there is a subsequence $n_k$ and an operator $L$ such that $L_{n_k}(u-v,x)$ converges to $L(u-v,x)$, and moreover $L\in\mathcal{D}_I$, by the definition of $\mathcal{D}_I$. Then, we conclude that
\begin{align*}
  I(u,x) & \leq I(v,x)+L(u-v,x) \leq \sup \limits_{L\in \mathcal{D}_I} \{ I(v,x)+L(u-v,x)\}.
\end{align*}
The above holds for any pair of functions $u$ and $v$ and any point $x\in\mathbb{R}^d$. Taking the minimum over all $v$, we obtain for any $u$ and $x$,
\begin{align*}
  I(u,x) & = \min \limits_{v\in C^\beta_b(\mathbb{R}^d)} \max \limits_{L\in \mathcal{D}_I} \left \{ I(v,x)-L(v,x)+L(u,x) \right \}. 
\end{align*}
Using $v\in C^\beta_b(\mathbb{R}^d)$ and $L \in \mathcal{D}_I$ as the set of labels, which we rename $ab$, and letting $f_{ab}(x)$ correspond to the functions $I(v,x)-L(v,x)$, we obtain the desired min-max representation. 

The $L^\infty$ bounds for the coefficients follow from the construction of $A_{\eta_k}$, etc... in (\ref{eqn:AsubL definition}), (\ref{eqn:BsubL definition}), (\ref{eqn:CsubL definition}).  The continuity of the coefficients and the L\'evy measures follows from Lemma \ref{lemma:DI_n sequences subconverge}.

\end{proof}

\begin{proof}[Proof of Theorem \ref{theorem:minmax with beta less than 2}]
  For the versions of Theorems \ref{theorem:MinMax Euclidean} and \ref{theorem:MinMax Translation Invariant} with $\beta<2$ we apply the last part of Lemma \ref{lemma:Courrege theorem for a linear functional} to conclude the functionals (or translation invariant operators) appearing in the min-max all have the corresponding simpler form. 
  As for Theorem \ref{theorem:MinMax Euclidean ver2}, we use instead the last part of Corollary \ref{corollary:L in DI representation formula final form} to obtain the simpler expresion for the L\'evy operators in the cases where $\beta<2$.

\end{proof}

%%%%%%%%%%%%%%%%%%%%%%%%%%%%%%%%%
%%%%%%%%%%%%%%%%%%%%%%%%%%%%%%%%%
%%%%%%%%%%%%%%%%%%%%%%%%%%%%%%%%%
%%%%%%%%%%%%%%%%%%%%%%%%%%%%%%%%%
%%%%%%%%%%%%%%%%%%%%%%%%%%%%%%%%%
%%%%%%%%%%%%%%%%%%%%%%%%%%%%%%%%%
\section{Some Examples}\label{section:examples}

In this section we list some examples to which our results apply, yet the integro-differential structure given in either (\ref{eqIN:LevyTypeLinear}) or (\ref{eqIN:MinMaxMeta}) is not readily apparent from the definition of the operator itself.  We emphasize that most cases of the \emph{linear} examples that we list were already contained in the classic work of Courr\`ege \cite{Courrege-1965formePrincipeMaximum}, but we include them here for the sake of illustration.  In all of these examples, the operators satisfy the GCP and the other technical requirements to apply the results presented above.  We do not intend to give all details, but rather just make a list, with some appropriate references.  At the end of the section, we list how these examples relate to Assumptions \ref{assumption:GCP}--\ref{assumption:coefficient regularity}.

\subsection{The statement of the examples.}

\begin{example}\label{EX:Generator}
	The generator of a Markov process.  Assume that $X_t$ is a Markov process taking values in $\real^d$, and that $\expct_x$ is the expectation of the process, having started from $x$ at $t=0$.  The generator is defined as the operator
	\begin{align*}
		L(u,x) = \lim_{t\to0}\frac{\expct(u(X_t))-\expct(u(X_0))}{t},
	\end{align*}
	over all $u$ for which the limit exists.  (See Liggett \cite[Chapter 3]{Liggett-ContinuousTimeMarkovBookAMS}.)
\end{example}
	Thanks to the fact that $\expct$ preserves ordering, one can immediately see that $L$ enjoys the GCP.
	When $X_t$ is such that $L:C^2_b\to C^2_b$, this example is covered by Courr\`ege \cite{Courrege-1965formePrincipeMaximum}; but if $X_t$ is such that $L:C^\beta_b\to C^0_b$ (in a Lipschitz fashion) for some $0<\beta<2$, then by Theorem \ref{theorem:minmax with beta less than 2}, there are fewer terms (see the list just above Theorem \ref{theorem:minmax with beta less than 2} for our use of the notation $C^\beta_b(\real^d)$).  In this context, the result of Courr\`ege can be seen as a version of the L\'evy-Khintchine formula for a process whose increments need not be stationary.

\begin{example}\label{EX:D-to-N}
	The Dirichlet-to-Neumann map for linear, elliptic operators on half-space.  Assume that $L$ is an operator that admits unique bounded solutions on $\real^{d+1}_+$ and that has a comparison principle.  What we mean by this is the following:  we can take $u\in C^{1,\al}_b(\real^d)$ and associate to it the unique bounded solution, $U_u$ of 
	\begin{align*}
	L(U_u,X) = 0\ \ \text{in}\ \real^{d+1}_+,\ \ \text{and}\ \ U_u=u\ \text{on}\ \real^d\times{0}.	
	\end{align*}
A couple of reasonable examples would be
\begin{align*}
	L(U,X) = \tr(A(X)D^2U(X))\ \ \text{or}\ \ L(U,X) = \dive(A(X)\nabla U),
\end{align*} 
where $A$ is uniformly elliptic and H\"older continuous.
The Dirichlet-to-Neumann map is then defined as
\begin{align*}
	I(u,x) := \partial_n U_u(x). 
\end{align*}
\end{example}
First of all, the assumptions on $A$ are such that for some $\al'$, $U_u\in C^{1,\al'}_b\left(\overline{\real^{d+1}_+}\right)$  and hence the normal derivative is well defined (see, e.g. \cite[Chapters 8, 9]{GiTr-98}).
It is not hard to check that this operator satisfies the GCP, and this fact comes entirely from the property that the solution operator, by the assumed comparison principle, preserves ordering of solutions whenever the boundary data are ordered (it has nothing to do with linearity of the solution operator). This is, again, within the context of Courr\`ege's result, but  we can invoke Theorem \ref{theorem:minmax with beta less than 2} to remove extra terms of order higher than $1$.  Ellipticity and scaling show that this is always an operator of order $1$ (and will map $C^{1,\al}\to C^{\al'}$).  We note that in this example, via linear equations with nice coefficients, one can derive lots of information about the operator $\partial_n U_u$ by directly using the Poisson kernel that represents the solution $U_u$.
	
	In the context of periodic equations, one can use the results in Sections \ref{section:Finite Dimensional Approximations} and \ref{section:Analysis of finite dimensional approximations} to show that the coefficients in the resulting L\'evy operators will share the same periodicity.  In fact, this is very straightforward if $I$ is linear.  If instead one looks at almost periodic coefficients, it seems reasonable to hope that the coefficients will also be almost periodic, but we have not checked this claim.  If it is the case, there could be an application to some boundary homogenization problems with irrationally oriented half-spaces inside a periodic medium, related to  \cite{GuSc-2018NeumannHomogPart2SIAM}.
	Operators related to the Dirichlet-to-Neumann mapping of this example are also of interest in conformal geometry, see Chang-Gonzalez \cite{ChGo-2011FracLaplaceGeometry}.  It is also possible to consider an elliptic equation with weights in order to obtain some operators of order different than 1, e.g. Caffarelli-Silvestre \cite{CaffarelliSilvestre-2007ExtensionCPDE}.

\begin{example}\label{EX:BoundaryProcess}
	The boundary process of a reflected diffusion. (See Hsu \cite{Hsu-1986ExcursionsReflectingBM}, or \cite[Chp. IV, Sec. 7]{IkedaWatanabe-1981SDE} and/or  \cite[Sec. 8]{SatoUeno-1965MultiDimDiffBoundryMarkov}.) 
\end{example}
	In this context, one starts with a diffusion in $\real^{d+1}_+$, say $X_t$, so that $X_t$ reflects off of the bottom boundary whenever it reaches it.  Under a time rescaling of $X_t$ (because it spends zero time on the boundary), the resulting process can be viewed at times only when it hits $\real^{d}\times\{0\}$, and induces a pure jump process on $\real^d\times\{0\}$.  This process is generated by an operator of the form (\ref{eqIN:LevyTypeLinear}) with $A\equiv 0$.  It turns out that this generator for the boundary process is exactly the Dirichlet-to-Neumann mapping from the previous example.  This process was studied in a smooth domains for Brownian motion by Hsu \cite{Hsu-1986ExcursionsReflectingBM}.

\begin{example}\label{EX:Subordination}
	Subordinated diffusions and Bernstein functions. (See Schilling-Song-Vondra\v{c}ek \cite{SchillingSongVondracek-2012BookBernsteinFunctions}.)  
\end{example}
The time-rescaling of the reflected diffusion in the previous example is just one choice of a rescaling, and in general one can time-rescale diffusions on $\real^d$ (so no boundary space here) in a myriad of fashions to create new stochastic processes from one reference Brownian motion.  This is a process known as subordination, and it can be used to create operators with generators in the class (\ref{eqIN:LevyTypeLinear}), starting with one that may simply only contain the second order term.  The generator for the subordinated process will enjoy the GCP because the generator of the original diffusion also enjoys the GCP.  This technique has played a large and fundamental role in the study of L\'evy processes, and one can see it in use in e.g., the book of Schilling-Song-Vondra\v{c}ek  \cite{SchillingSongVondracek-2012BookBernsteinFunctions}, especially \cite[Chapter 13]{SchillingSongVondracek-2012BookBernsteinFunctions}.  The subordination formula is closely related to an extension into plus one space variables, and this extension was used to create operators of fractional order that enjoy the GCP in the work of Stinga-Torrea \cite{StinTor2010Extension} and also provide other properties of the fractional operators.

\begin{example}\label{EX:MongeAmpere}
	The Monge-Amp\`ere operator, $\MA(u,x)=\det(D^2u)$. 
\end{example}
 When one restricts this operator to the subset of $C^2$ of convex functions, then $\MA$ is in fact (degenerate) elliptic and locally Lipschitz.  Specifically for each $\delta>0$, $\MA$ is uniformly elliptic (depending upon $\delta$), Lipschitz, and translation invariant as a mapping, 
 \begin{align*}
 \MA: \{u\in C^2_b(\real^d) : \frac{1}{\delta}>D^2u> \delta\}\to C^0_b(\real^d).  
 \end{align*}
 Thus, $\MA$, must enjoy a min-max structure.  Experts have known and utilized this min-max propert of $\MA$ in the study of fully nonlinear elliptic equations for a long time, and one can show that 
\begin{align*}
(\MA(u,x))^{1/d} =\frac{1}{d} \inf\{ \tr(AD^2u(x)) : A\geq 0,\ \text{and}\ \det(A)=1\}.  
\end{align*}
In fact, this formula is intimately connected with various investigations into nonlocal operators that should be an analog of $\MA$ in the fractional setting (as of yet, there is not one that is considered better than others).  Some works that address nonlocal analogs of $\MA$ are: \cite{CaffarelliCharro-2015FractionalMongeAmpereAPDE}, \cite{CaffarelliSilvestre-2016NonlocalMongeAmpereCAG}, and \cite{GuSc-12ABParma}.

\begin{example}\label{EX:CaSiExtremal}
	General nonlocal operators as treated in Caffarelli-Silvestre \cite{CaSi-09RegularityIntegroDiff} \cite{CaSi-09RegularityByApproximation}.  These are simply operators that are assumed to satisfy the GCP, are defined for all functions in $C^{1,1}(\real^d)$, map $C^2_b(\real^d)\to C^0_b(\real^d)$, and satisfy a form of uniform ellipticity that is given by the existence of concave respectively convex operators, $\M^-_\L$ and $\M^+_\L$ so that
	\begin{align}\label{eqEX:ExtremalIneq}
		\text{for all}\ u,v\in C^{1,1}(\real^d),\ 
		\M^-_\L(u-v,x)\leq I(u,x)-I(v,x)\leq \M^+_\L(u-v,x).
	\end{align}
Here, $\L$ is a class of linear operators that is usually a particular subset of those that satisfy the L\'evy type condition (\ref{eqIN:LevyTypeLinear}).
\end{example}

This context for nonlocal operators was given in \cite[Definition 3.1]{CaSi-09RegularityIntegroDiff}, and it played an important role in many of the results-- especially when $\L$ is chosen to contain certain classes of operators. These operators, in cases in which they are Lipschitz fall into the scope of our results, and furthermore, the role of the extremal operators gives extra information about the min-max formula.  In particular, as shown in \cite[Section 4.6]{GuSc-2016MinMaxNonlocalarXiv}, when ellipticity occurs with respect to $\M^\pm_\L$, then the min-max may be restricted to only utilize linear functionals (or linear operators) that also satisfy the extremal inequality in (\ref{eqEX:ExtremalIneq}).   This also appeared in a homogenization result by one of the authors in which they were unable to show that the limit operator had an explicit integro-differential formula, but rather was only integro-differential and uniformly elliptic in the sense of \cite[Definition 3.1]{CaSi-09RegularityIntegroDiff} ( see the homogenization in \cite{Schw-10Per}).

\begin{example}\label{EX:DtoNFullyNonBananas}
	The Dirichlet to Neumann map for fully nonlinear elliptic equations.  In Example \ref{EX:D-to-N}, the linearity of $L$ is not necessary, and the function $U_u$ can also be taken to solve a fully nonlinear, uniformly elliptic equation in $\real^{d+1}_+$.  These equations always possess a comparison principle (by definition), and under most reasonable assumptions, the solution $U_u$ will be globally $C^{1,\al'}$, allowing for the normal derivative to be defined classically (see \cite{SilvestreSirakov-2013boundary} for this regularity).
\end{example}

This was a main topic in the recent paper by the authors and Kitagawa \cite{GuillenKitagawaSchwab2017estimatesDtoN}.  It turns out that the extremal operators (as in Example \ref{EX:CaSiExtremal}) for the nonlinear D-to-N not only play a crucial role in investigating the L\'evy measures in the min-max, but they also take a refreshingly simple form.  The extremal operators in this case, $\M^\pm_\L$ of Example \ref{EX:CaSiExtremal}, are simply the Dirichlet-to-Neumann operators for the solutions of the corresponding extremal operators for the elliptic second order equation in $\real^{d+1}_+$.  These are usually called the Pucci extremal operators (see \cite{CaCa-95}), and solutions to their equations are generally very well behaved.  In \cite{GuillenKitagawaSchwab2017estimatesDtoN}, the properties of the L\'evy measures in the min-max are linked to the harmonic measures for linear equations with bounded measurable coefficients (e.g. \cite{Kenig-1993PotentialThoeryNonDiv}), but there is still more to learn about them before they can be connected with existing integro-differential theory.

\begin{example}\label{EX:HeleShawFTW}
	An operator that drives surface evolution in one and two phase free boundary problems related to a type of Hele-Shaw flow.  Given $f\in C^{1,\al}(\real^d)$, such that $0<\inf f\leq \sup f<\infty$, we can define the unique solution, $U_f$, of the elliptic equation,
	\begin{align*}
		&\Delta U_f = 0\ \text{in}\ \{(x,x_{d+1}) : 0<x_{d+1}<f(x) \},\\ 
		&U_f=1\ \text{on}\ \{x_{d+1}=0\},\
		U_f=0\ \text{on}\ \{ (x,d_{d+1}) : x_{d+1}=f(x) \}.
	\end{align*}
	This allows to define a (fully nonlinear) operator on $f$ as
	\begin{align*}
		I(f,x):= \partial_n U_f(x,f(x)),
	\end{align*}
	that is, the normal derivative of the solution on the upper boundary given by the graph of $f$.
\end{example}

For Hele-Shaw flow in the simplified setting that the free boundary is parametrized by the graph of $f(\cdot,t)$, it can be shown that the free boundary evolves by a normal velocity that at each time is given by $I(f,x)$.  The interpretation here is that fluid flows into the domain under a pressure at the bottom boundary, $x_{d+1}=0$, and the top edge of the fluid exists at $x_{d+1}=f(x)$, with $U_f$ representing the pressure of the fluid.  This pressure induces a force on the fluid, which is given by $\partial_n U_f(x,f(x))$ at the top boundary.  This operator, and its implications for rewriting a class of free boundary problems that are similar to Hele-Shaw was studied by the authors and Chang Lara in \cite{ChangLaraGuillenSchwab2018FBasNonlocal}.  In particular, the min-max formula makes it straightforward to convert the free boundary flow into a nonlocal parabolic equation for $f$, and this parabolic equation is very similar to ones that have already been studied in the nonlocal literature (e.g. \cite{Silv-2011DifferentiabilityCriticalHJ}).  When $U_f$ is defined to be harmonic in the domain determined by $f$, standard regularity theory immediately gives estimates that show there is some $\al'$ so that the mapping from $f$ to $I(f)$ is Lipschitz from $C^{1,\al}(\real^d)$ to $C^{\al'}(\real^d)$.  In \cite{ChangLaraGuillenSchwab2018FBasNonlocal} it was also shown that the same Lipschitz property can be obtained when $U_f$ is defined as the solution of a nonlinear uniformly elliptic second order equation instead of just the Laplacian.  This operator gives a good example of what can be said in the translation invariant case of the min-max, and its properties are studied initially in \cite{ChangLaraGuillenSchwab2018FBasNonlocal}.  Even in the simplest case of defining $U_f$ to be harmonic, the resulting operator $I$ will always be inherently nonlinear and nonlocal.

\subsection{Relationship to Assumptions \ref{assumption:GCP}--\ref{assumption:coefficient regularity}}

Here we list how each of the above examples fits within the context of Assumptions \ref{assumption:GCP}--\ref{assumption:coefficient regularity}.

\textbf{(Example \ref{EX:Generator}).}  By construction, this $L$ is always linear.  Thus, Assumption \ref{assumption:GCP} follows from simply saying that $L$ is a bounded operator on $C^{\beta}$, which of course requires assumptions on the process, $X_t$, or more specifically the transition probability measure for $X_t$.  Again, via linearity, Assumption \ref{assumption:translation invariance} follows whenever the process, $X_t$, has stationary and independent increments.  Assumptions \ref{assumption:tightness bound} and \ref{assumption:coefficient regularity} will be an extra requirement on the transition probability measure for $X_t$.  In particular (although a bit circular), Assumption \ref{assumption:coefficient regularity}, in view of linearity, is equivalent to the martingale problem for $X_t$ having a solution and the generator having uniformly continuous coefficients.  

\textbf{(Example \ref{EX:D-to-N}).}  (The interested reader can see \cite{GuillenKitagawaSchwab2017estimatesDtoN} for more details.) Assumption \ref{assumption:GCP} holds for $C^{1,\al}\to C^{\al'}$ when $A$ is $\al$-H\"older continuous.  Assumption \ref{assumption:translation invariance} holds if $A$ is a constant. Assumption \ref{assumption:tightness bound} holds in both of the above settings, by using a barrier argument (which is easier implemented for the non-divergence equation).  Since $I$ is linear, Assumption \ref{assumption:coefficient regularity} holds when $A$ is H\"older continuous.  Indeed, by linearity, checking Assumption \ref{assumption:coefficient regularity} is equivalent to estimating
\begin{align*}
	I(\tau_{-z}u,x+z)-I(u,x).
\end{align*}
In the case of divergence equations, one can write down the equations satisfied for $V=\tau_{-z}U_u$, and then also the equation satisfied by $W:= U_{\tau_{-z}u}-V$.  The desired estimate is then equivalent to estimating $\abs{\partial_n W(x+z)}$, i.e. a global Lipschitz estimate for $W$.  Since $W$ satisfies 
\begin{align*}
	\dive(A(X)\grad W(X))= -\dive((A(X)-A(x-z))\grad V),
\end{align*}
we see that by global Lipschitz estimates, 
\begin{align*}
	\abs{\grad W}\leq C\norm{(A-A(\cdot-z))\grad V}_{L^\infty}\leq C\abs{z}^\al,
\end{align*}
by the original assumption that $A$ is H\"older continuous.  (Note, the Lipschitz estimates here are a standard modification to, e.g. \cite[Lemma 3.2]{GruterWidman-1982GreenFunUnifEllipticManMath} to allow for a right hand side of the form $\dive(f)$ with $f\in L^\infty$.)

\textbf{(Example \ref{EX:BoundaryProcess}).}  In most reasonable situations in which the diffusion has regular coefficients, this is contained in the previous example.

\textbf{(Example \ref{EX:Subordination}).}  This, of course, depends heavily on the original Markov process and the choice of subordinator.  However, one of the most classical situations starts with a Brownian motion and then uses a L\'evy stable subordinator.  In this case, the resulting operator is translation invariant, and Assumptions \ref{assumption:GCP} and \ref{assumption:translation invariance} follow more or less by construction.

\textbf{(Example \ref{EX:MongeAmpere}).}  This is a translation invariant operator, and as mentioned already satisfies the Lipschitz property on the specified convex subsets of $C^2$.  So, Assumptions \ref{assumption:GCP} and \ref{assumption:translation invariance} hold.

\textbf{(Example \ref{EX:CaSiExtremal}).}  As this is a general example, the operators only satisfy the given assumptions when explicitly required to do so.  However, the interesting part of this example arises from the fact that the knowledge of the extremal inequalities in (\ref{eqEX:ExtremalIneq}) in fact gives more detailed information about the linear operators that will appear in the min-max of Theorems \ref{theorem:MinMax Euclidean}--\ref{theorem:minmax with beta less than 2}.  This is discussed in \cite[Section 4.6]{GuSc-2016MinMaxNonlocalarXiv}.

\textbf{(Example \ref{EX:DtoNFullyNonBananas}).}  This operator satisfies Assumption \ref{assumption:GCP} as a mapping of $C^{1,\al}\to C^{\al'}$ (for some $0<\al'<\al$) under standard assumptions about $F$.  The relevant regularity theory comes from Silvestre-Sirakov \cite{SilvestreSirakov-2013boundary}.  It can also be checked by using the same type of barrier argument that works for Example \ref{EX:D-to-N} will show Assumption \ref{assumption:tightness bound} is also satisfied.  Due to the nonlinear nature of the D-to-N in this setting, it is not obvious how to show that Assumption \ref{assumption:coefficient regularity} is satisfied-- we do not know if it satisfied or not.  Thus, the best one can say about this operator when it is not translation invariant is the outcome of Theorem \ref{theorem:MinMax Euclidean}.  We simply note to the interested reader that because of the lack of exact cancelation from the fact that the mapping is not linear, one probably needs more detailed information about $F$.  Indeed, using the extremal operators would not help because it would produce
\begin{align*}
	I(v+\tau_{-z} u, x+z)-I(v,x+z) -(I(v+u,x)-I(v,x))
	&\leq M^+(\tau_{-z}u,x+z) - M^-(u,x)\\
	=M^+(u,x)-M^-(u,x).
\end{align*}
Here we use $M^\pm$ as the extremal operators for $I$, and also that these are translation invariant.  This estimate completely neglects the influence of the shift, $\tau_z$, and so it would not be useful (furthermore, one expects that $M^+(u,x)>M^-(u,x)$).

\textbf{(Example \ref{EX:HeleShawFTW}).}  As it is stated above, this operator, $I$, is actually translation invariant, and so it is straightforward to check that Assumptions \ref{assumption:GCP} and \ref{assumption:translation invariance} hold.  In the case that the equation for $U$ (i.e. $\Delta U=0$) is replaced by either a fully nonlinear operator and/or and operator that is not translation invariant, it is harder to check all of the applicable assumptions.  Again, for fully nonlinear equations that define $U$, in \cite{ChangLaraGuillenSchwab2018FBasNonlocal} $I$ was checked to be Lipschitz as a map of $C^{1,\al}\to C^{\al'}$ (which took a reasonably non-trivial amount of work).

\appendix

%%%%%%%%%%%%%%%%%%%%%%%%%%%%%%%%%
%%%%%%%%%%%%%%%%%%%%%%%%%%%%%%%%%
%%%%%%%%%%%%%%%%%%%%%%%%%%%%%%%%%
%%%%%%%%%%%%%%%%%%%%%%%%%%%%%%%%%
%%%%%%%%%%%%%%%%%%%%%%%%%%%%%%%%%
%%%%%%%%%%%%%%%%%%%%%%%%%%%%%%%%%
\section{Additional proofs and computations}

\begin{proof}[Proof of Proposition \ref{proposition:discrete derivatives converge to real derivatives}]
  Fix $u \in C^{\beta}_b(\mathbb{R}^d)$, and let $x \in G_n$, then by the regularity of $u$,
  \begin{align*}
    |u(x \pm h_n e_k)-  (u(x) \pm h_n \nabla u(x_0)\cdot e_k) | \leq  C\|u\|_{C^\beta} h_n^{\min\{\beta-1,1\}}.
  \end{align*}
  Therefore, 
  \begin{align*}
    |u(x + h_ne_k) - u(x + h_n e_k) -2 h_n \nabla u(x_0)\cdot e_k  | \leq  C\|u\|_{C^\beta} h_n^{\min\{\beta-1,1\}}
  \end{align*}
  For the second estimate, we shall make use of 
  \begin{align*}
    |u(x + h_n e_k)-  (u(x) + h_n \nabla u(x_0)\cdot e +h_n^2\tfrac{1}{2}(D^2u(x) e,e)) | \leq  C\|u\|_{C^\beta} h_n^{\min\{\beta-2,1\}}.
  \end{align*}
  Therefore, 
  \begin{align*}
    & u(x+h_n e_k+h_n e_\ell) - u(x+h_n e_k) - u(x+h_n e_\ell) + u(x)\\
    & ``=''  u(x) + h_n \nabla u(x_0)\cdot (e_k+e_\ell) +h_n^2\tfrac{1}{2}(D^2u(x)(e_k+e_\ell,e_k+e_\ell)\\
    & \;\;\;\; -(u(x) + h_n \nabla u(x_0)\cdot e_k +h_n^2\tfrac{1}{2}(D^2u(x) e_k,e_k))\\
    & \;\;\;\; -(u(x) + h_n \nabla u(x_0)\cdot e_\ell +h_n^2\tfrac{1}{2}(D^2u(x) e_\ell,e_\ell)) + u(x)\\
    & = h_n^2\tfrac{1}{2} \left ( (D^2u(x)(e_k+e_\ell,e_k+e_\ell)-(D^2u(x) e_k,e_k))-(D^2u(x) e_\ell,e_\ell)) \right )\\
    & = h_n^2 (D^2u(x)e_k,e_\ell)		  			
  \end{align*}
  It follows that
  \begin{align*}
    |u(x+h_n e_k+h_n e_\ell) - u(x+h_n e_k) - u(x+h_n e_\ell) + u(x)- h_n^2 (D^2u(x)e_k,e_\ell)| \leq C\|u\|_{C^\beta} h_n^{\min\{\beta-2,1\}},		  			
  \end{align*}
  and the proposition is proved.
\end{proof}

\begin{proof}[Proof of Proposition \ref{proposition:discrete derivatives regularity}]
  Fix $u \in C^{\beta}_b(\mathbb{R}^d)$.

  \emph{Step 1}. Let $x \in G_n$, then
  \begin{align*}
    & | (\nabla_n)^1u(x)-\nabla u(x)| \leq C\|u\|_{C^{\beta}}h_n^{\beta-1},\;\textnormal{ if } \beta \in [1,2],\\
    & | (\nabla_n)^2u(x)-D^2 u(x)| \leq C\|u\|_{C^{\beta}}h_n^{\beta-2},\;\textnormal{ if } \beta \in [2,3].	
  \end{align*}
  
  \emph{Proof of Step 1.}  By the regularity of $u$,
  \begin{align*}
    |u(x \pm h_n e_k)-  (u(x) \pm h_n \nabla u(x_0)\cdot e_k) | \leq  C\|u\|_{C^\beta} h_n^{\min\{\beta-1,1\}}.
  \end{align*}
  Therefore, 
  \begin{align*}
    |u(x + h_ne_k) - u(x + h_n e_k) -2 h_n \nabla u(x_0)\cdot e_k  | \leq  C\|u\|_{C^\beta} h_n^{\min\{\beta-1,1\}}
  \end{align*}

  \emph{Step 2.} Given $x\in G_n$, we have
  \begin{align*}
    & | (\nabla_n)^1u(x)| \leq C\|u\|_{C^{1}},\;| (\nabla_n)^2u(x)| \leq C\|u\|_{C^{2}}.	
  \end{align*}

  \emph{Step 3.} 
  
  \begin{align*}
    |(\nabla_n)^1u(\hat x)-(\nabla_n)^1u(\hat y)| \leq C\|u\|_{C^\beta} d(\hat x,\hat y)^{\beta-1}, \;\textnormal{ if } \beta \in [1,2],\\
    |(\nabla_n)^2u(\hat x)-(\nabla_n)^2u(\hat y)| \leq C\|u\|_{C^\beta} d(\hat x,\hat y)^{\beta-2}, \;\textnormal{ if } \beta \in [2,3].	
  \end{align*}
\end{proof}

\begin{proof}[Computation for Lemma \ref{lemma:Whitney Extension Is Almost Order Preserving}]
  \begin{align*}
    \nabla \tilde R(x) = 2C\|w\|_{C^{\beta_0}} \eta'\left ( \frac{|x-x_0|^{\beta_0}}{h_n} \right ) \beta_0 |x-x_0|^{\beta_0-1}\frac{(x-x_0)}{|x-x_0|}
  \end{align*}
  If $|x-x_0|^{\beta_0}\leq h_n$, then
  \begin{align*}
    \nabla \tilde R(x) = 2C\|w\|_{C^{\beta_0}}\beta_0 |x-x_0|^{\beta_0-1}\frac{(x-x_0)}{|x-x_0|}
  \end{align*}
  This expression is zero except when $|x-x_0|\leq h_n^{1/\beta_0}$, so
  \begin{align*}
    |\nabla \tilde R(x)| \leq 2C\|w\|_{C^{\beta_0}}\beta_0 h_n^{1-1/\beta_0}.\\
  \end{align*}  
  Furthermore, for $x,x'$ such that $|x-x_0|^{\beta_0}\leq h_n$, we have
  \begin{align*}
    |\nabla \tilde R(x)-\nabla \tilde R(x')| & \leq 2C\beta_0 \|w\|_{C^\beta_0} \left | |x-x_0|^{\beta_0-1}\frac{(x-x_0)}{|x-x_0|}-|x'-x_0|^{\beta_0-1}\frac{(x'-x_0)}{|x'-x_0|} \right | \\
	& \leq C\|w\|_{C^{\beta_0}}h_n^{\beta_0-\beta}|x-x'|^{\beta}.
  \end{align*}
  In conclusion,
  \begin{align*}
    \|\tilde R\|_{L^\infty}+ \|\nabla \tilde R\|_{L^\infty} + [\nabla \tilde R]_{C^{\beta-1}} \leq C\|w\|_{C^{\beta_0}}(h_n+h_n^{1-1/\beta_0}+h_n^{\beta_0-\beta}) \leq C\|w\|_{C^{\beta_0}}h_n^\gamma.
  \end{align*}

\end{proof}

%%%%%%%%%%%%%%%%%%%%%%%%%%%%%%%%%%%%%%%%%%%%%%
 
\bibliography{../refs}
\bibliographystyle{plain}
%%%%%%%%%%%%%%%%%%%%%%%%%%%%%%%%%%%%%%%%%%%%%%
\end{document}